\documentclass[11pt]{amsart}
\usepackage{amsmath,amssymb}
\usepackage{fullpage}
\usepackage[dvips]{graphicx}
\usepackage{psfrag}  
\def\COMMENT#1{}
\def\TASK#1{}
\def\noproof{{\unskip\nobreak\hfill\penalty50\hskip2em\hbox{}\nobreak\hfill%
        $\square$\parfillskip=0pt\finalhyphendemerits=0\par}\goodbreak}
\def\endproof{\noproof\bigskip}
\newdimen\margin   
\def\textno#1&#2\par{%
    \margin=\hsize
    \advance\margin by -4\parindent
           \setbox1=\hbox{\sl#1}%
    \ifdim\wd1 < \margin
       $$\box1\eqno#2$$%
    \else
       \bigbreak
       \hbox to \hsize{\indent$\vcenter{\advance\hsize by -3\parindent
       \sl\noindent#1}\hfil#2$}%
       \bigbreak
    \fi}
\def\proof{\removelastskip\penalty55\medskip\noindent{\bf Proof. }}

\def\eps{\varepsilon}

\def\T{\mathcal{T}}

\newcommand{\msc}[1]{\begin{center}MSC2000: #1.\end{center}}
\newtheorem{firstthm}{Proposition}[section]
\newtheorem{thm}[firstthm]{Theorem}
\newtheorem{prop}[firstthm]{Proposition}
\newtheorem{lemma}[firstthm]{Lemma}

\newtheorem{conj}[firstthm]{Conjecture}
\newtheorem{claim}[firstthm]{Claim}
\newtheorem{fact}[firstthm]{Fact}
\newtheorem{ques}[firstthm]{Question}

\begin{document}
\title{On directed versions of the Hajnal--Szemer\'edi theorem}
\author{Andrew Treglown}
\thanks{This research was in part undertaken whilst the author was a researcher at Charles University, Prague. The work leading to this invention has received funding from the European Research Council under the European Union's Seventh Framework Programme (FP7/2007-2013)/ERC grant agreement no.~259385.}

\label{firstpage}
\date{\today} 
\begin{abstract} 
We say that a (di)graph $G$ has a perfect $H$-packing if there exists a set of vertex-disjoint 
copies of $H$ which cover all the vertices in $G$. The seminal Hajnal--Szemer\'edi theorem characterises the minimum degree
that ensures a graph $G$ contains a perfect $K_r$-packing. In this paper we prove the following analogue for
directed graphs: 
Suppose that $T$ is a tournament on $r$ vertices  and $G$ is a digraph of sufficiently large order $n$ where $r$ divides $n$.
If $G$ has minimum in- and outdegree at least $ (1-1/r)n$ then $G$ contains a perfect $T$-packing. 

In the case when $T$ is a cyclic triangle, this result verifies a recent conjecture of Czygrinow, Kierstead and Molla~\cite{ckm}
(for large digraphs).
Furthermore, in the case when $T$ is transitive we conjecture that it suffices for every vertex in $G$ to have
sufficiently large indegree or outdegree. We prove this conjecture for transitive triangles and asymptotically for all $r \geq 3$.
Our approach  makes use of a result of Keevash and Mycroft~\cite{my1} concerning almost perfect matchings in hypergraphs as well as the Directed Graph Removal lemma \cite{alon, fox}.
\end{abstract}
\maketitle
\msc{5C35, 5C20, 5C70}
\section{Introduction} 
\subsection{Perfect packings in undirected graphs}
Given two (di)graphs $H$  and $G$, an \emph{$H$-packing} in $G$ 
is a collection of vertex-disjoint copies of $H$ in $G$. An
$H$-packing is called \emph{perfect} if it covers all the vertices of $G$.
Perfect $H$-packings are also referred to as \emph{$H$-factors} or \emph{perfect $H$-tilings}. 
Note that perfect $H$-packings are generalisations of perfect matchings (which correspond to the case when $H$ is a single edge).
Tutte's theorem characterises all those graphs that contain a perfect matching. On the other hand,
Hell and Kirkpatrick~\cite{hell} showed that the decision problem of
whether a graph $G$ has a perfect $H$-packing is NP-complete precisely when $H$ has a
component consisting of at least $3$ vertices. Thus, for such graphs $H$, it is unlikely that there is a complete characterisation
of those graphs containing a perfect $H$-packing. It is natural therefore to ask for simple sufficient conditions which force a graph
to contain a perfect $H$-packing. 

A seminal result in the area is the following theorem of Hajnal and Szemer\'edi~\cite{hs}.
\begin{thm}[Hajnal and Szemer\'edi~\cite{hs}]\label{hs}
Every graph $G$ whose order $n$
is divisible by $r$ and whose minimum degree satisfies $\delta (G) \geq (1-1/r)n$ contains a perfect $K_r$-packing. 
\end{thm}
It is easy to see that the minimum degree condition here cannot be lowered. 
In recent years there have been several generalisations of the Hajnal--Szemer\'edi theorem. K\"uhn and Osthus~\cite{kuhn, kuhn2} characterised, up to an additive constant, the minimum degree which ensures that a graph $G$ 
contains a perfect $H$-packing for an \emph{arbitrary} graph $H$.
Kierstead and Kostochka~\cite{kier} proved an \emph{Ore-type} analogue of
the Hajnal--Szemer\'edi theorem: If $G$ is a graph whose order $n$ is divisible by $r$, then $G$ contains a perfect $K_r$-packing
provided that $d(x)+d(y) \geq 2(1-1/r)n-1$ for all non-adjacent $x \not = y \in V(G)$.
K\"uhn, Osthus and Treglown~\cite{kotore}  characterised, asymptotically, the Ore-type degree condition which ensures that a graph $G$ 
contains a perfect $H$-packing for an arbitrary graph $H$.
Recently, Keevash and Mycroft~\cite{my} proved the following $r$-partite version of the Hajnal--Szemer\'edi theorem, thereby 
tackling a conjecture of Fischer~\cite{fish} for sufficiently large graphs.

\begin{thm}[Keevash and Mycroft~\cite{my}]\label{rpart}
Given $r \in \mathbb N$ there exists an $n_0 \in \mathbb N$ such that the following holds. Suppose $G$ is an $r$-partite graph
with vertex classes $V_1, \dots, V_r$ where $|V_i|=n\geq n_0$ for all $1\le i \leq r$.
If $$\overline{\delta}(G) \geq (1-1/r)n+1$$ then $G$ contains a perfect $K_r$-packing.
\end{thm}
(Here $\overline{\delta} (G)$ denotes the minimum degree of a vertex from one vertex class $V_i$ to another vertex class $V_j$.)
Keevash and Mycroft~\cite{my} actually proved a stronger result than Theorem~\ref{rpart}. Indeed, they showed that the minimum
degree condition here can be relaxed to $\overline{\delta}(G) \geq (1-1/r)n$ provided that $G$ is not isomorphic to one special construction. Further, their result extends to perfect $K_k$-packings where 
$1\leq k \leq r$ (see~\cite{my} for more details).

\subsection{Packing tournaments in directed graphs}\label{sec1}
It is natural to seek analogues of the Hajnal--Szemer\'edi theorem in the digraph and oriented graph settings.
We consider digraphs with no loops and at most one edge in each direction between every pair of vertices.
An oriented graph is a digraph without 
$2$-cycles. In this paper we restrict our attention to the problem for digraphs. See~\cite{problem, blm} for an overview of the known results
concerning perfect packings in oriented graphs. 

The \emph{minimum semidegree} $\delta ^0 (G)$ of a digraph $G$ is the minimum of its minimum outdegree $\delta ^+ (G)$ and its
minimum indegree $\delta ^- (G)$. 
Let $\delta (G)$ denote the minimum degree of $G$, that is, the minimum number of edges incident to a vertex in $G$. (Note that if both $xy$ and $yx$ are directed edges in $G$, they are counted as two separate edges.)
Denote by $\mathcal T_r$ the set of all tournaments on $r$ vertices. 
Our main result is
 an analogue of the Hajnal--Szemer\'edi theorem for perfect tournament packings in digraphs.
\begin{thm}\label{mainthm} Given an integer $r \geq 3$, there exists an $n_0 \in \mathbb N$ such that the following holds.
Suppose $T \in \mathcal T_r$ and $G$ is a digraph on $n \geq n_0$ vertices where $r$ divides $n$.
If $$\delta ^0 (G)\geq (1-1/r)n$$ then $G$ contains a perfect $T$-packing.
\end{thm}
Notice that the minimum semidegree condition in Theorem~\ref{mainthm} is tight. Indeed, let $n,r \in \mathbb N$ such that
$r$ divides $n$. Let $G'$ be the digraph obtained from the complete digraph on $n$ vertices by removing all those edges
lying in a given vertex set of size $n/r+1$. Then $\delta ^0 (G') =(1-1/r)n-1$ and $G'$ does not contain a perfect $T$-packing for any
$T \in \mathcal T_r$. In general, any digraph $G''$ on $n$ vertices with an independent set of size $n/r+1$ (and $\delta ^0 (G'') =(1-1/r)n-1$)
does not contain a perfect $T$-packing.

In the case when $T$ is the cyclic triangle $C_3$, Theorem~\ref{mainthm} verifies a recent conjecture of
Czygrinow, Kierstead and Molla~\cite{ckm} for large digraphs. 
Further, notice that Theorem~\ref{mainthm} is a `true generalisation' of the Hajnal--Szemer\'edi theorem in the sense that the former implies the latter for large graphs.

We remark that, by applying the same probabilistic trick used by Keevash and Sudakov in Section 7 of~\cite{keevs}, one can obtain an asymptotic version of Theorem~\ref{mainthm} from Theorem~\ref{rpart}. In~\cite{ckm} it was also shown that such an asymptotic version of
Theorem~\ref{mainthm} for $T=C_3$ follows from a result concerning perfect packings in multigraphs.

Similarly to many proofs in the area, our argument splits into `extremal' and `non-extremal' cases. When
$T\not = C_3$, the extremal case considers digraphs $G$ containing a set of vertices of size $n/r$ that spans an `almost' independent set (i.e. $G$ is `close'
to an extremal graph $G''$ as above). Interestingly,  when $T=C_3$ we have an extra extremal configuration (see Section~\ref{subex}),
and thus have two separate extremal cases.
In the non-extremal case our proof splits into two main tasks: finding an `almost' perfect $T$-packing in $G$ and finding a so-called
`absorbing set' that can be used to cover the remaining vertices with disjoint copies of $T$ (see Section~\ref{subabs1} for the
precise definition of such a set). To obtain the former we apply a result of Keevash and Mycroft~\cite{my1} concerning almost perfect matchings in hypergraphs. We also make use of the Directed Graph Removal lemma (see e.g. \cite{alon, fox}). A substantial proportion of the paper is devoted to obtaining our desired absorbing set. A more detailed overview of the proof is given in Section~\ref{overview}.

\subsection{Degree conditions forcing perfect transitive tournament packings}
Although the minimum semidegree condition in Theorem~\ref{mainthm} is `best-possible' one could replace the condition by a weaker one.
 Indeed,  for \emph{transitive} tournaments we conjecture that  the following stronger statement is true. Let $T_r$ denote the transitive tournament on $r$ vertices.

\begin{conj}\label{conj2} Let $n,r \in \mathbb N$ such that $r$ divides $n$.
 Suppose that $G$ is a digraph on $n $ vertices so that for any $x \in V(G)$,
\begin{align}\label{conmin1}
 \ d ^+ (x) \geq  ( 1- {1}/{r} ) n
\ \text{or} \ d ^- (x) \geq ( 1- {1}/{r}  ) n.
\end{align}
Then $G$ contains a perfect $T_r$-packing.
\end{conj}

 Conjecture~\ref{conj2} would imply the following very recent result of
 Czygrinow, DeBiasio,  Kierstead and  Molla~\cite{cdkm}.
\begin{thm}[Czygrinow, DeBiasio,  Kierstead and  Molla~\cite{cdkm}]\label{cdthm}
Let $n,r \in \mathbb N$ such that $r$ divides $n$. Then every digraph $G$ on $n$
vertices with $$\delta ^+ (G) \geq (1-1/r)n$$  contains a perfect $T_r$-packing.
\end{thm}
In Section~\ref{easy} we give a short proof of Conjecture~\ref{conj2} in the case when $r=3$. We also prove 
the following asymptotic version of Conjecture~\ref{conj2}.
\begin{thm}\label{2nonex}
Let $\eta >0$ and $r \geq 3$. Then there exists an $n_0 \in \mathbb N$ such that the following holds. Suppose that $G$ is a digraph on $n\geq n_0$ vertices 
where $r$ divides $n$ and that for any $x \in V(G)$,
\begin{align*}
d^+ (x) \geq \left(1-{1}/{r} +\eta \right) n \ \text{ or } \ d^- (x) \geq \left(1-{1}/{r} +\eta \right) n.
\end{align*}
Then $G$ contains a perfect $T_r$-packing.
\end{thm}
We give a unified approach to proving Theorems~\ref{mainthm} and~\ref{2nonex}, though the proof of the former
is substantially more involved. 

For `most' tournaments $T$, there does not  exist a `non-trivial' minimum outdegree condition which forces a digraph
$G$ to contain a perfect $T$-packing. Indeed, let $T \in \mathcal T_r$ such that every vertex in $T$ has an inneighbour. Let $n \in \mathbb N$ such that $r$ divides $n$. Obtain the digraph $G$ from the complete digraph on $n-1$ vertices by adding a vertex $x$ that
sends out all possible edges to the other vertices (but receives none). Then $\delta ^+ (G)=n-2$ but $G$ does not contain a 
perfect $T$-packing since $x$ does not lie in a copy of $T$.

So certainly Conjecture~\ref{conj2} and Theorem~\ref{cdthm} cannot be generalised to arbitrary tournaments $T$.
It would be interesting to establish whether the degree conditions in Conjecture~\ref{conj2} and Theorem~\ref{cdthm}  force a perfect $T$-packing
for some \emph{non-transitive} tournament $T$ on $r$ vertices.

\COMMENT{e.g. two disjoint complete digraphs joined together by balanced bipartite tournament has no perfect $k_3 ^-$-packing if
choose class sizes to be different sizes}

In the next section we give an outline of the proofs as well as  details about the organisation of the paper.

\section{Overview of the proofs and  organisation of the paper}\label{overview}
\subsection{Outline of the proof of Theorem~\ref{mainthm}}
Suppose that $G$ and $T \in \mathcal T_r$ are as in Theorem~\ref{mainthm}. Further, suppose that there is a `small' set $M \subseteq V(G)$ with the property that both $G[M]$ and $G[M\cup Q]$ contain
perfect $T$-packings for \emph{any} `very small' set $Q \subseteq V(G)$ where $|Q|\in r \mathbb N$. Then notice that, to find a perfect $T$-packing in $G$, it suffices to find an `almost' perfect $T$-packing
in $G':=G\setminus M$. Indeed, suppose that $G'$ contains a $T$-packing $\mathcal M_1$ covering all but a very small set of vertices $Q$. Then by definition of $M$, $G[M\cup Q]$  contains a perfect $T$-packing $\mathcal M_2$. Thus, $\mathcal M_1 \cup \mathcal M_2$ is a perfect $T$-packing in $G$, as desired. 

Roughly speaking, we refer to such a set $M$ as an `absorbing set' (see Section~\ref{subabs1} for the precise definition of such a set). The `absorbing method' was first used in~\cite{rrs2} and has subsequently
been  applied to numerous embedding problems in extremal graph theory.

In general, a digraph $G$ as in Theorem~\ref{mainthm} may not contain an absorbing set. For example, consider the complete $3$-partite digraph $G_1$ with vertex classes $V_1,V_2,V_3$ of size $n/3$. (So $G_1$ contains all possible edges with endpoints in different vertex classes.) Then $G_1$ satisfies the hypothesis of Theorem~\ref{mainthm} in the case when $T=T_3$ and $r=3$. However, if $Q \subseteq V_1$ such that $|Q|=3$
then it is easy to see that there does not exist a set $M \subseteq V(G_1)$ such that both $G_1[M]$ and $G_1[M \cup Q]$ contain perfect $T_3$-packings.
Notice though that $G_1$ is `close to extremal' (i.e. $G_1$ contains an independent set of size $n/3$). 

It turns out that being `close' to an extremal example is the only barrier preventing our digraph $G$ from containing an absorbing set $M$. Indeed, in the case when $T \not =C_3$ we show that if $G$ does not
contain an `almost' independent set of size $n/r$ then $G$ contains our desired set $M$. As mentioned in Section~\ref{sec1}, when $T=C_3$ we have an extra extremal configuration $Ex(n)$ (see Section~\ref{subex}). In this case we show that if $G$ is far from $Ex(n)$ and does not contain an `almost' independent set of size $n/3$ then $G$ contains our desired set $M$ (see Theorem~\ref{nonex1}).

\noindent{\bf{Constructing the absorbing set in the non-extremal case.}}
The crucial idea in proving that a non-extremal digraph $G$ contains an absorbing set $M$ is to first show that $G$ has many `connecting structures' of a certain type. For example, to find our desired  absorbing set it suffices to show that, for any $x,y \in V(G)$, there are `many' $(r-1)$-sets 
$X \subseteq V(G)$ so that both $X \cup \{x\}$ and $X\cup \{y\}$ span copies of $T$ in $G$.
In Section~\ref{consec} we prove a number of so-called connection lemmas that guarantee such connecting structures. This turns out to be quite a subtle process as we prove different connection lemmas
depending on the structure and size of $T$. In particular, we need to deal with the case when $T=C_3$ separately. (This stems from the fact
that we now have two extremal cases. See Section~\ref{consec} for more details.) In Section~\ref{secnon} we use the connection lemmas to construct the
absorbing set $M$.

\noindent{\bf{Covering the remaining vertices of $G$ in the non-extremal case.}}
As mentioned earlier, once we have constructed an absorbing set $M$ in a non-extremal digraph $G$, it suffices to find an `almost' perfect $T$-packing in $G'=G \setminus M$. For this, we translate the problem into  one about almost perfect matchings in hypergraphs. Indeed, from $G'$ we construct a hypergraph $J$ on $V(G')$ where, for any $1 \leq i \leq r$, an $i$-tuple $Y \subseteq V(G')$
forms an edge in $J$ precisely when $Y$ spans a subtournament of $T$ of size $i$ in $G'$. So one may think of $J$ as consisting of `layers' $J_1, \dots , J_r$ where $J_i$ contains the edges of size $i$.
For example, if $T=T_3$, then the edge set of $J_1$ is $V(G')$, the edge set of $J_2$ consists of all pairs $\{x,y\}$ where $xy \in E(G')$ or $yx \in E(G')$
and the edge set of $J_3$ consists of all triples $\{x,y,z\}$ that span a copy of $T_3$ in $G'$. 
$J$ is an example of a so-called $r$-complex (see Section~\ref{sec:almost} for the precise definition).

Vitally, $J$ has the property that a matching in $J_r$ corresponds to a $T$-packing in $G'$. We thus apply a result of Keevash and Mycroft~\cite{my1} on almost perfect matchings in $r$-complexes. (In order to apply this result we again use that $G$ is non-extremal.)
This ensures an almost perfect matching in $J_r$ and thus an almost perfect $T$-packing in $G '$, as desired.

\noindent{\bf{The extremal cases.}}
Finally, we deal with the case when $G$ is close to an extremal example. If $T=C_3$ and $G$ is close to $Ex(n)$ then a relatively short argument shows that $G$ must contain a perfect $C_3$-packing (see Lemma~\ref{exC3}). On the other hand, the general extremal case when $G$ contains an almost independent set of size $n/r$ is more involved (see Lemma~\ref{ex1}). (Note that the class of digraphs $G$ on $n$ vertices
with an almost independent set of size $n/r$ and $\delta ^0 (G) \geq (1-1/r)n$ is wide.)
We draw on ideas from~\cite{kss} to tackle this case.

The extremal cases are the only parts of the proof where we use the full force of the minimum semidegree condition on $G$. Indeed, the argument in the non-extremal case holds even if we relax the condition
to $\delta ^0 (G) \geq (1-1/r-o(1))n$.
\subsection{The proof of Theorem~\ref{2nonex}}
The proof of Theorem~\ref{2nonex} follows the same general approach as that of Theorem~\ref{mainthm} in the \emph{non-extremal} case: Again our two main tasks are to (i) find an absorbing set and (ii) cover  almost all of the remaining vertices with a $T_r$-packing. Thus, where possible, we present the tools for both proofs in a unified way. Indeed, many of our auxiliary results are applied in both proofs.
\subsection{Organisation of the paper}
In the next section we formally introduce the notion of an absorbing set and define the extremal digraph $Ex(n)$. We also introduce other notation and definitions. We prove Conjecture~\ref{conj2} in the case of
transitive triangles in Section~\ref{easy}. In Section~\ref{pfsec} we state the main auxiliary results that we prove in the paper and derive Theorems~\ref{mainthm} and~\ref{2nonex} from them.
In Section~\ref{tusec} we prove Tur\'an-type results for digraphs. These results will be applied both when constructing our absorbing sets and when finding an almost perfect $T$-packing in the non-extremal case.
Section~\ref{sec:almost} deals with this latter task. We state and prove the connection lemmas for  Theorems~\ref{mainthm} and~\ref{2nonex} in Section~\ref{consec}. These are then used to construct our absorbing sets in Section~\ref{secnon}. After giving a number of useful results in Section~\ref{secex} we tackle the extremal cases of Theorem~\ref{mainthm} in Sections~\ref{secex1} and~\ref{secexC3}.

\section{Notation and preliminaries}\label{notation}
\subsection{Definitions and notation}
Given two vertices~$x$ and~$y$ of a digraph~$G$, we write~$xy$ for the edge directed from~$x$ to~$y$.
We write $V(G)$ for the vertex set of $G$, $E(G)$ for the edge set of $G$ and define
$|G|:=|V(G)|$ and $e(G):= |E(G)|$. 
We denote by~$N^+ _G (x)$ and~$N^- _G (x)$ the out- and the inneighbourhood of~$x $
and by $d^+_G(x)$ and $d^-_G(x)$ its out- and indegree. 
We will write $N^+ (x)$ for example, if this is unambiguous. 
For a vertex $x \in V(G)$ and a set $Y \subseteq V(G)$ we write $d^+ _G (x,Y)$ to denote the number of edges in $G$ with startpoint $x$ and
endpoint in $Y$. We define $d^- _G (x,Y)$ analogously.
The \emph{minimum semidegree} $\delta ^0 (G)$ of $G$ is the minimum of its minimum outdegree $\delta ^+ (G)$ and its
minimum indegree $\delta ^- (G)$. 
Let $\delta (G)$ denote the minimum degree of $G$, that is, the minimum number of edges incident to a vertex in $G$. (Note that if both $xy$ and $yx$ are directed edges in $G$, they are counted as two separate edges.)

Given a subset $X \subseteq V(G)$, we write $G[X]$ for the subdigraph of $G$ induced by $X$. 
We write $G\setminus X$ for the subdigraph of $G$ induced by $V(G)\setminus X$.
For $x_1, \dots , x_m \in V(G)$
we define $G[x_1, \dots , x_m]:=G[\{x_1, \dots , x_m\}]$. 

Given a set $X \subseteq V(G)$ and a digraph $H$ on $|X|$ vertices we say that $X$ \emph{spans a copy of $H$ 
in $G$} if $G[X]$ contains a copy of $H$. In particular, this does not necessarily mean that
$X$ induces a copy of $H$ in $G$. For disjoint $X,Y \subseteq V(G)$ we let $G[X,Y]$ denote the digraph with vertex set $X\cup Y$ whose edge set consists of all those edges
$xy \in E(G)$ with $x \in X$ and $y \in Y$.
If $G$ and $H$ are digraphs, we write $G \cup H$ for the digraph whose vertex set is
$V(G) \cup V(H)$ and whose edge set is $E(G) \cup E(H)$. If $G$ and $H$ have the same vertex set $V$ then let $G-H$ denote the digraph with vertex set $V$ and edge set $E(G) \setminus E(H)$.

Given digraphs $G$ and $H$, we say that $G$ is \emph{$H$-free} if $G$ does not contain $H$ as a subdigraph.
Let $G$ be a (di)graph on $n$ vertices and let $\gamma >0$. We say that a set $S \subseteq V(G)$ is \emph{$\gamma$-independent} if 
$G[S]$ contains at most $\gamma n^2$ edges. 
Given two digraphs $G$ and $H$ on $n$ vertices we say that $G$ \emph{$\gamma$-contains $H$} if, after adding at most $\gamma n^2$ edges
to $G$, the resulting digraph contains a copy of $H$. More precisely, $G$ {$\gamma$-contains $H$} if there is an isomorphic copy
$G'$ of $G$ such that $V(G')=V(H)$ and $|E(H)\setminus E(G')|\leq \gamma n^2.$

For a (di)graph $G$ and disjoint $A,B \subseteq V(G)$, we write $e_G (A,B)$ for the number of edges in $G$ with one endpoint in $A$ and the other in $B$. (So $e_G (A,B)=e_G (B,A)$.)
Given a (di)graph $T$, 
let $2T$ denote the disjoint union of two copies of $T$.

Recall that $T_r$ denotes the transitive tournament of $r$ vertices.
Given $1 \leq i \leq r$, we say a vertex $x \in V(T_r)$ is the \emph{$i$th vertex of $T_r$}
if $x$ has indegree $i-1$ and outdegree $r-i$ in $T_r$.
Given a set $X$  and $r \in \mathbb N$ we denote by $\binom{X}{r}$ the set of all $r$-subsets of $X$.

Throughout the paper, we write $0<\alpha \ll \beta \ll \gamma$ to mean that we can choose the constants
$\alpha, \beta, \gamma$ from right to left. More
precisely, there are increasing functions $f$ and $g$ such that, given
$\gamma$, whenever we choose some $\beta \leq f(\gamma)$ and $\alpha \leq g(\beta)$, all
calculations needed in our proof are valid. 
Hierarchies of other lengths are defined in the obvious way.

\subsection{The extremal digraph $Ex(n)$}\label{subex}
Suppose that $n \geq 3 $ and  $c$  are non-negative integers. Define $a_1,a_2, a_3 \in \mathbb N$ such that 
$\lfloor n/3 \rfloor \leq a_1\leq a_2 \leq a_3 \leq \lceil n/3 \rceil$ where $a_1+a_2+a_3=n$.
Let $A_1,A_2$ and $A_3$ be disjoint vertex sets of size $a_1-c, a_2 +c$ and $a_3$ respectively.
Let $Ex_c (n)$ denote the digraph with vertex set $A_1 \cup A_2 \cup A_3$ and whose edge set is defined as follows:
\begin{itemize}
\item $A_i$ induces a complete digraph in $Ex_c(n)$ (for all $1\leq i \leq 3$);
\item If $x \in A_i$ and $y \in A_{i+1}$ then $xy \in E(Ex_c(n))$ (for all $1\leq i \leq 3$, where indices are taken mod $3$).
\end{itemize}
Define $Ex(n):=Ex_0 (n)$. (See Figure~1.) We call $A_1,A_2$ and $A_3$ the \emph{vertex classes of $Ex(n)$}.
\begin{figure}[htb!]
\begin{center}\footnotesize
\psfrag{1}[][]{ $A_1$}
\psfrag{2}[][]{$A_2$}
\psfrag{3}[][]{ $A_3$}
\includegraphics[width=0.35\columnwidth]{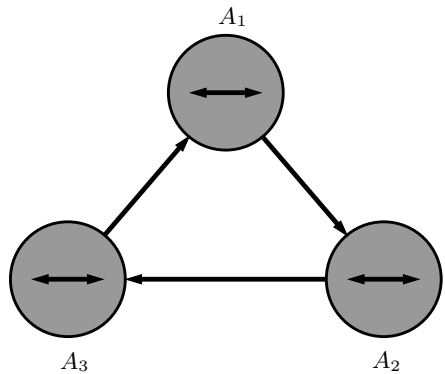}  
\caption{The extremal digraph $Ex(n)$}
\end{center}
\end{figure}

Suppose that $n$ is divisible by $3$. Note that $\delta ^0 (Ex_1 (n))=2n/3-2$ but $Ex_1(n)$ does not contain a perfect $C_3$-packing.
Thus, in the proof of Theorem~\ref{mainthm} for $T=C_3$ we have two extremal cases to consider: when $G$ contains an `almost' independent
set of size $n/3$ and when $G$ `almost' contains $Ex(n)$.

\subsection{Absorbing sets}\label{subabs1}
Let $T \in \mathcal T_r$. Given a digraph $G$, a set $S \subseteq V(G)$ is called a \emph{$T$-absorbing set for $Q \subseteq V(G)$}, if both
$G[S]$ and $G[S\cup Q]$ contain perfect $T$-packings. In this case we say that \emph{$Q$ is $T$-absorbed by $S$}. Sometimes we will simply refer
to a set $S \subseteq V(G)$ as a $T$-absorbing set if \emph{there exists} a set $Q \subseteq V(G)$ that is $T$-absorbed by $S$.

When constructing our absorbing sets in Section~\ref{secnon}
we will use the following Chernoff bound for binomial
distributions (see e.g.~\cite[Corollary 2.3]{Janson&Luczak&Rucinski00}).
Recall that the binomial random variable with parameters $(n,p)$ is the sum
of $n$ independent Bernoulli variables, each taking value $1$ with probability $p$
or $0$ with probability $1-p$.

\begin{prop}\label{chernoff}
Suppose that $X$ has binomial distribution and $0<a<3/2$. Then
$\mathbb{P}(|X - \mathbb{E}X| \ge a\mathbb{E}X) \le 2 e^{-\frac{a^2}{3}\mathbb{E}X}$.
\end{prop}

\section{Proof of Conjecture~\ref{conj2} for transitive triangles}\label{easy}
Let $\mathcal H$ be a collection of digraphs and $G$ a digraph. We say that $G$ contains a \emph{perfect $\mathcal H$-packing} if $G$ contains a collection of vertex-disjoint copies of elements
from $\mathcal H$ that together cover all the vertices of $G$. We now prove Conjecture~\ref{conj2} in the case of transitive triangles.
\begin{thm} 
Let $m \in \mathbb N$.
 Suppose that $G$ is a digraph on $n:=3m $ vertices so that for any $x \in V(G)$,
\begin{align}\label{conmin}
 \ d ^+ (x) \geq  2n/3
\ \text{or} \ d ^- (x) \geq 2n/3.
\end{align}
Then $G$ contains a perfect $T_3$-packing.
\end{thm}
\proof
Let $G$ be a digraph as in the statement of the theorem. Remove as many edges from $G$ as possible so that (\ref{conmin}) still holds. Let $G'$ denote the graph on $V(G)$ where $xy \in E(G')$
if and only if $xy \in E(G)$ or $yx \in E(G)$. So $\delta (G) \geq 2n/3$ by (\ref{conmin}). Thus, Theorem~\ref{hs} implies that $G'$ contains a perfect $K_3$-packing and so
$G$ contains a perfect $\{T_3, C_3\}$-packing. Let $\mathcal M$ denote the perfect $\{T_3,C_3\}$-packing in $G$ that contains the most copies of $T_3$. 

Suppose for a contradiction that $\mathcal M$ is not a perfect $T_3$-packing. Then there is a copy $C'_3$ of $C_3$ in $\mathcal M$. Let $V(C'_3)= \{x,y,z\}$ where $xy,yz,zx \in E(C'_3)$.
Suppose that $d^- _G (w) < 2n/3$ for some $w \in V(C'_3)$. Without loss of generality assume that $w=x$. Then (\ref{conmin}) implies that $d^+ _G (x) \geq 2n/3$.
If $d^+ _G (z) <2n/3$ then we may remove the edge $zx$ from $G$ and still (\ref{conmin}) holds, a contradiction to the minimality of $G$. So $d^+ _G (z) \geq 2n/3$. An identical argument implies that
$d^+ _G (y) \geq 2n/3$. This shows that $d^- _G (w) \geq 2n/3$ for all $w \in V(C'_3)$ or $d^+ _G (w) \geq 2n/3$ for all $w \in V(C'_3)$.

Without loss of generality assume that $d^+ _G (w) \geq 2n/3$ for all $w \in V(C'_3)$. (The other case is analogous.)
Note that $G[x,y,z]$ contains precisely three edges (else $V(C'_3)$ spans a copy of $T_3$, a contradiction to the maximality of $\mathcal M$).
In particular, there are at least $2n-3=6m-3>6(|\mathcal M|-1)$ edges in $G$ with startpoint in $V(C'_3)$ and endpoint in $V(G)\setminus V(C'_3)$.
This implies that there is an element $T \in \mathcal M \setminus \{C'_3\}$ that receives at least $7$ edges from $V(C'_3)$ in $G$.

So there is a vertex, say $x$, in $V(C'_3)$ such that $d^+ _G (x,V(T))=3$. Furthermore, $y$ and $z$ have a common outneighbour in $G$ that lies in $V(T)$.
Together this implies that $V(C'_3) \cup V(T)$ spans a copy of $2T_3$ in $G$.
This yields a perfect $\{T_3, C_3\}$-packing in $G$ containing more copies of $T_3$ than $\mathcal M$, a contradiction. So the assumption that $\mathcal M$ is not a perfect $T_3$-packing is false, as desired.
\endproof

\section{Deriving Theorems~\ref{mainthm} and~\ref{2nonex} from the auxiliary results}\label{pfsec}
In this section we state a number of auxiliary results that we will prove in the paper. We then combine these results to prove Theorems~\ref{mainthm} and~\ref{2nonex}.
Roughly speaking, the following result states that if $G$ is as in Theorem~\ref{mainthm} (namely has large semi-degree) and is \emph{non-extremal} then $G$ contains 
a `small' absorbing set that absorbs \emph{any} `very small' set of vertices in $G$.
\begin{thm}\label{nonex1}
Let $0 < 1/n \ll \eps \ll \xi \ll \gamma , \alpha \ll 1/r$ where $n,r \in \mathbb N$ and $r \geq 3$, and let $T \in \mathcal T_r$. Suppose that $G$ is a digraph on $n$ vertices so that
\begin{align*}
\delta ^0  (G) \geq \left ( 1-{1}/{r} -\eps \right ) n.
\end{align*}
Further suppose that	
\begin{itemize}
	\item $G$ does not contain any $\gamma$-independent set of size at least $n/r$;
	\item If $T=C_3$ then $G$ does not $\alpha$-contain $Ex(n)$.
	\end{itemize}
Then $V(G)$ contains a set $M$ so that $|M|\leq \xi n$ and $M$ is a $T$-absorbing set for any $W \subseteq V(G) \setminus M$ such that $|W| \in r \mathbb N$ and  $|W|\leq \xi ^2 n$.
\end{thm}
The next result is an analogue of Theorem~\ref{nonex1} that will be applied in the proof of Theorem~\ref{2nonex}.
\begin{thm}\label{nonex2}
Let $0 < 1/n \ll \eps \ll \xi \ll \gamma  \ll 1/r$ where $n,r \in \mathbb N$ and $r \geq 3$. Suppose that $G$ is a digraph on $n$ vertices so that, for any $x \in V(G)$,
\begin{align*}
d^+ (x) \geq \left(1-{1}/{r} -\eps \right) n \ \text{ or } \ d^- (x) \geq \left(1-{1}/{r} -\eps \right) n.
\end{align*}
Further suppose that	
$G$ does not contain any $\gamma$-independent set of size at least $n/r$.
Then $V(G)$ contains a set $M$ so that $|M|\leq \xi n$ and $M$ is a $T_r$-absorbing set for any $W \subseteq V(G) \setminus M$ such that $|W| \in r \mathbb N$ and  $|W|\leq \xi ^2 n$.
\end{thm}
We prove Theorems~\ref{nonex1} and~\ref{nonex2} in Section~\ref{secnon}.
The crucial tools used in these proofs are so-called `connection lemmas' which we introduce in Section~\ref{consec}.
\begin{thm}\label{almostthm}
Let 
$0<1/n \ll 1/\ell \ll \eps \ll \gamma \ll 1/r$ and $T \in \mathcal T_r$ for some $r \geq 3$.
Suppose that $G$ is a digraph on $n$ vertices such that
\begin{align}\label{minA}
\delta ^0 (G) \geq \left( 1 -{1}/{r}-\eps \right ) n.
\end{align} 
Then at least one of the following properties holds:
\begin{itemize}
\item[(i)] $G$ contains a $T$-packing that covers all but at most $\ell $ vertices;
\item[(ii)] $G$ contains a $\gamma$-independent set of size at least $n/r$.
\end{itemize}
\end{thm}
Theorems~\ref{nonex1} and~\ref{almostthm} together ensure that a non-extremal digraph $G$ in Theorem~\ref{mainthm} contains a perfect $T$-packing.
The following result is an analogue of Theorem~\ref{almostthm} that will be
applied in the proof of Theorem~\ref{2nonex}.
\begin{thm}\label{almostthm+-}
Let 
$0<1/n \ll 1/\ell \ll \eps \ll \gamma \ll 1/r$ where
$n,r \in \mathbb N$ and  $r \geq 3$.
Suppose that $G$ is a digraph on $n$ vertices such that, for any $x\in V(G)$,
\begin{align}\label{miny}
 \ d ^+ (x) \geq \left ( 1- {1}/{r} -\eps \right ) n
\ \text{or} \ d ^- (x) \geq \left ( 1- {1}/{r} -\eps \right ) n.
\end{align}
Further suppose that, given any $x,y \in V(G)$, if $d ^+ (x) <\left ( 1- {1}/{r} -\eps \right ) n$ and
$d ^- (y)< \left ( 1- {1}/{r} -\eps \right ) n$
then $xy \not \in E(G)$.
Then at least one of the following properties holds:
\begin{itemize}
\item[(i)] $G$ contains a $T_r$-packing that covers all but at most $\ell $ vertices;
\item[(ii)] $G$ contains a $\gamma$-independent set of size at least $n/r$.
\end{itemize}
\end{thm}
In Section~\ref{sec:almost} we deduce Theorems~\ref{almostthm} and~\ref{almostthm+-} from a result of Keevash and Mycroft~\cite{my1} concerning almost perfect matchings in hypergraphs. 
The next two results cover the extremal cases of Theorem~\ref{mainthm}.
\begin{lemma}\label{ex1} Let $r \in \mathbb N$ such that $ r \geq 3$. There exist $\gamma >0$ and $n_0 \in \mathbb N$ such that the following holds. Suppose that $T \in \mathcal T_r$ and $G$
is a digraph on $n \geq n_0$ vertices where $n$ is divisible by $r$. If
\begin{align}\label{minex1}
\delta ^0  (G) \geq \left ( 1-{1}/{r} \right ) n
\end{align}
and $G$ contains a $\gamma$-independent set of size $n/r$ then $G$ contains a perfect $T$-packing.
\end{lemma}

\begin{lemma}\label{exC3} There exist   $\alpha >0$ and $n_0 \in \mathbb N$ such that the following holds.
Suppose that $G$ is a digraph on $n \geq n_0$ vertices where $n$ is divisible by $3$. If
\begin{itemize}
\item $\delta ^0 (G) \geq 2n/3-1$ and
\item $G$ $\alpha$-contains $Ex(n)$,
\end{itemize}
then $G$ contains  a perfect $C_3$-packing.
\end{lemma}
Lemmas~\ref{ex1} and \ref{exC3} are proved in Sections~\ref{secex1} and~\ref{secexC3} respectively.
We now deduce Theorem~\ref{mainthm} from Theorems~\ref{nonex1} and~\ref{almostthm} and Lemmas~\ref{ex1} and~\ref{exC3}.

\noindent
{\bf Proof of Theorem~\ref{mainthm}.}
Define constants $\eps, \xi , \gamma , \alpha$ and integers $n_0 , \ell$ such that
$$0<1/n_0 \ll 1/\ell \ll \eps \ll \xi \ll \gamma , \alpha \ll 1/r.$$
Let $T \in \mathcal T_r$ and suppose that $G$ is a digraph on $n \geq n_0$ vertices such that $r$ divides $n$ and $\delta ^0 (G) \geq (1-1/r)n$. By Lemmas~\ref{ex1} and~\ref{exC3}
we may assume that
\begin{itemize}
	\item[(i)] $G$ does not contain any $\gamma$-independent set of size  $n/r$;
	\item[(ii)] If $T=C_3$ then $G$ does not $\alpha$-contain $Ex(n)$.
	\end{itemize}
(Otherwise $G$ contains a perfect $T$-packing, as desired.) 
Thus, we can apply Theorem~\ref{nonex1} to obtain a 
set $M\subseteq V(G)$ so that $|M|\leq \xi n$ and $M$ is a $T$-absorbing set for any $W \subseteq V(G) \setminus M$ such that $|W| \in r \mathbb N$ and  $|W|\leq \xi ^2 n$.
Set $G':=G\setminus M$ and let $n':=|G'|\geq (1-\xi)n$. Since $n$ is divisible by $r$ and $M$ is a $T$-absorbing set, $n'$ is also divisible by $r$. Further,
$$\delta ^0 (G') \geq (1-1/r)n-\xi n\geq (1-1/r-\xi )n'.$$
Notice that $G'$ does not contain any $\gamma/2$-independent set of size at least  $n'/r$. (Otherwise $G$ contains a $\gamma$-independent set of size  $n/r$, a contradiction to (i).)
Therefore, by applying Theorem~\ref{almostthm} with $G', n',\xi, \gamma/2$ playing the roles of $G,n, \eps, \gamma$, we obtain a $T$-packing $\mathcal M_1$ in $G'$ that covers all but at most $\ell $ vertices. Let $W$ denote the set of vertices in $G'$ that are not covered by $\mathcal M_1$. 
So $|W|\leq \ell \leq \xi ^2 n$ and,
since $n'$ is divisible by $r$, $|W|\in r \mathbb N$. Thus, by definition of $M$, $G[M\cup W]$ 
contains a perfect $T$-packing $\mathcal M_2$. Therefore, $\mathcal M_1 \cup \mathcal M_2$ is a perfect $T$-packing in $G$, as desired.
\endproof

Similarly we deduce Theorem~\ref{2nonex} from Theorems~\ref{nonex2} and~\ref{almostthm+-}.

\noindent
{\bf Proof of Theorem~\ref{2nonex}.}
Define  additional constants $ \eps, \xi , \gamma  $ and  integers $n_0,  \ell$ such that
$$0<1/n_0 \ll 1/\ell \ll \eps \ll \xi \ll \gamma  \ll 1/r, \eta .$$
Suppose that $G$ is a digraph on $n \geq n _0$ vertices where $r$ divides $n$ and:
\begin{itemize}
\item[(i)] For any $x\in V(G)$,  $d ^+ (x) \geq \left ( 1- 1/r +\eta \right ) n$
{or}  $d ^- (x) \geq \left ( 1- 1/r  +\eta \right ) n$.
\end{itemize}
Suppose that for some $x,y \in V(G)$,  $d ^+ (x) <( 1-1/r+\eta) n$, 
$d ^- (y)< ( 1-1/r +\eta ) n$ 
and  $xy  \in E(G)$. Then if we remove the edge $xy$ from $G$, (i) still holds. In particular, this implies that  we may assume:
\begin{itemize}
\item[(ii)] Given any  $x,y \in V(G)$,  if  $d ^+ (x) <\left ( 1- 1/r +\eta \right  ) n$  and 
$d ^- (y)< \left ( 1- 1/r  +\eta \right ) n $
then  $xy \not \in E(G)$.
\end{itemize}
Note that (i) implies that:
\begin{itemize}
\item[(iii)] $G$ does not contain any $\gamma$-independent set of size $n/r$.
\end{itemize}

Apply Theorem~\ref{nonex2} to obtain a 
set $M\subseteq V(G)$ so that $|M|\leq \xi n$ and $M$ is a $T_r$-absorbing set for any $W \subseteq V(G) \setminus M$ such that $|W| \in r \mathbb N$ and  $|W|\leq \xi ^2 n$.
Set $G':=G\setminus M$ and let $n':=|G'|\geq (1-\xi)n$. Since $n$ is divisible by $r$ and $M$ is a $T_r$-absorbing set, $n'$ is also divisible by $r$. Further, (i) implies that
for any $x\in V(G')$,  $$d ^+ _{G'} (x) \geq \left ( 1- \frac{1}{r} -\eps \right ) n' \ \text{ or } \
d ^- _{G'} (x) \geq \left ( 1- \frac{1}{r} -\eps \right ) n'.$$

Suppose that for some $x,y \in V(G')$,  $d ^+ _{G'} (x) <( 1-1/r - \eps) n'$ and 
$d ^- _{G'} (y)< ( 1-1/r -\eps) n'$. Then $d ^+ _{G} (x) <( 1-1/r - \eps) n'+\xi n  \leq(1-1/r+\eta)n$ and $d ^- _{G} (y)< ( 1-1/r +\eta)n$. Thus, by (ii), $xy \not \in E(G')$. 
Notice that $G'$ does not contain any $\gamma/2$-independent set of size at least  $n'/r$. (Otherwise $G$ contains a $\gamma$-independent set of size  $n/r$, a contradiction to (iii).)
Therefore, by applying Theorem~\ref{almostthm+-} with $G', n', \gamma/2$ playing the roles of $G,n, \gamma$, we obtain a $T_r$-packing $\mathcal M_1$ in $G'$ that covers all but at most $\ell $ vertices. Let $W$ denote the set of vertices in $G'$ that are not covered by $\mathcal M_1$. 
So $|W|\leq \ell \leq \xi ^2 n$ and,
since $n'$ is divisible by $r$, $|W|\in r \mathbb N$. Thus, by definition of $M$, $G[M\cup W]$ 
contains a perfect $T_r$-packing $\mathcal M_2$. Hence, $\mathcal M_1 \cup \mathcal M_2$ is a perfect $T_r$-packing in $G$, as desired.
\endproof

Suppose that $G$ is a digraph on $n$ vertices that satisfies (\ref{conmin1}). 
Suppose that for some $x,y \in V(G)$,  $d ^+ (x) <( 1-1/r) n$, 
$d ^- (y)< ( 1-1/r  ) n$ 
and  $xy  \in E(G)$. Then if we remove the edge $xy$ from $G$, (\ref{conmin1}) still holds. Thus, to prove Conjecture~\ref{conj2} it suffices to  consider digraphs $G$ with the following additional assumption:
Given any  $x,y \in V(G)$,  if  $d ^+ (x) <\left ( 1- 1/r \right  ) n$  and 
$d ^- (y)< \left ( 1- 1/r   \right ) n $
then  $xy \not \in E(G)$. The next result states that such a digraph $G$  contains a perfect $T_r$-packing or contains an `almost' independent set of size $n/r$.
\begin{thm}\label{2nonex1} Given any $\gamma >0$ and an integer $r\geq 3$ there exists an $n_0 \in \mathbb N$ such that the following holds.
 Suppose that $G$ is a digraph on $n\geq n_0$ vertices where $r$ divides $n$ and that, for any $x \in V(G)$,
\begin{align*}
d^+ (x) \geq \left(1-{1}/{r}  \right) n \ \text{ or } \ d^- (x) \geq \left(1-{1}/{r} \right) n.
\end{align*}
Further suppose that, given any $x,y \in V(G)$, if $d ^+ (x) <\left ( 1- {1}/{r}  \right ) n$ and
$d ^- (y)< \left ( 1- {1}/{r}  \right ) n$
then $xy \not \in E(G)$.
Then at least one of the following properties holds:
\begin{itemize}
\item[(i)] $G$ contains a perfect $T_r$-packing;
\item[(ii)] $G$ contains a $\gamma$-independent set of size at least $n/r$.
\end{itemize}
\end{thm}
\proof
The proof  is almost identical to that of Theorem~\ref{2nonex} so we omit it.
\endproof
So Theorem~\ref{2nonex1} implies that to prove Conjecture~\ref{conj2} for large digraphs it suffices to prove  the extremal case.

\section{Tur\'an-type stability results for embedding tournaments}\label{tusec}
\subsection{The Tur\'an result for Theorem~\ref{mainthm}}
The aim of this subsection is to prove Proposition~\ref{turanstab} which, roughly speaking,
states that a digraph $G$ on $n$ vertices of sufficiently large semidegree (i) contains many copies of a fixed
$T \in \mathcal T_r$ or (ii) contains an `almost' independent set of size $n/r$. Proposition~\ref{turanstab}
will be applied in the proof of both Theorem~\ref{nonex1} and Theorem~\ref{almostthm}. 

The next result is an immediate consequence of Tur\'an's theorem.

\begin{prop} \label{turan}
Let $n,r  \in \mathbb N$ where $r \geq 2$. Suppose that $G$ is a digraph on $n$ vertices such that
\begin{align}\label{edge1}
e(G) >\left(1-\frac{1}{r-1} \right) \frac{n^2}{2} +\binom{n}{2}.
\end{align}
Then $G$ contains a copy of $K_r$.
\end{prop}
\proof
Let $G'$ be the graph on $V(G)$ whose edge set consists of all pairs $xy$ where $xy,yx \in E(G)$.
Then (\ref{edge1}) implies that $e(G') >(1-{1}/(r-1) ) n^2/2$ and thus $G'$ contains a copy
of $K_r$ by Tur\'an's theorem. Hence, $K_r \subseteq G$ as required.
\endproof

\begin{prop}\label{ind}
Let $1/n \ll \alpha \ll 1/r$ with  $n,r \in \mathbb N$ and $r \geq 3$, and let $T \in \mathcal T_r$. Suppose that $G$ is a digraph on $n$ vertices such that
\begin{align}\label{min1}
\delta ^0 (G) \geq \left ( 1- \frac{1}{r-1} -\alpha \right ) n.
\end{align}
If $G$ is $T$-free then $G$ contains an independent set of size at least 
$\left(\frac{1}{r-1} -2r^2\alpha\right) n$.
\end{prop}
\proof 
Let $V(T) =\{ v_1,\dots , v_{r-2}, a,b \}$ and set $T':=T[v_1, \dots , v_{r-2}]$. Using (\ref{min1}), greedily construct
a copy $T''$ of $T'$ in $G$. To simplify notation,
for each $1\leq i \leq r-2$, we will refer to the vertex in $T''$ (and thus $G$) corresponding to the vertex 
$v_i$ in $T'$ as $v_i$.

We say that a vertex $v \in V(G)$ is a \emph{candidate for $a$ in $G$} if the following conditions hold:
\begin{itemize}
\item If $av_i \in E(T)$  then $vv_i \in E(G)$ (for each $1\leq i \leq r-2$);
\item If $v_ia \in E(T)$  then $v_iv \in E(G)$  (for each $1\leq i \leq r-2$).
\end{itemize}
We give an analogous definition of a \emph{candidate for $b$ in $G$}.
Let $A$ denote the set of candidates for $a$ in $G$ and let $B$ denote the set of candidates for
$b$ in $G$. Thus, (\ref{min1}) implies that 
\begin{align}\label{ABboundy}
|A|,|B|\geq \left(\frac{1}{r-1} -(r-2)\alpha\right) n.
\end{align}
Without loss of generality, suppose that $ab \in E(T)$. Since $G$ is $T$-free, there  is no edge in $G$
whose startpoint lies in $A$ and whose endpoint lies in $B$. In particular, $A\cap B$ is an independent
set. 

Set $A':=A\setminus B$.
Suppose for a contradiction that $|A'|\geq 2(r-1)^2 \alpha n$. 
Given any vertex $x \in A'$, since $x$ sends no edges to $B$, (\ref{min1}) and (\ref{ABboundy}) imply
that there are at most
$$\left( \frac{1}{r-1}+\alpha \right) n-\left( \frac{1}{r-1}-(r-2)\alpha \right) n=(r-1)\alpha n$$
vertices in $A'$ that $x$ does not send an edge to (including itself). Thus,
$$\delta ^+ (G[A'])\geq|A'|-(r-1)\alpha n \geq \left (1- \frac{1}{2(r-1)}\right)|A'|$$
and so
$$e(G[A'])\geq \left (1-\frac{1}{2(r-1)}\right)|A'|^2 >\left( 1 -\frac{1}{r-1} \right )\frac{|A'|^2}{2}
+\binom{|A'|}{2}.$$
Hence, Proposition~\ref{turan} implies that $K_r \subseteq G[A']$ and so $G$ contains a copy of
$T$, a contradiction. Therefore, $|A'|<2(r-1)^2 \alpha n$. Together with (\ref{ABboundy})
this implies that the independent set $A \cap B$ is of size at least
$\left(\frac{1}{r-1} -(r-2)\alpha\right) n -2(r-1)^2 \alpha n \geq 
\left(\frac{1}{r-1} -2r^2\alpha\right) n,$ as required.
\endproof
To prove Proposition~\ref{turanstab} we will apply Proposition~\ref{ind} together with the following
directed version of the Removal lemma (see e.g. \cite{alon, fox}).

\begin{lemma}[Directed Graph Removal lemma]\label{dirl}
Let $\gamma >0$ and $t \in \mathbb N$. Given any digraph $H$ on $t$ vertices, there exists 
$\alpha = \alpha (H, \gamma )>0$ and $n_0 =n_0 (H, \gamma ) \in \mathbb N$ such that the following
holds. Suppose that $G$ is a digraph on $n \geq n_0$ vertices such that $G$ contains at most $\alpha n^t$
copies of $H$. Then $G$ can be made $H$-free by deleting at most $\gamma n^2$ edges.
\end{lemma}

\begin{prop}\label{turanstab}
Let $0<1/n \ll \alpha \ll \eps \ll 1/r$ where $r,n \in \mathbb N$ and $r \geq 2$, and let $T \in \mathcal T_r$. 
Suppose that $G$ is a digraph on $n$ vertices such that
$$\delta ^0 (G) \geq \left ( 1- \frac{1}{r-1} -\eps\right ) n$$
and so that $G$ contains at most $\alpha n^r$ copies of $T$.
Then $G$ contains a $\sqrt{\eps}$-independent set of size at least $n/(r-1)$.
\end{prop}
\proof The case when $r=2$ is trivial thus we may assume that $r \geq 3$.
Define an additional constant $ \gamma$  so that
$ \alpha \ll \gamma \ll \eps .$
Suppose that $G$ is as in the statement of the proposition. Since $G$ contains at most 
$\alpha n^r$ copies of $T$, Lemma~\ref{dirl} implies that one can remove at most $\gamma n^2$ edges
from $G$ to obtain a spanning subdigraph $G'$ that is $T$-free. 
So at most $\sqrt{\gamma}n$ vertices in $G$ are incident to more than $2 \sqrt{\gamma} n$ of the edges in $G-G'$.
Therefore, since $\gamma \ll \eps$, there exists an induced
 subdigraph $G''$ of $G'$ such that $n'':=|G''|\geq (1-\eps)n$ and 
$\delta ^0 (G'')\geq (1-1/(r-1)-2\eps)n''$.

Since $G''$ is $T$-free, Proposition~\ref{ind} implies that $G''$ contains an independent set $S$ of 
size at least $$\left(\frac{1}{r-1}-4r^2 \eps \right) n''\geq \left(\frac{1}{r-1}-5r^2 \eps \right) n.$$
By construction of $G''$, $S$ is a $\gamma$-independent set in $G$. By adding at most
$5r^2 \eps n$ arbitrary vertices to $S$ we obtain a $\sqrt{\eps}$-independent set in $G$ of size at least $n/(r-1)$, as desired.
\endproof

\subsection{The Tur\'an result for Theorem~\ref{2nonex}}
In this section we give an analogue of Proposition~\ref{turanstab} which will be applied in the proof of both Theorem~\ref{nonex2} and
Theorem~\ref{almostthm+-}. The next result is an analogue of Proposition~\ref{ind}.
\begin{prop}\label{ind+-}
Let $1/n \ll \alpha \ll 1/r$ where $n,r \in \mathbb N$ and  $r \geq 3$. 
Suppose that $G$ is a digraph on $n$ vertices such that, for any $x\in V(G)$,
\begin{align}\label{min1+-}
 \ d ^+ (x) \geq \left ( 1- \frac{1}{r-1} -\alpha \right ) n
\ \text{or} \ d ^- (x) \geq \left ( 1- \frac{1}{r-1} -\alpha \right ) n.
\end{align}
If $G$ is $T_r$-free then $G$ contains an independent set of size at least 
$\left(\frac{1}{r-1} -r\alpha\right) n$.
\end{prop}
\proof
Let $r' \in \mathbb N$. Suppose that $T'$ is a copy of $T_{r'}$ in $G$.
Let $V(T')=\{x_1, \dots , x_{r'}\}$ where $x_i$ plays the role of the $i$th vertex
of $T_{r'}$. We say that $T'$ is \emph{consistent} if there exists $0\leq s' \leq r'$
such that
\begin{itemize}
\item $d ^+ (x_i) \geq  ( 1- {1}/{(r-1)} -\alpha  ) n$ for all $i \leq s'$;
\item $d ^- (x_i) \geq  ( 1- {1}/{(r-1)} -\alpha  ) n$ for all $i >s'$.
\end{itemize}
We call $s'$ a \emph{turning point of $T'$}. (Note that
$T'$ could have more than one turning point.)

(\ref{min1+-}) implies that every copy of $T_1$ in $G$ is consistent.
Suppose that, for some $1 \leq r' < r-2$, we have found a consistent copy $T'$ of $T_{r'}$ in $G$.
As before, let $V(T')=\{x_1, \dots , x_{r'}\}$ where $x_i$ plays the role of the $i$th vertex
of $T_{r'}$ and let $s'$ denote a turning point of $T'$. 
Set $N':=\bigcap _{i\leq s'} N^+(x_i) \cap \bigcap _{i>s'} N^- (x_i).$
Since $T'$ is consistent with turning point $s'$ and $r'<r-2$,
$$|N'|\geq \left (1 -\frac{r'}{r-1} -r' \alpha \right )n >0.$$
Consider any
$x \in N'.$
Then $V(T') \cup \{x\}$ spans a consistent copy of $T_{r'+1}$ in $G$ where $x$ plays the
role of the $(s'+1)$th vertex in $T_{r'+1}$. (This is true regardless of whether
$d ^+ (x) \geq  ( 1- {1}/{(r-1)} -\alpha  ) n$ or $d ^- (x) \geq  ( 1- {1}/{(r-1)} -\alpha  ) n$.)

This observation  implies that we can greedily construct a
consistent copy $T$ of $T_{r-2}$ in $G$. Let $V(T)=\{y_1, \dots , y_{r-2}\}$ where $y_i$ plays the role of the $i$th vertex of $T_{r-2}$ and let $s$ denote a turning point of $T$.
Set $N:=\bigcap _{i\leq s} N^+(y_i) \cap \bigcap _{i>s} N^- (y_i) $.
Since $T$ is consistent with turning point $s$, 
$$\left | N \right |
\geq \left( 1- \frac{r-2}{r-1} -(r-2)\alpha  \right)n\geq \left(\frac{1}{r-1} -r\alpha\right) n.$$

Suppose that there is an edge $xy \in E(G[N])$. Then $V(T) \cup\{x,y\}$ spans a copy of
$T_r$ in $G$ where $x$ and $y$ play the roles of the $(s+1)$th and $(s+2)$th vertices in $T_r$
respectively. This is a contradiction, so $N$ is an independent set in $G$, as required.
\endproof

\begin{prop}\label{turanstab+-}
Let $0<1/n \ll \alpha \ll \eps \ll 1/r$ where $n, r \in \mathbb N$ and  $r \geq 2$. 
Suppose that $G$ is a digraph on $n$ vertices such that, for any $x\in V(G)$,
\begin{align*}
 \ d ^+ (x) \geq \left ( 1- \frac{1}{r-1} -\eps \right ) n
\ \text{or} \ d ^- (x) \geq \left ( 1- \frac{1}{r-1} -\eps \right ) n.
\end{align*}
Further suppose that $G$ contains at most $\alpha n^r$ copies of $T_r$.
Then $G$ contains a $\sqrt{\eps}$-independent set of size at least $n/(r-1)$.
\end{prop}
\proof
We follow the same argument as in the proof of Proposition~\ref{turanstab} except that
we apply Proposition~\ref{ind+-} rather than Proposition~\ref{ind}.
\endproof

\section{$k$-complexes and almost perfect tournament packings}\label{sec:almost}
The key tool in the proofs of Theorems~\ref{almostthm} and~\ref{almostthm+-} is a result of Keevash and Mycroft~\cite[Theorem 2.3]{my1} concerning almost perfect matchings
in so-called $k$-complexes. 
To state this result we require some more definitions. Let $k \in \mathbb N$. A \emph{$k$-system} is 
a hypergraph $J$ in which every edge of $J$ contains at most $k$ vertices and $ \emptyset \in E(J)$. For $0\leq i \leq k$,
we refer to the edges of size $i$ in $J$ as the \emph{$i$-edges of $J$}, and write $J_i$ to denote the $i$-uniform hypergraph
on $V(J)$ induced by these edges. A \emph{$k$-complex $J$} is a $k$-system whose edge set is closed under inclusion. That is,
if $e \in E(H)$ and $e' \subseteq e$ then $e' \in E(H)$. 

Let $J$ be a $k$-complex. 
For any edge $e \in E(J)$, the \emph{degree $d(e)$ of $e$} is the number of $(|e|+1)$-edges $e'$ of $J$ that contain 
$e$ as a subset. The \emph{minimum $r$-degree $\delta _r (J)$} of $J$ is the minimum of $d(e)$ taken over all $r$-edges $e \in E(J)$.
The \emph{degree sequence of $J$} is defined as $\delta (J):= (\delta _0 (J) , \delta _1 (J), \dots , \delta _{k-1} (J) )$.
Given a vector $\underline{a}=(a_0, \dots , a_{k-1})$ of positive integers we write $\delta (J) \geq \underline{a}$ to mean
that $\delta _i (J) \geq a_i$ for all $0 \leq i \leq k-1$.

Suppose $V$ is a set of $n$ vertices, $1\leq j \leq k-1$ and $S \subseteq V$. Define $J(S,j)$ to be the $k$-complex on $V$
in which $J(S,j)_i$ (for $0 \leq i \leq k$) consists of all $i$-sets in $V$ that contain at most $j$ vertices of $S$. Let $\beta >0$.
Given $k$-uniform hypergraphs $H, K$ on the same vertex set of size $n$ we say that $K$ is \emph{$\beta$-contained in $H$} if, by adding at most
$\beta n^k$ edges to $H$, we can find a copy of $K$ in $H$. A \emph{matching} in a hypergraph $H$ is a collection of vertex-disjoint edges from $H$.

\begin{thm}[Keevash and Mycroft~\cite{my1}]\label{myalmost}
Suppose that $1/n \ll 1/{\ell} \ll \eps \ll \beta \ll 1/k$. Let $J$ be a $k$-complex on $n$ vertices such that
$$\delta (J) \geq \left ( n, \left (1-\frac{1}{k} -\eps \right ) n, \left (1-\frac{2}{k} -\eps \right ) n,
\dots , \left ( \frac{1}{k} -\eps \right ) n \right ).$$  
Then at least one of the following properties holds:
\begin{itemize}
\item[(i)] $J_k$ contains a matching that covers all but at most $\ell $ vertices;
\item[(ii)] $J_k$ is $\beta$-contained in $J(S,j)_k$ for some $1 \leq j \leq k-1$ and $S \subseteq V(J)$
with $|S|=\lfloor jn/k \rfloor$.
\end{itemize}
\end{thm}

In the following two subsections we apply Theorem~\ref{myalmost} to prove both Theorem~\ref{almostthm} and Theorem~\ref{almostthm+-}.

\subsection{Proof of Theorem~\ref{almostthm}}
Define additional constants $ \beta, \alpha, \alpha ',  \eps '$ such that
$$0<1/n \ll 1/\ell \ll \eps \ll \beta \ll \alpha \ll \alpha ' \ll \eps ' \ll \gamma \ll 1/r.$$
Let $G$ be a digraph as in the statement of the theorem.

Our first task is to construct an $r$-complex $J$ from $G$ so that we can apply Theorem~\ref{myalmost}.
Let $J$ be the $r$-system on $V(G)$ where, for each $0 \leq i \leq r$, $J_i$ is defined as follows:
\begin{itemize}
\item For each subtournament $T'$ of $T$ on $i$ vertices, any $i$-tuple 
in $V(G)$ that spans a copy of $T'$ in $G$ forms an $i$-edge in $J_i$.
\end{itemize}
So for example, if $T=C_3$, then $E(J_1)=V(J_1)$, $E(J_2)$ is the set of all pairs $\{x,y\}$ where $xy \in E(G)$ or $yx \in E(G)$
and $E(J_3)$ is the set of all triples $\{x,y,z\}$ that span a copy of $C_3$ in $G$. (Note though that $\{x,y,z\}$ does not
have to \emph{induce} a copy of $C_3$ in $G$. For example, we could have $G[x,y,z]=K_3$.)

By construction $J$ is an $r$-complex. Further, notice that a matching in the $r$-uniform hypergraph $J_r$ corresponds to a 
$T$-packing in $G$. Clearly $\delta _0 (J)=n$.
Set $1 \leq i \leq r-1$ and let $T'$ be a subtournament of $T$ on $i$ vertices. If $T''$ is a copy of $T'$ in $G$
then (\ref{minA}) implies that there are at least $(1-i/r-i\eps)n \geq (1-i/r-\eps r)n$
vertices $x$ in $G$ such that $V(T'') \cup \{x\}$ spans a copy of a subtournament of $T$ on $i+1$ vertices. 
This therefore implies that,
$$\delta (J) \geq \left( n, \left(1 -\frac{1}{r} -\eps r\right) n, \left(1 -\frac{2}{r} -\eps r\right) n, \dots
, \left(\frac{1}{r} -\eps r\right) n \right).$$
Hence, we can apply Theorem~\ref{myalmost} with $r, \eps r$ playing the roles of $k, \eps$. So at least one of the following
conditions holds:
\begin{itemize}
\item[(a)] $J_r$ contains a matching that covers all but at most $\ell $ vertices;
\item[(b)] $J_r$ is $\beta$-contained in $J(S,j)_r$ for some $1 \leq j \leq r-1$ and $S \subseteq V(J)=V(G)$
with $|S|=\lfloor jn/r \rfloor$.
\end{itemize}
If (a) holds then this implies that (i) is satisfied.
So we may assume that (b) holds. We will show that this implies that (ii) is satisfied.

Let $j$ be as in (b) and consider an arbitrary subtournament $T'$ of $T$ on $j+1$ vertices.
Suppose for a contradiction that there are at least $\alpha n^{j+1}$ $(j+1)$-tuples in $S$ that
span a copy $T''$ of $T'$ in $G$.
If $j=r-1$ then $T'=T$ and so this implies that $J_r$ contains at least $\alpha n^{j+1}=\alpha n^r >\beta n^r$ 
edges that lie in $S$. This is a contradiction as $J_r$ is $\beta $-contained in $J(S,j)_r$.
So suppose that $j <r-1$. Then (\ref{minA}) implies that, for each copy $T''$ of $T'$ in $G$, there are at least
\begin{align*}
& \frac{1}{(r-j-1)!} \left(1-\frac{j+1}{r}-(j+1)\eps \right) n \times 
\left(1-\frac{j+2}{r}-(j+2)\eps \right) n \times \dots \times \left(\frac{1}{r}-(r-1)\eps \right) n \\
& \geq \frac{1}{(r-j-1)!} \times \frac{1}{2 r^r} n^{r-j-1} \geq \frac{1}{2 r^{2r}} n^{r-j-1}
\end{align*}
$(r-j-1)$-tuples $X$ in $V(G)$ such that $V(T'') \cup X$ spans a copy of $T$ in $G$.
Since $\beta \ll \alpha , 1/r$, this implies that there are at least
$$\frac{1}{\binom{r}{j+1}} \times \alpha n^{j+1} \times \frac{1}{2 r^{2r}} n^{r-j-1} > \beta n^r$$
$r$-tuples in $V(G)$ that span a copy of $T$ and which contain at least $j+1$ vertices from $S$.
So $J_r$ contains more than $\beta n^r$ edges that contain at least $j+1$ vertices from $S$.
This is a contradiction as $J_r$ is $\beta $-contained in $J(S,j)_r$.

So there are at most $\alpha n^{j+1}$ $(j+1)$-tuples in $S$ that
span a copy of $T'$ in $G$.
Thus, since any $(j+1)$-tuple of vertices spans at most $2^{j+1}$ copies of $T'$, there are at most $2^{j+1} \alpha n^{j+1}\leq 2^{r} \alpha n^{j+1}\leq \alpha ' |S|^{j+1}$ copies of $T'$ in $G[S]$. Further, since $|S|=\lfloor jn/r \rfloor$,
(\ref{minA}) implies that
$$\delta ^0 (G[S]) \geq |S|-\frac{n}{r} -\eps n \geq \left(1- \frac{1}{j}-\eps' \right) |S|.$$
Apply Proposition~\ref{turanstab} with $G[S],T',j+1,\eps ', \alpha '$ playing the roles of $G,T,r,\eps, \alpha$.
This implies that $G[S]$ contains a $\sqrt{\eps '}$-independent set of size at least $|S|/j \geq
n/r-1$. Since $\sqrt{\eps '} |S|^2 \ll \gamma n^2$, this implies that $G$ contains a $\gamma$-independent
set of
size at least $n/r$, as required.
\endproof

\subsection{Proof of Theorem~\ref{almostthm+-}}
Define additional constants $ \beta, \alpha, \alpha ',  \eps '$ such that
$$0<1/n \ll 1/\ell \ll \eps \ll \beta \ll \alpha \ll \alpha '  \ll \eps ' \ll \gamma \ll 1/r.$$
Let $G$ be a digraph as in the statement of the theorem.

Let $r' \in \mathbb N$ and suppose that $T'$ is a copy of $T_{r'}$ in $G$.
Let $V(T')=\{x_1, \dots , x_{r'}\}$ where $x_i$ plays the role of the $i$th vertex
of $T_{r'}$. We say that $T'$ is \emph{consistent} if there exists $0\leq s \leq r'$
such that
\begin{itemize}
\item $d ^+ (x_i) \geq  ( 1- 1/r -\eps  ) n$ for all $i \leq s$;
\item $d ^- (x_i) \geq  ( 1- 1/r -\eps  ) n$ for all $i >s$.
\end{itemize}
We call $s$ a \emph{turning point of $T'$}. ($T'$ could have more than one turning point.)

(\ref{miny}) implies that every copy of $T_1$ in $G$ is
consistent. Suppose, for a contradiction, that $T$ is a copy
of $T_{r'}$ in $G$ that is not consistent (for some $r'\geq2$).
Let $y_i$ denote the vertex in $T$ that plays the role of the $i$th vertex of $T_{r'}$. Let $k$ be the smallest positive
integer such that $d^+ (y_k) <( 1- 1/r -\eps  ) n$
(and so $d^- (y_k) \geq( 1- 1/r -\eps  ) n$ by (\ref{miny})); such an 
integer exists else $T$ is consistent with turning point 
$r'$. Then there exists $k'>k$ such that 
$d^- (y_{k'}) <( 1- 1/r -\eps  ) n$ (otherwise $T$ is
consistent with turning point $k-1$). But then 
$y_ky_{k'} \in E(T) \subseteq E(G)$ where
 $d^+ (y_k) <( 1- 1/r -\eps  ) n$ and $d^- (y_{k'}) <( 1- 1/r -\eps  ) n$. This is a contradiction to the hypothesis
 of the theorem. Thus, every transitive tournament in $G$
 is consistent.

Let $J$ be the $r$-system on $V(G)$ where, for each $0 \leq i \leq r$, $J_i$ is defined as follows:
\begin{itemize}
\item Any $i$-tuple in $V(G)$ that spans a copy of $T_i$ in $G$ forms an
$i$-edge in $J_i$.
\end{itemize}
By construction  $J$
is an $r$-complex. Further, a matching in $J_r$ corresponds to a $T_r$-packing in $G$.
Clearly $\delta _0 (J)=n$. Suppose that $T$ is a (consistent) copy of $T_i$ in $G$ for some
$1 \leq i \leq r-1$ and let $s$ denote a turning point of $T$. As before, let $y_k$ denote the vertex
in $T$ that plays the role of the $k$th vertex in $T_i$. Set $N:=\bigcap _{k\leq s} N^+(y_k) \cap \bigcap _{k>s} N^- (y_k) $.
Since $T$ is consistent with turning point $s$, 
$$\left | N \right |
\geq \left( 1- \frac{i}{r} -i\eps \right)n \geq \left( 1- \frac{i}{r} -r\eps \right)n .$$
Further, given any $x \in N$, $V(T) \cup \{x\}$ spans a  copy of $T_{i+1}$ in
$G$. 
So $V(T) \cup \{x\}$
is an edge in $J$.
This implies that
$$\delta (J) \geq \left( n, \left(1 -\frac{1}{r} -\eps r\right) n, \left(1 -\frac{2}{r} -\eps r\right) n, \dots
, \left(\frac{1}{r} -\eps r\right) n \right).$$
Hence, we can apply Theorem~\ref{myalmost} with $r, \eps r$ playing the roles of $k, \eps$. So at least one of the following
conditions holds:
\begin{itemize}
\item[(a)] $J_r$ contains a matching that covers all but at most $\ell $ vertices;
\item[(b)] $J_r$ is $\beta$-contained in $J(S,j)_r$ for some $1 \leq j \leq r-1$ and $S \subseteq V(J)=V(G)$
with $|S|=\lfloor jn/r \rfloor$.
\end{itemize}
If (a) holds then (i) is satisfied.
So we may assume that (b) holds. We will show that this implies that (ii) is satisfied.

The argument now closely follows the proof of Theorem~\ref{almostthm}. Indeed,
let $j$ be as in (b).
Suppose for a contradiction that there are at least $\alpha n^{j+1}$ $(j+1)$-tuples in $S$ that
span a copy $T$ of $T_{j+1}$ in $G$.
If $j=r-1$ then $T_{j+1}=T_r$ and so this implies that $J_r$ contains at least $\alpha n^{j+1}=\alpha n^r >\beta n^r$ 
edges that lie in $S$. This is a contradiction as $J_r$ is $\beta $-contained in $J(S,j)_r$.
So suppose that $j <r-1$. 

Recall that every copy $T$ of $T_{j+1}$ in $G$ is consistent.
This implies that there are at least $(1-\frac{j+1}{r}-(j+1) \eps )n$ vertices
$x \in V(G)$ such that $V(T) \cup \{x\}$ spans a copy of $T_{j+2}$ in $G$.
Repeating this process we see that, for every copy  $T$ of $T_{j+1}$ in $G$,
 there are at least 
\begin{align*}
& \frac{1}{(r-j-1)!} \left(1-\frac{j+1}{r}-(j+1)\eps \right) n \times 
\left(1-\frac{j+2}{r}-(j+2)\eps \right) n \times \dots \times \left(\frac{1}{r}-(r-1)\eps \right) n \\
& \geq \frac{1}{(r-j-1)!} \times \frac{1}{2 r^r} n^{r-j-1} \geq \frac{1}{2 r^{2r}} n^{r-j-1}
\end{align*}
$(r-j-1)$-tuples $X$ in $V(G)$ such that $V(T) \cup X$ spans a copy of $T_r$ in $G$.
Since $\beta \ll \alpha , 1/r$, this implies that there are at least
$$\frac{1}{\binom{r}{j+1}} \times \alpha n^{j+1} \times \frac{1}{2 r^{2r}} n^{r-j-1} > \beta n^r$$
$r$-tuples in $V(G)$ that span a copy of $T_r$ and which contain at least $j+1$ vertices from $S$.
So $J_r$ contains more than $\beta n^r$ edges that contain at least $j+1$ vertices from $S$.
This is a contradiction as $J_r$ is $\beta $-contained in $J(S,j)_r$.

Thus, there are at most $\alpha n^{j+1}$ $(j+1)$-tuples in $S$ that
span a copy of $T_{j+1}$ in $G$.
Since any $(j+1)$-tuple of vertices spans at most $2^{j+1}$ copies of $T_{j+1}$, there are at most $2^{j+1} \alpha n^{j+1}\leq 2^{r} \alpha n^{j+1}\leq \alpha ' |S|^{j+1}$ copies of $T_{j+1}$ in $G[S]$.
 Further, since $|S|=\lfloor jn/r \rfloor$,
(\ref{miny}) implies that, for every $x \in V(G[S])$,
$$d ^+  _{G[S]} (x) \geq |S|-\frac{n}{r} -\eps n \geq \left(1- \frac{1}{j}-\eps' \right) |S| \text{ or }
d ^-  _{G[S]} (x)  \geq \left(1- \frac{1}{j}-\eps' \right) |S|.
$$
Apply Proposition~\ref{turanstab+-} with $G[S],j+1,\eps ', \alpha '$ playing the roles of $G,r,\eps, \alpha$.
This implies that $G[S]$ contains a $\sqrt{\eps '}$-independent set of size at least $|S|/j \geq
n/r-1$. Since $\sqrt{\eps '} |S|^2 \ll \gamma n^2$, this implies that $G$ contains a $\gamma$-independent
set of
size at least $n/r$, as required.

\endproof

\section{The connection lemmas}\label{consec}
In Section~\ref{secnon} we prove Theorems~\ref{nonex1} and~\ref{nonex2}. Roughly speaking, in both these theorems, when our digraph $G$ is
non-extremal  we require  a `small' $T$-absorbing set in $G$ that absorbs \emph{any} `very small' set of vertices in $G$ (for some
$T \in \mathcal T_r$).
 
The crucial idea in finding such an absorbing set is to first prove that our digraph $G$ has many `connecting structures' of a certain type. More precisely, to find our desired  absorbing set it suffices to show that, for any $x,y \in V(G)$, there are `many' $(r-1)$-sets 
$X \subseteq V(G)$ so that both $X \cup \{x\}$ and $X\cup \{y\}$ span copies of $T$ in $G$. Our main task therefore is to prove so-called `connection 
lemmas' that guarantee such sets $X$.
In the case when $T=C_3$ though, we may not be able to find such sets. However, in this case we instead find, for any $x,y \in V(G)$, `many' $5$-sets $X  \subseteq V(G)$ so that both $X \cup \{x\}$ and $X\cup \{y\}$ span copies of $2T$ in $G$.

The approach of proving connection lemmas to find absorbing structures  for packing problems has been very fruitful. Indeed,  
Lo and Markstr\"om~\cite{lo2, lo} used such connection lemmas to tackle perfect packing problems for hypergraphs (although their terminology for
these `connecting structures' differs from ours). Further, recently 
the author~\cite{dshs}  applied this method to prove a degree sequence version of the Hajnal--Szemer\'edi theorem.

A single connection lemma (Lemma~\ref{con+-}) for Theorem~\ref{nonex2} is given in Section~\ref{con2}.
However, for  Theorem~\ref{nonex1}, we need to prove a number of separate connection lemmas. 
In Section~\ref{sec5con} we prove the connection lemma for Theorem~\ref{nonex1} for those $T \in \mathcal
T_r$ where $r \geq 5$ (see Lemma~\ref{5con}). The proof relies on $T$ containing $T_3$ as a subtournament. (So this method
certainly cannot be applied in the case when $T=C_3$.) It is easy to see that all tournaments on at least four vertices
contain $T_3$. However for $T \in \mathcal T_4$, the minimum semidegree condition on our digraph $G$ is not high enough for the
proof method of Lemma~\ref{5con} to go through. Thus, we use a different approach to prove the connection lemma in this case
(see Section~\ref{sec4con}). This method makes use of a simple structural property of tournaments on four vertices (see Fact~\ref{T4fact}).
The case when $T=T_3$ is covered by Lemma~\ref{con+-}.

Finally, we need a separate connection lemma for when $T=C_3$ (see Section~\ref{sec3con}). Of all the connection lemmas, this one
has the most involved proof.
 This stems from the fact
that we now have two extremal cases. Thus, to find our connecting structures in a non-extremal digraph $G$ on $n$ vertices we must use both the
property that $G$ does not contain an `almost' independent set of size $n/3$ \emph{and} that $G$ does not $\alpha$-contain
$Ex(n)$.


\subsection{The connection lemma for Theorem~\ref{nonex2}}\label{con2}
The following connection lemma is a straightforward consequence of Proposition~\ref{turanstab+-}.
\begin{lemma}\label{con+-}
 Let $0<1/n \ll \eps \ll \eta \ll \gamma \ll 1/r$ where $n, r \in \mathbb N$ and $r \geq 3$.
Suppose that $G$ is a digraph on $n$ vertices such that, for any $z \in V(G)$,
\begin{align}\label{mincon+-}
 \ d ^+ (z) \geq \left ( 1- {1}/{r} -\eps \right ) n
\ \text{or} \ d ^- (z) \geq \left ( 1- {1}/{r} -\eps \right ) n.
\end{align}
Further, suppose that $G$ does not contain a $\gamma$-independent set on at least $n/r$ vertices.
Given any $x,y \in V(G)$, there exist at least $\eta n^{r-1}$ $(r-1)$-sets $X \subseteq V(G)$ such
that both $X \cup \{x\}$ and $X\cup \{y\}$ span copies of $T_r$ in $G$.
\end{lemma}
\proof
Let $x, y \in V(G)$. Suppose that $d ^+ (x) \geq  ( 1- 1/r -\eps  ) n$ and $d ^- (y) \geq  ( 1- 1/r -\eps  ) n$.
(The other cases are identical.) Set $S:=N^+ (x) \cap N^- (y)$. So $|S|\geq (1-2/r-2\eps)n =\frac{r-2-2\eps r}{r}n.$
This implies that,
\begin{align*}
|S|-\frac{n}{r}-\eps n & \geq |S| -\frac{|S|}{r-2-2\eps r}- 2\eps r |S| \geq |S|-\frac{|S|}{r-2} -3\eps r|S|-2\eps r |S|
\geq\left( 1 - \frac{1}{r-2} -\frac{\gamma ^2}{4} \right) |S| .
\end{align*}
Together with (\ref{mincon+-}) this implies that, for any $z \in S$,
\begin{align*}
 \ d_{G[S]} ^+ (z) \geq \left( 1 - \frac{1}{r-2} -\frac{\gamma ^2}{4} \right) |S|
\ \text{or} \ d _{G[S]} ^- (z) \geq  \left( 1 - \frac{1}{r-2} -\frac{\gamma ^2}{4} \right) |S|.
\end{align*}

Suppose for a contradiction that $G[S]$ contains at most $2 ^{3r-3} \eta |S|^{r-1}$ copies of $T_{r-1}$. Note that
$1/|S| \ll 2^{3r-3} \eta \ll \gamma ^2 /4 \ll 1/(r-1)$. Hence, applying Proposition~\ref{turanstab+-} with
$G[S], r-1, 2^{3r-3} \eta , \gamma ^2 /4$ playing the roles of $G,r,\alpha$ and $\eps$ respectively, we obtain a
$\gamma/2$-independent set $S'$ in $G[S]$ of size at least $|S|/(r-2)$. Note that 
$$\frac{|S|}{r-2} \geq \frac{(r-2-2\eps r)}{r(r-2)}n\geq \frac{n}{r} -2\eps n.$$
Therefore, $S'$ is a $\gamma/2$-independent set in $G$ of size at least $n/r-2\eps n$. By adding at most
$2\eps n$ arbitrary vertices to $S'$, we obtain a $\gamma$-independent set in $G$ of size at least $n/r$, a contradiction.

Thus, there are at least $2^{3r-3} \eta|S|^{r-1} \geq  2 ^{r-1} \eta n^{r-1}$ copies of $T_{r-1}$ in $G[S]$.
(The inequality here follows since $|S|\geq n/4$.) Any $(r-1)$-set in $S$ spans at most
$2^{r-1}$ different copies of $T_{r-1}$ in $G[S]$. So there are at least $\eta n^{r-1}$ $(r-1)$-sets $X \subseteq S$ that
span a copy of $T_{r-1}$ in $G[S]$. By definition of $S$, for each such set $X$, both $X\cup \{x\}$ and $X\cup \{y\}$ span
copies of $T_r$ in $G$, as required.
\endproof

\subsection{The connection lemma for tournaments on four vertices}\label{sec4con}
The following simple fact will be used in the proof of the connection lemma for tournaments on four vertices (Lemma~\ref{T4con}).
\begin{fact}\label{T4fact}
If $T \in \T_4$ then there is a subset $S \subseteq V(T)$ with $|S| \in \{1,2\}$ and such that given any $s \in S$, either
			\begin{itemize}
				\item $ss^\prime \in E(T)$ for all  $s' \in V(T) \setminus S$ or
				\item $s^\prime s \in E(T)$ for all  $s' \in V(T) \setminus S$.
			\end{itemize}
\end{fact}

\begin{lemma}\label{T4con}
Let $0 < 1/n \ll \eps \ll \eta \ll \gamma \ll 1$ and $T \in \T_4$. Suppose that $G$ is a digraph on $n$ vertices such that
			\begin{align}\label{minT4}
				 \delta^0(G) \ge \displaystyle{\left({3}/{4} - \eps\right)} n
			\end{align}
and so that $G$ does not contain any $\gamma$-independent set of size at least $n/4$.
Then, for any  $x,y \in V(G)$, there exist at least $\eta n^3$ 3-sets $X \subseteq V(G)$ such that $X \cup \{x\}$ and $X \cup \{y\}$ span copies of $T$ in $G$.
\end{lemma}

\begin{proof}
By Fact~\ref{T4fact}, $T$ contains a subset $S \subseteq V(T)$ with $|S| \in \{1,2\}$ and such that given any $s \in S$, either
			\begin{itemize}
				\item $ss^\prime \in E(T)$ for all  $s' \in V(T) \setminus S$ or
				\item $s^\prime s \in E(T)$ for all  $s' \in V(T) \setminus S$.
			\end{itemize}
We divide the proof into two cases depending on $|S|$.

\medskip

\noindent \textbf{Case 1:} $|S| = 1$\\
Let $V(T) = \{x_1, x_2, x_3, s\}$ where $S = \{s\}$. Consider the case when  $sx_i \in E(T)$ for $i = 1,2,3$. 
(The other case, when  $x_is \in E(T)$ for $i = 1,2,3$, is analogous.) Set $A := N^+_G(x) \cap N^+_G(y)$ and let $T':=T[x_1,x_2,x_3]$. 
Our aim is to find $\eta n^3$ 3-sets $X\subseteq A$ that span copies of $T'$ in $G[A]$. Then the choice of $s$ and $A$ ensures that each such set $X$ has the property that  $X \cup \{x\}$ and $X \cup \{y\}$ span copies of $T$ in $G$ (where $x$ and $y$ respectively play the role of $s$), as desired.

By (\ref{minT4}) we have that $|A|\geq (1/2-2\eps )n$ and so
$$\delta^0(G[A]) \ge |A|-(1/4+\eps )n \geq |A|- (1/4+\eps)\frac{|A|}{1/2-2\eps} \geq |A|-(1/2+5 \eps )|A|\geq(1/2-\gamma ^2 /4)|A|.$$
Suppose for a contradiction that $G[A]$ contains at most $65 \eta |A|^{3}$ copies of $T'$. Note that
$1/|A| \ll 65 \eta \ll \gamma ^2 /4 \ll 1/3$. Hence, applying Proposition~\ref{turanstab} with
$G[A], T', 3, 65 \eta , \gamma ^2 /4$ playing the roles of $G,T,r,\alpha$ and $\eps$ respectively, we obtain a
$\gamma/2$-independent set $A'$ in $G[A]$ of size at least $|A|/2 \geq (1/4-\eps)n$.  By adding at most
$\eps n$ arbitrary vertices to $A'$, we obtain a $\gamma$-independent set in $G$ of size at least $n/4$, a contradiction.

Thus, there are at least $65 \eta |A|^3 \geq 2 ^{3} \eta n^{3}$ copies of $T'$ in $G[A]$. Any $3$-set in $A$ spans at most
$2^{3}$ different copies of $T'$ in $G[A]$. So there are at least $\eta n^{3}$ $3$-sets $X \subseteq A$ that
span a copy of $T'$ in $G[A]$, as required.

\medskip

\noindent \textbf{Case 2:} $|S| = 2$\\
Let $V(T) = \{x_1, x_2, s_1, s_2\}$ where $S = \{s_1, s_2\}$.
Assume that $x_1, x_2 \in N_T ^+(s_1)$ and $x_1, x_2 \in N_T ^-(s_2)$. (The other cases can be dealt with analogously.) We may further assume that $s_2s_1 \in E(T)$ (otherwise, we can reset $S = \{s_1\}$ and then follow the argument from Case 1). Finally, we may assume that
$x_1x_2 \in E(T)$.

For each of our desired 3-sets $X$, the vertices $x$ and $y$ will play the role of $s_1$ in the copy of $T$ that spans $X \cup \{x\}$ and $X \cup \{y\}$, respectively. Let $s'_2 \in V(G)$ such that $s'_2x,s'_2 y\in E(G)$. By (\ref{minT4}) there are at least $(1/2-2\eps)n$ choices
for $s'_2$. Set $A:=N^+ _G(x) \cap N^+ _G(y) \cap N^- _G(s_2 ')$. Then (\ref{minT4}) implies that
$$|A| \ge \left(\frac{1}{4} - 3\eps \right) n.$$
Suppose that $G[A]$ contains at most $\gamma n^2 /2$ edges. Then by adding at most $3\eps n$ vertices to $A$, we obtain a 
$\gamma$-independent set in $G$ of size at least $n/4$, a contradiction.
So $G[A]$ contains at least $\gamma n^2/2$ edges.

Given any $x'_1x'_2 \in E(G[A])$, set $X:=\{s'_2,x'_1,x'_2\}$. By construction, $X\cup \{x\}$ spans a copy of $T$ in $G$ where
$x, s'_2, x'_1, x'_2$ play the roles of $s_1,s_2,x_1$ and $x_2$ respectively and $X\cup \{y\}$ spans a copy of $T$ in $G$ where
$y, s'_2, x'_1, x'_2$ play the roles of $s_1,s_2,x_1$ and $x_2$ respectively.

Recall that there are at least $(1/2-2\eps)n$ choices
for $s'_2$ and at least $\gamma n^2/2$ choices for $x'_1x'_2$. Overall, this implies that there are at least
$$(1/2-2\eps )n \times \frac{\gamma n^2}{2} \times \frac{1}{3!} \geq \eta n^3$$
choices for $X$, as desired.
\end{proof}

\subsection{The connection lemma for tournaments on at least five vertices}\label{sec5con}

\begin{lemma}\label{5con}
 Let $0<1/n \ll \eps \ll \eta \ll \gamma \ll 1/r$ where $n, r \in \mathbb N$ and $r \geq 5$, and let $T \in \mathcal T_r$.
Suppose that $G$ is a digraph on $n$ vertices such that
\begin{align}\label{absdeg}
\delta ^0 (G) \geq \left (1 -{1}/{r} -\eps \right )n 
\end{align}
and so that $G$ does not contain a $\gamma$-independent set on at least $n/r$ vertices.
Given any $x,y \in V(G)$, there exist at least $\eta n^{r-1}$ $(r-1)$-sets $X \subseteq V(G)$ such
that both $X \cup \{x\}$ and $X\cup \{y\}$ span copies of $T$ in $G$.
\end{lemma}
\proof
Define $\gamma '$ such that $\eta \ll \gamma ' \ll \gamma$.
Since $|T| \geq 5$, $T$ contains a copy of $T_3$. Let $V(T)=\{x_1, \dots, x_{r-3}, s_1, s_2, s_3 \}$
where $T[s_1,s_2,s_3]=T_3$ so that $s_is_j \in E(T)$ for $i<j$. 

Consider any $x,y \in V(G)$.
We now explain how we construct our desired $(r-1)$-sets $X$. For each such $X$, $x$ will play the role of
$x_1$ in the copy of $T$ spanning $X \cup \{x\}$ and $y$ will play the role of $x_1$ in the copy of
$T$ spanning $X \cup \{y\}$. When constructing each $X$ we  introduce a special vertex $x^*$ that will 
play the role of $s_1,s_2$ or $s_3$ in the copies of $T$ spanned by $X \cup \{x\}$ and $X\cup \{y\}$.

Let $x^* \in V(G)$ be such that $xx^*,x^*x, yx^*,x^*y \in E(G)$. By (\ref{absdeg}) there are at least
$(1-4/r-4\eps)n \geq n/2r$ choices for $x^*$. (Note that we could not guarantee such a vertex $x^*$ exists if $r
=4$. This is the reason why we cannot generalise this proof to work for $r \geq 4$.)

Next we iteratively choose vertices $x'_2, \dots , x'_{r-3} \in V(G)$ such that,
for each $2 \leq i \leq r-3$, the following conditions hold:
\begin{itemize}
\item[($a_i$)] $x^* x'_i, x'_i x^* \in E(G)$;
\item[($b_i$)] $xx'_i , yx'_i \in E(G)$ if $x_1x_i \in E(T)$  and $x'_i x, x'_iy \in E(G)$ if $x_i x_1 \in E(T)$;
\item[($c_i$)]  $x'_{i'}x'_{i} \in E(G)$ if $x_{i'}x_i \in E(T)$ and
$x'_i x'_{i'} \in E(G)$ if $x_i x_{i'} \in E(T)$ (for each $2 \leq i' < i$).
\end{itemize}
Suppose that, for some $2 \leq j \leq r-4$, we have already chosen $x'_2, \dots , x'_{j}$ so
that ($a_i$)--($c_i$) hold for $ 2 \leq i \leq j$. Then (\ref{absdeg}) implies that there are at least
$$\left( 1 - \frac{j+3}{r}-(j+3) \eps \right) n \geq \left( 1- \frac{r-1}{r} -(r-1) \eps \right)n 
\geq \frac{n}{2r}$$
choices for $x'_{j+1}$ so that ($a_{j+1}$)--($c_{j+1}$) are satisfied.

Conditions ($a_i$)--($c_i$) ensure that $x'_i$ can play the role of $x_i$ for each $2\leq i \leq r-3$. Since there are double edges between $x^*$ and 
$x,y, x'_2, \dots , x'_{r-3}$ there is currently freedom as to whether $x^*$ will play the role
of $s_1, s_2$ or $s_3$.

Our next task is to construct sets $S_1, S_2,S_3$ of `candidates' to play the role of $s_1,s_2$ and $s_3$
respectively. More precisely, we say that a vertex $z \in V(G)$ is a \emph{candidate for $s_1$ in $G$} if:
\begin{itemize}
\item $zx^* \in E(G)$;
\item $xz,yz \in E(G)$ if $x_1s_1 \in E(T)$ and $zx,zy \in E(G)$ if $s_1 x_1 \in E(T)$;
\item $x'_i z \in E(G)$ if $x_i s_1 \in E(T)$ and $z x'_i \in E(G)$ if $ s_1 x_i \in E(T)$
(for each $2 \leq i \leq r-3$).
\end{itemize}
We say that a vertex $z \in V(G)$ is a \emph{candidate for $s_2$ in $G$} if:
\begin{itemize}
\item $x^*z \in E(G)$;
\item $xz,yz \in E(G)$ if $x_1s_2 \in E(T)$ and $zx,zy \in E(G)$ if $s_2 x_1 \in E(T)$;
\item $x'_i z \in E(G)$ if $x_i s_2 \in E(T)$ and $z x'_i \in E(G)$ if $ s_2 x_i\in E(T)$
(for each $2 \leq i \leq r-3$).
\end{itemize}
Similarly we say that a vertex $z \in V(G)$ is a \emph{candidate for $s_3$ in $G$} if:
\begin{itemize}
\item $x^*z \in E(G)$;
\item $xz,yz \in E(G)$ if $x_1s_3 \in E(T)$ and $zx,zy \in E(G)$ if $s_3 x_1 \in E(T)$;
\item $x'_i z \in E(G)$ if $x_i s_3 \in E(T)$ and $z x'_i \in E(G)$ if $ s_3x_i \in E(T)$
(for each $2 \leq i \leq r-3$).
\end{itemize}
(Note that it is important that the first condition in the definition of a candidate for $s_1$ differs
from the first condition in the definitions of candidates for $s_2$ and $s_3$.)

Let $S_1 , S_2, S_3$ denote the set of candidates for $s_1,s_2$ and $s_3$ respectively. (\ref{absdeg})
implies that 
\begin{align}\label{SSS}
|S_1|,|S_2|,|S_3| \geq \left( 1 -\frac{r-1}{r} -(r-1) \eps \right ) n \geq \frac{n}{r} -r\eps n.
\end{align}

\medskip

\noindent{\bf Case 1:} $|S_1 \cup S_2|\geq n/r +\gamma ' n$. 
\newline
\noindent In this case (\ref{absdeg}) implies that every $z \in S_3$ receives at least $\gamma 'n-\eps n
\geq \gamma ' n/2$ edges from $S_1 \cup S_2$ in $G$. So (\ref{SSS}) implies that there are at least 
$\gamma ' n^2 /3r$ edges in $G$ with startpoint in $S_1 \cup S_2$ and endpoint in $S_3$.
Without loss of generality assume that
there are at least $\gamma ' n^2 /6r$ edges in $G$ with startpoint in  $S_1 $ and endpoint in $S_3$. Let $s'_1s'_3$ be such an edge. Notice
that, by definition of candidates for $s_1$ and $s_3$, $\{s'_1, x^*, s'_3\}$ spans a copy of $T_3$ in
$G$ with $s'_1, x^*$ and $s'_3$ playing the roles of $s_1$, $s_2$ and $s_3$ respectively. 
Set $X:=\{ x^*, x'_2, \dots , x'_{r-3}, s'_1, s'_3\}$. By construction 
$X\cup \{x\}$ spans a copy of $T$ in $G$ where $x$ plays the role of $x_1$, $x'_i$ plays the
role of $x_i$ (for $2 \leq i \leq r-3$), $s'_1$ plays the role of $s_1$, $x^*$ plays the role of $s_2$
and $s'_3$ plays the role of  $s_3$. Similarly, 
$X\cup \{y\}$ spans a copy of $T$ in $G$ where $y$ plays the role of $x_1$, $x'_i$ plays the
role of $x_i$ (for $2 \leq i \leq r-3$), $s'_1$ plays the role of $s_1$, $x^*$ plays the role of $s_2$
and $s'_3$ plays the role of  $s_3$.

\medskip

\noindent{\bf Case 2:} $|S_1 \cup S_2|< n/r +\gamma ' n$. 
\newline
\noindent
In this case,
$$|S_1 \cap S_2|=|S_1|-|S_1\setminus S_2| \stackrel{(\ref{SSS})}{\geq} \left(\frac{n}{r}
-r\eps n \right) -\left( \gamma ' n +r \eps n \right) \geq \frac{n}{r}- 2 \gamma ' n.$$
 Note that there must be at least $\gamma n^2 /2\geq \gamma 'n^2 /6r$
  edges in $S_1 \cap S_2$ otherwise,
 by adding at most $2 \gamma ' n$ arbitrary vertices to $S_1 \cap S_2$, we obtain a $\gamma$-independent
 set in $G$ of size at least $n/r$, a contradiction. Consider any edge $s'_1s'_2 \in E(G[S_1 \cap S_2])$.

By definition of candidates for $s_1$ and $s_2$, $\{s'_1,  s'_2, x^*\}$ spans a copy of $T_3$ in
$G$ with $s'_1, s'_2$ and $x^*$ playing the roles of $s_1$, $s_2$ and $s_3$ respectively. (Indeed, by definition of $S_1 \cap S_2$ there is a double edge from $x^*$ to both $s_1$ and $s_2$ in $G$. Further,
$s'_1s'_2 \in E(G)$.)
Set $X:=\{ x^*, x'_2, \dots , x'_{r-3}, s'_1, s'_2\}$. By construction 
$X\cup \{x\}$ spans a copy of $T$ in $G$ where $x$ plays the role of $x_1$, $x'_i$ plays the
role of $x_i$ (for $2 \leq i \leq r-3$), $s'_1$ plays the role of $s_1$, $s'_2$ plays the role of $s_2$
and $x^*$ plays the role of  $s_3$. Similarly,
$X\cup \{y\}$ spans a copy of $T$ in $G$.

\medskip

Recall that there are at least $n/2r$ choices for $x^*$, at least $n/2r$ choices for
each $x'_i$ and at least $\gamma' n^2 /6r$ choices for the edges selected in Cases 1 and 2.
Overall, this implies that there are at least
$$\frac{n}{2r} \times \left(\frac{n}{2r} \right) ^{r-4} \times \frac{\gamma ' n^2}{6r} \times
\frac{1}{(r-1)!} \geq \eta n^{r-1}$$
choices for $X$, as desired.
\endproof

\subsection{The connection lemma for cyclic triangles}\label{sec3con}
\begin{lemma}\label{C3con}
Let $0 < 1/n \ll \eps \ll \eta \ll \gamma, \alpha \ll 1$. Suppose that $G$ is a digraph on $n$ vertices such that
\begin{align}\label{C3min}
\delta^0(G) \ge \displaystyle{\left({2}/{3} - \eps\right)} n.
\end{align}			
Further suppose that	
			\begin{itemize}
				\item $G$ does not contain any $\gamma$-independent set of size at least $n/3$, and
				\item $G$ does not $\alpha$-contain $Ex(n)$.
			\end{itemize}
Then, given any $x,y \in V(G)$, at least one of the following conditions holds. 
\begin{itemize}
\item[(i)] There are at least $\eta n^2$ $2$-sets $X\subseteq V(G)$ such that $X \cup \{x\}$ and $X \cup \{y\}$ span copies of $C_3$ in $G$.
\item[(ii)] There are at least $\eta n^5$ $5$-sets $X\subseteq V(G)$ such that $X \cup \{x\}$ and $X \cup \{y\}$ span copies of $2C_3$ in $G$.
\end{itemize}

\end{lemma}

\begin{proof}
Define additional constants $\eps ', \eps ''$ so that $ \eta  \ll \eps ' \ll \eps '' \ll \gamma, \alpha$. Consider any $x,y \in V(G)$. Set $A':= N_G ^+(x) \cap N_G ^+(y)$ and $B' := N_G ^-(x) \cap N_G ^-(y)$.
Note that (\ref{C3min}) implies  $|A'|,|B'|\geq (1/3-2\eps )n$.
Further, define $A:= A' \setminus B'$, $B:=B'\setminus A'$, $C := V(G) \setminus (A' \cup B')$ and $D:= A' \cap B'$. (So $A,B,C,D$ is a partition of $V(G)$.)

Note that given any $ab \in E(G)$ with $a \in A'$ and $b \in B'$, $\{x,a,b\}$ and $\{y,a,b\}$ span copies of $C_3$ in $G$. Thus, if there are at least
$2 \eta n^2$ such edges $ab \in E(G)$, then we obtain at least $\eta n^2$ $2$-sets $X\subseteq V(G)$ such that $X \cup \{x\}$ and $X \cup \{y\}$ span copies of $C_3$ in $G$, as desired.
Therefore, we may assume that
\begin{itemize}
\item[($\alpha $)] there are at most $2 \eta n^2$ edges $ab \in E(G)$ with $a \in A'$ and $b \in B'$.
\end{itemize}
Suppose that $|B'| >(1/3+8 \eta )n$. Then by (\ref{C3min}), every vertex in $A'$ sends out at least $8 \eta n -\eps n \geq 7 \eta n$ edges to $B'$ in $G$.
Hence, there are at least $$7\eta n |A'|\geq 7 \eta n \times  (1/3-2\eps)n  > 2 \eta n^2$$
edges $ab \in E(G)$ with $a \in A'$ and $b \in B'$, a contradiction to ($\alpha $). Together with an analogous argument, this implies that
\begin{align}\label{A'B'bound}
\left( \frac{1}{3}-2\eps \right )n \leq |A'|,|B'| \leq \left( \frac{1}{3}+8 \eta \right ) n
\end{align}
and so
\begin{align}\label{Cbound}
\left ( \frac{1}{3} -16 \eta \right) n \leq |C|.
\end{align}

Suppose that $|D| > \left(\frac{1}{3} - \eps ''\right)n.$
Then by ($\alpha$),  $e(G[D]) \leq 2 \eta n^2$. By adding at most $\eps ''n$ arbitrary vertices to $D$ we obtain a $\gamma$-independent set in $G$ of size at least $n/3$, a contradiction.
Hence, $$|D|\leq \left(\frac{1}{3} - \eps ''\right)n.$$
We now split the proof into two cases depending on the size of $D$.

\medskip

\noindent\textbf{Case 1:} $|D| < \eps ' n$.
\newline
\noindent In this case,
\begin{align}\label{ABbound}
\left(\frac{1}{3} - 2\eps '\right) n \leq  \left(\frac{1}{3} - 2\eps\right) n - \eps ' n \stackrel{(\ref{A'B'bound})}{\le} |A|, |B| \stackrel{(\ref{A'B'bound})}{\le} \left(\frac{1}{3} + 8\eta \right) n
\end{align}
and so 
\begin{align}\label{Cbound2}
\left ( \frac{1}{3} -16 \eta \right) n \stackrel{(\ref{Cbound})}{\leq} |C| \stackrel{(\ref{ABbound})}{\leq} \left( \frac{1}{3}+ 4 \eps ' \right ) n.
\end{align}
By ($\alpha $), all but at most $2 \sqrt{\eta} n$ vertices $a \in A$ send out at most $\sqrt{\eta}n$ edges to $B$ in $G$. 
So each such vertex $a$ sends out at least $(2/3-\eps )n-\sqrt{\eta}n -\eps ' n -|C| \geq (1/3- 6 \eps ' )n$ edges to $A $ in $G$ and at least 
$(2/3-\eps )n-\sqrt{\eta}n -\eps ' n -|A| \geq (1/3- 2 \eps ' )n$ edges to $C $ in $G$.
Altogether, this implies that
\begin{align}\label{GA}
e({G}[A]) \geq (|A|-2\sqrt{\eta }n)(1/3-6 \eps ')n
 \stackrel{(\ref{ABbound})}{\geq} (|A|-2\sqrt{\eta }n)( |A|-7 \eps 'n) \stackrel{(\ref{ABbound})}{\geq} |A|^2 -3 \eps ' n^2
\end{align}
and 
\begin{align}\label{GAC}
e({G}[A,C]) \geq (|A|-2\sqrt{\eta }n) (1/3-2 \eps ')n
\stackrel{(\ref{Cbound2})}{\geq}(|A|-2\sqrt{\eta }n)(|C|-6\eps 'n) \stackrel{(\ref{ABbound}),(\ref{Cbound2})}{\geq} |A||C|-3 \eps ' n^2.
\end{align}
An analogous argument implies that 
\begin{align}\label{GB}
e({G}[B]){\geq} |B|^2- 3 \eps ' n^2  \ \text{ and } \ e({G}[C,B]) {\geq} |C||B|-3 \eps ' n^2.
\end{align}

\smallskip

Suppose that $d^- _G (x,A)\geq \eps '' n$ and  $d^- _G (y,A)\geq \eps '' n$. Then since $e({G}[A]) \geq |A|^2 -3 \eps ' n^2$, there are at least
$$\frac{\eps '' n(\eps '' n-1)}{2} -3 \eps ' n^2 \geq \eta n^2$$ pairs of distinct vertices $a,a'$ where $a \in (N^- _G (x) \cap A)$, $a' \in (N^- _G (y) \cap A)$ and
$aa',a'a \in E(G)$. For each such pair $a,a'$, $\{x,a,a'\}$ and $\{y,a,a'\}$ both span copies of $C_3$ in $G$ (in fact, they both span copies of $K_3 ^-$).
This implies that (i) is satisfied. 
Similarly, (i) holds  if both  $d^+ _G (x,B)\geq \eps '' n$ and  $d^+ _G (y,B)\geq \eps '' n$.

We may therefore assume that $d^- _G (x,A)< \eps '' n$ or  $d^- _G (y,A)<\eps '' n$. Without loss of generality assume that
\begin{align*}
d^- _G (x,A)< \eps '' n.
\end{align*}
This implies that
\begin{align}\label{b2}
 d^- _G (x,C)\stackrel{(\ref{C3min}),(\ref{ABbound})}{\ge} (2/3-\eps)n-\eps 'n -\eps ''n -(1/3+8 \eta )n \geq (1/3-2 \eps '' )n \stackrel{(\ref{Cbound2})}{\ge}
|C|-3 \eps ''n.
\end{align}
Furthermore, we may assume that  $d^+ _G (x,B)< \eps '' n$ or $d^+ _G (y,B)< \eps '' n$.
We now deal with these two subcases separately. 

\smallskip

\textbf{Case 1a:} $d^+ _G (x,B)< \eps '' n$.

In this case we will show that (ii) is satisfied.
Note that
\begin{align}\label{xC}
d^+ _G (x, C)\stackrel{(\ref{C3min}),(\ref{ABbound})}{\geq} (2/3- \eps)n-\eps ''n -\eps ' n -(1/3+8 \eta)n  \stackrel{(\ref{Cbound2})}{\ge}
|C|-3 \eps ''n.
\end{align}
If $d^+ _G (y, C) > 3 \eps ''n$ then (\ref{xC}) implies that there is a vertex $c \in (N^+ _G (x) \cap N^+ _G (y) \cap C)=A'\cap C$. But by definition $A' \cap C=\emptyset$, a contradiction.
Thus,
\begin{align}\label{yC}
d^+ _G (y, C) \leq 3 \eps ''n.
\end{align}

\begin{claim}\label{cBC}
If $e(G[B,C]) \geq 6 \eps '' n^2$ then (ii) is satisfied.
\end{claim}
\proof
Suppose that $e(G[B,C]) \geq 6 \eps '' n^2$. This implies that there are at least $5 \eps '' n$ vertices $c \in C$ that receive at least $\eps '' n$ edges from $B$ in $G$.
By (\ref{GB}), all but at most $3 \sqrt{\eps '} n$ vertices $c \in C$ send out at least $|B|-\sqrt{\eps '} n$ edges to $B$ in $G$. 
Together with (\ref{xC}) this implies that there are at least $5 \eps '' n -3 \sqrt{\eps '} n -3 \eps '' n -1\geq \eps ''n $ vertices $c \in C \setminus \{y\}$ such that
\begin{itemize}
\item $c \in N^+ _G (x)$;
\item $d^- _G (c,B) \geq \eps '' n$ and $d^+ _G (c,B) \geq |B|-\sqrt{\eps '} n.$
\end{itemize}
Fix such a vertex $c$. By the choice of $c$ and (\ref{GB}) there are at least $\eps '' n -\sqrt{\eps '} n-3\sqrt{\eps '}n \geq \eps '' n/2$ vertices $b_1 \in B$ so that
\begin{itemize}
\item $b_1c,cb_1 \in E(G)$;
\item $d^- _G (b_1, B)\geq |B|-\sqrt{\eps '} n$.
\end{itemize}
Fix such a vertex $b_1$. By definition of $B$, $b_1 \in N^- _G (x)$. Thus, $\{x,c,b_1\}$ spans a copy of $C_3$ in $G$.

(\ref{C3min}), (\ref{ABbound}) and (\ref{yC}) imply that
$$d^+ _G (y,B) \geq (2/3-\eps )n-3\eps '' n-\eps 'n -(1/3+8 \eta )n \geq |B|-4\eps '' n.$$
Together with (\ref{GB}) this implies that there are at least $|B|- 5 \eps '' n$ vertices $b_2 \in B \setminus \{b_1\}$ so that
\begin{itemize}
\item $b_2 \in N^+ _G (y)$;
\item $d^+ _G (b_2, B), d^- _G (b_2, B)\geq |B|-\sqrt{\eps '} n$.
\end{itemize}
Fix such a vertex $b_2$. Next fix a vertex $b_3 \in B \setminus \{b_1\}$ such that
\begin{itemize}
\item $b_3 \in N^+ _G (b_2)$;
\item $d^+ _G (b_3, B)\geq |B|-\sqrt{\eps '} n$.
\end{itemize}
There are at least $|B|-5\sqrt{\eps '} n$ choices for $b_3$. By definition of $B$, $b_3 \in N^- _G (y)$. Thus, $\{y,b_2,b_3\}$ spans a copy of $C_3$ in $G$.

Finally, choose a vertex $b_4 \in B$ such that
\begin{itemize}
\item $b_4 \in N^+ _G (c) \cap N^- _G (b_1) \cap N^- _G (b_2) \cap N^+ _G (b_3)$.
\end{itemize}
The choice of $c,b_1,b_2$ and $b_3$ ensures that there are at least $|B|-4\sqrt{\eps '} n$ choices for $b_4$.
Set $X:= \{c,b_1,b_2,b_3,b_4\}$. By construction both $X \cup \{x\}$ and $X\cup \{y\}$ span copies of $2C_3$ in $G$ (see Figure~2).
\begin{figure}[htb!]
\begin{center}\footnotesize
\includegraphics[width=0.5\columnwidth]{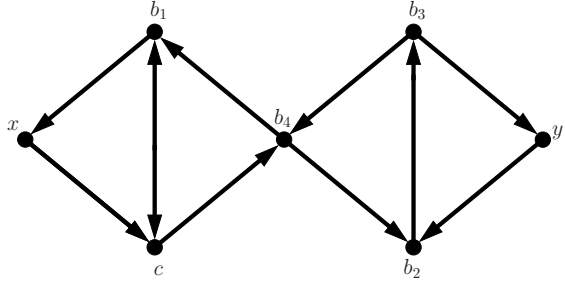}  
\caption{The connecting structure in Case 1a}
\end{center}
\end{figure}

Recall that there are at least $\eps '' n$ choices for $c$, at least $\eps'' n/2$ choices for $b_1$, at least $|B|-5\eps '' n$ choices for $b_2$,
at least $|B|-5\sqrt{\eps '} n$ choices for $b_3$ and at least $|B|-4\sqrt{\eps '} n$ choices for $b_4$.
Overall, this implies that there are at least
$$\eps '' n \times \frac{\eps '' n}{2} \times (|B|-5\eps '' n ) \times (|B|-5\sqrt{\eps '} n) \times (|B|-4\sqrt{\eps '} n) \times \frac{1}{4!} \geq \eta n^5$$
choices for $X$. So indeed (ii) is satisfied. This proves the claim.
\endproof

\medskip

Assume for a contradiction that $e(G[B,C]) < 6 \eps '' n^2$. 
This implies that
\begin{align}\label{GCf}
e(G[C]) & \geq  \delta ^- (G)|C| -e(G[B,C])-e(G[A,C]) -e (G[D,C])  \\ & \stackrel{(\ref{C3min})}{\geq} (2/3-\eps)n|C| -6\eps '' n^2 -|A||C|-\eps 'n |C|\nonumber  \\ &
\stackrel{(\ref{ABbound}),(\ref{Cbound2})}{\geq} |C|(2/3-\eps -24\eps '' -1/3-8\eta -\eps ')n \geq |C|(1/3-25 \eps '' )n   \nonumber \\ & \stackrel{(\ref{Cbound2})}{\geq} |C|^2 -\sqrt{\eps ''} n^2.\nonumber
\end{align}
Similarly,
\begin{align}\label{GBAf}
e(G[B,A]) & \geq  \delta ^+ (G)|B| -e(G[B])-e(G[B,C]) -e (G[B,D])  \\ &  \stackrel{(\ref{C3min})}{\geq} (2/3-\eps)n|B| -|B|^2-6\eps '' n^2 -\eps 'n |B| \nonumber \\ &
\stackrel{(\ref{ABbound})}{\geq} |B|(2/3-\eps -1/3-8\eta -24\eps ''  -\eps ')n \geq |B|(1/3-25 \eps '' )n  \nonumber \\ &  \stackrel{(\ref{ABbound})}{\geq} |B||A| -\sqrt{\eps ''} n^2.\nonumber
\end{align}
Let $A'', B'',C''$ be a partition of $V(G)$ such that 
\begin{itemize}
\item $\lfloor n/3 \rfloor \leq |A''|\leq|B''|\leq|C''| \leq \lceil n/3 \rceil$;
\item $|A''\setminus A|,|B''\setminus B|, |C''\setminus C|\leq 3 \eps 'n$.
\end{itemize}
Such a partition exists by (\ref{ABbound}) and (\ref{Cbound2}).
Further,
\begin{itemize}
\item $e(G[A'',C'']) \stackrel{(\ref{GAC})}{\geq} |A''||C''|-\alpha n^2 /6;$
\item $e(G[C'',B'']) \stackrel{(\ref{GB})}{\geq} |C''||B''|-\alpha n^2 /6;$
\item $e(G[B'',A'']) \stackrel{(\ref{GBAf})}{\geq} |B''||A''|-\alpha n^2 /6;$
\item $e(G[A'']) \stackrel{(\ref{GA})}{\geq} |A''|^2-\alpha n^2 /6;$
\item $e(G[B'']) \stackrel{(\ref{GB})}{\geq} |B''|^2-\alpha n^2 /6;$
\item $e(G[C'']) \stackrel{(\ref{GCf})}{\geq} |C''|^2-\alpha n^2 /6.$
\end{itemize}
This implies that $G$ $\alpha$-contains $Ex(n)$, a contradiction. So $e(G[B,C]) \geq 6 \eps '' n^2$. Claim~\ref{cBC} therefore implies that (ii) holds, as required.

\smallskip

\textbf{Case 1b:} $d^+ _G (y,B)< \eps '' n$.

In this case we will show that (i) is satisfied. Since $d^+ _G (y,B)< \eps '' n$,
\begin{align}\label{yC2}
d^+ _ G (y,C) \stackrel{(\ref{C3min}),(\ref{ABbound})}{\geq} (2/3- \eps)n-\eps ''n -\eps ' n -(1/3+8 \eta)n \geq (1/3-2 \eps '')n \stackrel{(\ref{Cbound2})}{\ge}
|C|-3 \eps ''n.
\end{align}
If $d^- _G (y, C) > 3 \eps ''n$ then (\ref{b2}) implies that there is a vertex $c \in (N^- _G (x) \cap N^- _G (y) \cap C)=B'\cap C$. But by definition $B' \cap C=\emptyset$, a contradiction.
Thus,
\begin{align*}
d^- _G (y, C) \leq 3 \eps ''n.
\end{align*}
This implies that 
\begin{align}\label{yA1}
d^- _ G (y,A) \stackrel{(\ref{C3min}),(\ref{ABbound})}{\geq} (2/3- \eps)n-3\eps ''n -\eps ' n -(1/3+8 \eta)n \geq (1/3-4 \eps '')n \stackrel{(\ref{ABbound})}{\ge}
|A|-5 \eps ''n.
\end{align}

\begin{claim}\label{c1b}
If $e(G[C,A]) \geq 14 \eps '' n^2$ then (i) is satisfied.
\end{claim}
\proof
Suppose that $e(G[C,A]) \geq 14 \eps '' n^2$. This implies that there are at least $8 \eps '' n$ vertices $c \in C$ that send out at least $6\eps ''n$ edges to $A$ in $G$.
By (\ref{GAC}) all but at most $3\sqrt{\eps '} n$ vertices $c \in C$ receive at least $|A|-\sqrt{\eps '}n$ edges from $A$ in $G$. Together with (\ref{b2}) and (\ref{yC2}) this 
implies that there are at least $8 \eps '' n -3 \sqrt{\eps '} n -6 \eps '' n \geq \eps '' n$ vertices $c \in C$ so that
\begin{itemize}
\item $c \in N^- _G (x) \cap N^+ _G (y)$; 
\item $d^+ _G (c,A) \geq 6 \eps '' n$ and $d^- _G (c,A) \geq |A|-\sqrt{\eps '}n$.
\end{itemize}
Fix such a vertex $c$. Let $a \in A$ such that
\begin{itemize}
\item $ a \in N^-_G (y) \cap N^+ _G (c) \cap N^- _G (c)$.
\end{itemize}
The choice of $c$ together with (\ref{yA1}) implies that there are at least
$6 \eps ''n -\sqrt{\eps '}n- 5 \eps '' n \geq \eps '' n/2$ choices for $a$. Since $a \in A$, $a \in N^+ _G (x)$. Set $X:=\{a,c\}$. By construction $X\cup \{x\}$ and $X\cup \{y\}$ both span copies
of $C_3$ in $G$.
In total there are at least 
$$\eps '' n \times \eps '' n/2 \geq \eta n^2$$ choices for $X$. Therefore (i) is satisfied. This proves the claim.
\endproof

\medskip

Assume for a contradiction that $e(G[C,A]) < 14 \eps '' n^2$. 
This implies that
\begin{align}\label{GC1}
e(G[C]) & \geq  \delta ^+ (G)|C| -e(G[C,A])-e(G[C,B]) -e (G[C,D])  \\ & \stackrel{(\ref{C3min})}{\geq} (2/3-\eps)n|C| -14\eps '' n^2 -|C||B|-\eps 'n |C| \nonumber  \\ &
\stackrel{(\ref{ABbound}),(\ref{Cbound2})}{\geq} |C|(2/3-\eps -50\eps '' -1/3-8\eta -\eps ')n \geq |C|(1/3-51 \eps '' )n \nonumber \\ &  \stackrel{(\ref{Cbound2})}{\geq} |C|^2 -\sqrt{\eps ''} n^2.\nonumber
\end{align}
Similarly,
\begin{align}\label{GBA}
e(G[B,A]) & \geq  \delta ^- (G)|A| -e(G[A]) -e (G[C,A]) -e(G[D,A]) \\ & \stackrel{(\ref{C3min})}{\geq} (2/3-\eps)n|A| -|A|^2-14\eps '' n^2 -\eps 'n |A| \nonumber \\ &
\stackrel{(\ref{ABbound})}{\geq} |A|(2/3-\eps -1/3-8\eta -50\eps ''  -\eps ')n \geq |A|(1/3-51 \eps '' )n \nonumber \\ &  \stackrel{(\ref{ABbound})}{\geq} |B||A| -\sqrt{\eps ''} n^2.\nonumber
\end{align}
By arguing precisely as in Case 1a, (\ref{GA})--(\ref{GB}), (\ref{GC1}) and (\ref{GBA}) imply that
 $G$ $\alpha$-contains $Ex(n)$, a contradiction. So $e(G[C,A]) \geq 14 \eps '' n^2$. Claim~\ref{c1b} therefore implies that (i) holds, as required.

\medskip

\noindent\textbf{Case 2:} $\eps ' n \leq |D| \leq (1/3 -\eps '')n $.
\newline
\noindent
In this case we will show that (ii) is satisfied.
Set $d:=|D|/n$. So
\begin{align}\label{d}
\eps ' \leq d \leq 1/3 - \eps ''.
\end{align}
This implies that
\begin{align}\label{AB2}
 \eps '' n/2 \stackrel{(\ref{d})}{\leq} (1/3-2 \eps -d)n \stackrel{(\ref{A'B'bound})}{\leq} |A|,|B| \stackrel{(\ref{A'B'bound})}{\leq}(1/3+8 \eta -d )n .
\end{align}
Thus,
\begin{align}\label{Cb3}
 |C| \stackrel{(\ref{AB2})}{\leq} n-dn-2 (1/3-2 \eps -d)n =(1/3+4 \eps +d)n 
\end{align}
and
\begin{align}\label{Cb4}
 |C| \stackrel{(\ref{AB2})}{\geq} n-dn-2 (1/3+8 \eta -d)n =(1/3-16 \eta  +d)n \stackrel{(\ref{d})}{\geq} (1/3+\eps '/2)n.
\end{align}
Hence,
\begin{align}\label{conxy}
d^+ _G (x,C), d^+ _G (y,C) \stackrel{(\ref{C3min}), (\ref{Cb4})}{\geq} (2/3-\eps )n - (2/3-\eps '/2)n \geq \eps ' n/3.
\end{align}

By ($\alpha $), all but at most $2 \sqrt{\eta} n$ vertices $b \in B$ receive at most $\sqrt{\eta}n$ edges from $A\cup D=A'$ in $G$. 
So each such vertex $b$ receives at least 
$$(2/3-\eps )n-\sqrt{\eta}n  -|B| \stackrel{(\ref{AB2})}{\geq}  (1/3-2\sqrt{\eta} +d)n 
\stackrel{(\ref{Cb3})}{\geq} |C|-3 \sqrt{\eta} n$$ 
edges from $C$ in $G$.
This implies that
\begin{align}\label{GBb}
e({G}[C,B]) \geq (|B|-2\sqrt{\eta }n)(|C|-3 \sqrt{\eta} n)\geq  |C||B| -5 \sqrt{\eta} n^2.
\end{align}
By ($\alpha$), 
\begin{align}\label{GDAD}
e(G[D])+e(G[A,D]) \leq 2\eta n^2 \ \text{ and } \ e(G[D])+e(G[D,B]) \leq 2\eta n^2.
\end{align}
Therefore,
\begin{align}\label{GBD}
e(G[B,D])  & \geq \delta ^- (G)|D|-e(G[A,D])-e(G[D])-e(G[C,D]) 
\\ &
\stackrel{(\ref{C3min}),(\ref{GDAD})}{\geq} (2/3-\eps )n|D|- 2 \eta n^2 -|C||D| 
 \stackrel{(\ref{Cb3})}{\geq}  (1/3-\sqrt{\eta }-d )n|D| 
 \stackrel{(\ref{AB2})}{\geq} |B||D|- \sqrt{\eta } n^2 \nonumber
\end{align}
and 
\begin{align}\label{GCD1}
e(G[D,C])   \stackrel{(\ref{C3min}),(\ref{GDAD})}{\geq} (2/3-\eps )n|D|- 2 \eta n^2 -|D||A| 
 \stackrel{(\ref{AB2})}{\geq}  (1/3-\sqrt{\eta }+d )n|D| 
 \stackrel{(\ref{Cb3})}{\geq} |C||D|- \sqrt{\eta } n^2.
\end{align}

Fix $c_1 \in C\setminus \{y\}$ such that
\begin{itemize}
\item $xc_1 \in E(G)$;
\item $d^+ _G (c_1, B) \geq |B| - \eta ^{1/4} n$;
\item $d^- _G (c_1, D) \geq |D| - \eta ^{1/4} n$.
\end{itemize}
(\ref{conxy}), (\ref{GBb}) and (\ref{GCD1}) imply that there are at least $\eps ' n/3-6 {\eta}^{1/4} n -1 \geq \eps 'n /4$ choices for $c_1$.
Next fix $b_1 \in B$ such that 
\begin{itemize}
\item $c_1b_1 \in E(G)$;
\item $d^+ _G (b_1, D) \geq |D| - \eta ^{1/4} n$.
\end{itemize}
The choice of $c_1$ together with (\ref{GBD}) implies that there are at least $|B|-2\eta ^{1/4} n  \geq \eps ''n /3$ choices for $b_1$. Further,
$b_1x \in E(G)$ by definition of $B$. Thus, $\{x,b_1,c_1\}$ spans a copy of $C_3$ in $G$.

Fix $c_2 \in C\setminus \{c_1, x \}$ such that
\begin{itemize}
\item $yc_2 \in E(G)$;
\item $d^+ _G (c_2, B) \geq |B| - \eta ^{1/4} n$;
\item $d^- _G (c_2, D) \geq |D| - \eta ^{1/4} n$.
\end{itemize}
Again (\ref{conxy}), (\ref{GBb}) and (\ref{GCD1}) imply that there are at least $\eps 'n /4$ choices for $c_2$.
Next fix $b_2 \in B\setminus \{b_1\}$ such that 
\begin{itemize}
\item $c_2b_2 \in E(G)$;
\item $d^+ _G (b_2, D) \geq |D| - \eta ^{1/4} n$.
\end{itemize}
There are at least $|B|-2\eta ^{1/4} n  -1 \geq \eps ''n /3$ choices for $b_2$. Since $b_2y \in E(G)$, $\{y,b_2,c_3\}$ spans a copy of $C_3$ in $G$.

Finally let $d \in (N^+ _G (b_1) \cap N^+ _G (b_2) \cap  N^- _G (c_1)  N^- _G (c_2) \cap D)$. There are at least
$|D|- 4 \eta ^{1/4}n \geq \eps 'n/2$ choices for $d$. 
Set $X:=\{b_1,b_2,c_1,c_2, d\}$. By construction $X \cup \{x\}$ and $X \cup \{y\}$ both span copies of $2 C_3$ in $G$ (see Figure~3).
\begin{figure}[htb!]
\begin{center}\footnotesize
\includegraphics[width=0.5\columnwidth]{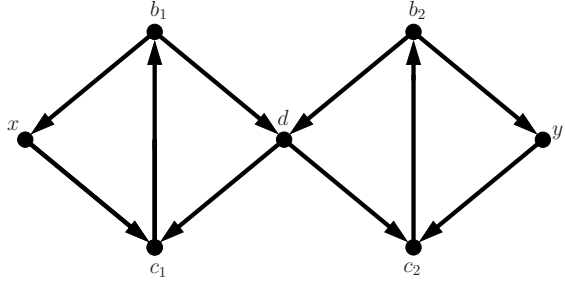}  
\caption{The connecting structure in Case 2}
\end{center}
\end{figure}

In total there are at least
$$ \frac{\eps ' n}{4} \times \frac{\eps ''n}{3} \times \frac{\eps ' n}{4} \times \frac{\eps ''n}{3} \times \frac{\eps ' n}{2} \times \frac{1}{5!} \geq \eta n^5$$
choices for $X$. Therefore (ii) is satisfied, as desired.
\end{proof}

\section{Proof of Theorems~\ref{nonex1} and~\ref{nonex2}}\label{secnon}
In this section we apply our connection lemmas to prove Theorems~\ref{nonex1} and~\ref{nonex2}. 
Following the ideas in \cite{rrs2,rrs}, we first show in Lemma~\ref{lem:abs} that in order to find the absorbing set described in Theorems~\ref{nonex1} and~\ref{nonex2}, it suffices to prove that there are at least $\xi n^{2r^2}$ $T$-absorbing $2r^2$-sets for every fixed $r$-set from $V(G)$.
\begin{lemma}[Absorbing lemma]
\label{lem:abs}
Let $0<\xi \ll 1$ and let $r \geq 2$. Then there exists an $n_0 \in \mathbb N$ such that the following holds. Let $T \in \mathcal T_r$. Consider a digraph $G$ on $n \geq n_0$ vertices. Suppose that any $r$-set of vertices $Q \subseteq V(G)$ can be $T$-absorbed by at least
$\xi n^{2r^2}$ $2r^2$-sets of vertices from $V(G)$. Then $V(G)$ contains a set $M$ so that
\begin{itemize}
\item $|M|\leq \xi n$;
\item $M$ is a $T$-absorbing set for any $W \subseteq V(G) \setminus M$ such that $|W| \in r \mathbb N$ and  $|W|\leq \xi ^2 n$.
\end{itemize}
\end{lemma}
The proof of Lemma~\ref{lem:abs} follows the same ideas as other such absorbing lemmas in the area. In particular, the proof of Lemma~\ref{lem:abs} follows the proof of Lemma~5.2 in~\cite{zhao}
very closely. For completeness, we give the proof in Section~\ref{subabs}.

\begin{lemma}\label{absorb1}
Let $0 < 1/n \ll \eps \ll \xi \ll \gamma , \alpha \ll 1/r$ where $n,r \in \mathbb N$ and $r \geq 3$, and let $T \in \mathcal T_r$. Suppose that $G$ is a digraph on $n$ vertices so that
\begin{align}\label{minabs}
\delta ^0  (G) \geq \left ( 1-{1}/{r} -\eps \right ) n.
\end{align}
Further suppose that	
			\begin{itemize}
				\item $G$ does not contain any $\gamma$-independent set of size at least $n/r$, and
				\item If $T=C_3$ then $G$ does not $\alpha$-contain $Ex(n)$.
			\end{itemize}
			Then there are at least $\xi n^{2r^2}$ $T$-absorbing $2r^2$-sets in $V(G)$ for every $r$-subset of $V(G)$.
\end{lemma}
Theorem~\ref{nonex1} follows immediately from Lemmas~\ref{lem:abs} and~\ref{absorb1}. 
Similarly, Theorem~\ref{nonex2} follows immediately from Lemma~\ref{lem:abs} and the following result.
\begin{lemma}\label{absorb2}
Let $0 < 1/n \ll \eps \ll \xi \ll \gamma  \ll 1/r$ where $n,r \in \mathbb N$ and $r \geq 3$. Suppose that $G$ is a digraph on $n$ vertices so that, for any $x \in V(G)$,
\begin{align}\label{minabs+-}
d^+ (x) \geq \left(1-{1}/{r} -\eps \right) n \ \text{ or } \ d^- (x) \geq \left(1-{1}/{r} -\eps \right) n.
\end{align}
Further suppose that	
$G$ does not contain any $\gamma$-independent set of size at least $n/r$.
			Then there are at least $\xi n^{2r^2}$ $T_r$-absorbing $2r^2$-sets in $V(G)$ for every $r$-subset of $V(G)$.
\end{lemma}
The rest of the section is devoted to the proofs of Lemmas~\ref{lem:abs}--\ref{absorb2}.

\subsection{Proof of Lemma~\ref{lem:abs}.}\label{subabs}
Given an $r$-set $Q \subseteq V(G)$, let $L_Q$ denote the family of all $T$-absorbing $2r^2$-sets for $Q$ in $\binom{V(G)}{2r^2}$. By assumption, $|L_Q| \ge \xi n^{2r^2}$.
Let $F$ be the family of $2r^2$-sets obtained by selecting each of the $\binom{n}{2r^2}$ elements of $\binom{V(G)}{2r^2}$
independently with probability $p:= \xi / n^{2r^2-1}$.
Then
$$\mathbb{E}(|F|) = p \binom{n}{2r^2}  < \frac{\xi}{(2r^2)!} n \ \text{ and } \
\mathbb{E}(|L_Q\cap F|)\ge p \,\xi n^{2r^2}= \xi^2 n$$ for every set $Q\in  \binom{V(G)}{r}$.

Since $n$ is sufficiently large, Proposition~\ref{chernoff} implies that with high probability we have
\begin{equation}\label{eq:F}
    |F|\le 2 \mathbb{E}(|F|)< \frac{2\xi}{(2r^2)!} n,
\end{equation}
\begin{equation}\label{eq:LF}
   |L_Q\cap F| \ge \frac12 \, \mathbb{E}(|L_Q\cap F|) \ge \frac{\xi^2}{2} n \quad \text{for all  } Q\in \binom{V(G)}{r}.
\end{equation}
Let $Y$ be the number of intersecting pairs of members of $F$. Then
\[
\mathbb{E}(Y)\le p^2 \binom{n}{2r^2} 2r^2 \binom{n}{2r^2-1}\le \frac{\xi^2 n}{(2r^2-1)!(2r^2-1)!}.
\]
By Markov's bound, the probability that $Y\le \frac{2 \xi^2}{(2r^2-1)!(2r^2-1)!} n$ is at least $\frac{1}{2}$. Therefore we can find a family $F$ of $2r^2$-sets satisfying \eqref{eq:F} and \eqref{eq:LF} and having at most $\frac{2 \xi^2}{(2r^2-1)!(2r^2-1)!} n$
intersecting pairs. Removing all non-absorbing $2r^2$-sets and one set from each of the intersecting pairs in $F$, we obtain a family $F'$ of disjoint $T$-absorbing $2r^2$-sets such that $|F'|\le |F|\le \frac{2\xi}{(2r^2)!} n \leq \xi n/2r^2$ and for all $Q\in \binom{V(G)}{r}$,
\begin{equation}\label{eq:LF'}
    |L_Q\cap F'| \ge \frac{\xi^2}{2} n - \frac{2 \xi^2}{(2r^2-1)!(2r^2-1)!} n > \frac{\xi^2}{r} n.
\end{equation}
Let $M$ denote the disjoint union of the sets in $F'$. Then $|M|=|F'|2r^2\leq \xi n$.
Since $F'$ consists of disjoint $T$-absorbing sets and each $T$-absorbing set is covered by a perfect $T$-packing, $G[M]$ contains a perfect $T$-packing.
Now let $W\subseteq V(G)\backslash M$ be a set of at most $\xi^2 n$ vertices such  that $|W|= r\ell$ for some $\ell \in \mathbb N$. We arbitrarily partition $W$ into $r$-sets $Q_1, \dots, Q_{\ell}$. Because of \eqref{eq:LF'}, we are able to $T$-absorb each $Q_i$ with a different $2r^2$-set from $L_{Q_i}\cap F'$. Therefore $G[M\cup W]$ contains a perfect $T$-packing, as desired.
\endproof

\subsection{Proof of Lemma~\ref{absorb1}.}\label{subabs2}
Define $\eta$ such that $\xi \ll \eta \ll \gamma$. Note that (\ref{minabs}) implies that there are at least
\begin{align}\label{noT}
n \times \left (1- \frac{1}{r}-\eps \right ) n \times \left (1- \frac{2}{r}-2\eps \right ) n \times \dots \times  \left (1- \frac{r-1}{r}-(r-1)\eps \right ) n \times \frac{1}{r!} \geq 2 \eta ^2 n^r
\end{align}
$r$-sets in $ V(G)$ that span copies of $T$ in $G$.
\begin{claim}\label{claimabs}
For any $x, y \in V(G)$ there are at least $\eta ^4 n^{2r-1}$ $(2r-1)$-sets $X\subseteq V(G)$ such that both $X \cup \{x\}$ and $X \cup \{y\}$ span copies of $2T$ in $G$.
\end{claim}
\proof
Suppose for a contradiction that Claim~\ref{claimabs} is false. 
Then if $T=C_3$, certainly Lemma~\ref{C3con}(ii) does not hold.
In particular, Lemmas~\ref{con+-}, \ref{T4con}--\ref{C3con} imply that, for any 
$x, y \in V(G)$ there are at least $\eta  n^{r-1}$ $(r-1)$-sets $X'\subseteq V(G)$ such that both $X' \cup \{x\}$ and $X' \cup \{y\}$ span copies of $T$ in $G$.
Fix such a set $X'$.  (\ref{noT}) implies that there are least
$$2 \eta^2 n^r - (r+1)\binom{n}{r-1} \geq \eta ^2 n^r$$
$r$-sets $X'' \subseteq V(G)$ that span copies of $T$ in $G$ and that are disjoint from $X' \cup \{x ,y\}$. Fix such a set $X''$ and define $X:=X' \cup X''$.
By construction both $X \cup \{x\}$ and $X \cup \{y\}$ span copies of $2T$ in $G$.
Further, since $\eta \ll  1/r$, there are at least
$$ \eta n^{r-1} \times \eta ^2 n^r \times \frac{1}{\binom{2r-1}{r-1}} \geq \eta ^4 n^{2r-1}$$
choices for $X$, a contradiction. This proves the claim.
\endproof

\medskip

Consider any $r$-subset $Q:=\{x_1, x_2, \dots , x_r \}$ of $V(G)$. Fix some $r$-subset $Y :=\{y_1, y_2 , \dots, y_r\}$ of $V(G)$ that spans a copy of $T$ in $G$
and that is disjoint from $Q$.
(\ref{noT}) implies that there are least
$$2 \eta ^2 n^r - r\binom{n}{r-1} \geq \eta ^2 n^r$$
choices for $Y$.
Next fix a $(2r-1)$-set $X_1\subseteq V(G)$ such that both $X _1 \cup \{x_1\}$ and $X_1 \cup \{y_1\}$ span copies of $2T$ in $G$ and so that
$X_1$ is disjoint from $Q\cup Y$. Claim~\ref{claimabs} implies that there are at least 
$$ \eta ^4 n^{2r-1} - 2r \binom{n}{2r-2} \geq \eta ^4 n^{2r-1} /2$$
choices for $X_1$. Similarly, Claim~\ref{claimabs} implies that we can iteratively choose $(2r-1)$-sets $X_2, \dots , X_r \subseteq V(G)$ such that, for each $2 \leq i \leq r$:
\begin{itemize}
\item Both $X _i \cup \{x_i\}$ and $X_i \cup \{y_i\}$ span copies of $2T$ in $G$;
\item $X_i$ is disjoint from $Q \cup Y$;
\item $X_i$ is disjoint from $X_j$ for all $1 \leq j <i$;
\item There are at least $\eta ^4 n^{2r-1} /2$ choices for $X_i$.
\end{itemize}
Set $S:= Y \cup \bigcup _{1\leq i \leq r} X_i$. Then $S$ is a $T$-absorbing $2r^2$-set for $Q$. Indeed, $G[X_i \cup \{y_i\}]$ contains a perfect $T$-packing
for all $1\leq i \leq r$ so $G[S]$ contains a perfect $T$-packing. Furthermore, $G[X_i \cup \{x_i\}]$ contains a perfect $T$-packing
for all $1\leq i \leq r$ and $Y$ spans a copy of $T$ in $G$ so $G[S \cup Q]$ contains a perfect $T$-packing.

In summary, there are at least $\eta ^2 n^r$ choices for $Y$ and at least $\eta ^4 n^{2r-1} /2$ choices for each of the $X_i$. Since each $T$-absorbing $2r^2$-set may be counted
$\binom{2r^2}{r}\binom{r(2r-1)}{2r-1}\binom{(r-1)(2r-1)}{2r-1} \dots \binom{2(2r-1)}{2r-1}$ times
there are at least
$$\eta ^2 n^r \times \left( \frac{\eta ^4 n^{2r-1}}{2} \right ) ^{r} \times \frac{1}{\binom{2r^2}{r}\binom{r(2r-1)}{2r-1}\binom{(r-1)(2r-1)}{2r-1} \dots \binom{2(2r-1)}{2r-1} }\geq \xi n^ {2r^2}$$
$T$-absorbing $2r^2$-sets for $Q$, as desired. 
\endproof

\subsection{Proof of Lemma~\ref{absorb2}.}
Define $\eta$ such that $\xi \ll \eta \ll \gamma$. 
By Lemma~\ref{con+-}, for every vertex $x \in V(G)$, there are at least $\eta n^{r-1}$ $(r-1)$-sets $X \subseteq V(G)$ such that $X \cup \{x\}$ spans a copy of $T_r$ in $G$. Thus, there are at least
\begin{align}\label{noT+-}
n \times \eta n^{r-1} \times \frac{1}{r} \geq 2 \eta ^2 n^r
\end{align}
$r$-sets in $V(G)$ that span copies of $T_r$ in $G$.

By now following the proof of Lemma~\ref{absorb1} identically (applying (\ref{noT+-}) and Lemma~\ref{con+-}) we conclude that there are at least $\xi n^{2r^2}$ $T_r$-absorbing $2r^2$-sets in $V(G)$ for every $r$-subset of $V(G)$, as required.
\endproof


\section{Tools for the proof of Lemma~\ref{ex1}}\label{secex}
In Sections~\ref{secex1} and~\ref{secexC3} we deal with the extremal cases of Theorems~\ref{mainthm}. The proof of Lemma~\ref{ex1} builds on the ideas from the extremal case in~\cite{kss}.
(Note though that~\cite{kss} concerns embedding \emph{powers of Hamilton cycles} in \emph{graphs}.)
 In this section we give a number of results that will be applied in the proof of Lemma~\ref{ex1}.

\subsection{Perfect $T$-packings in the non-extremal case}
In the proof of Lemma~\ref{ex1} we will apply the following result which is a direct consequence of Theorems~\ref{nonex1} and~\ref{almostthm}
(its proof is implicit in the proof of Theorem~\ref{mainthm} given in Section~\ref{pfsec}).
\begin{thm}\label{1nonex}
Let $0 < 1/n \ll \eps  \ll \gamma  \ll 1/r$ where $n,r \in \mathbb N$ and $r \geq 3$ so that $r$ divides $n$, and let $T \in \mathcal T_r\setminus \{C_3\}$. Suppose that $G$ is a digraph on $n$ vertices so that
\begin{align*}
\delta ^0  (G) \geq \left ( 1-{1}/{r} -\eps \right ) n.
\end{align*}
Further suppose that	 $G$ does not contain any $\gamma$-independent set of size at least $n/r$.
Then $G$ contains a perfect $T$-packing.
\end{thm}

\subsection{Perfect $K_r$-packings in $r$-partite digraphs}
We will also apply the following immediate consequence of Theorem~\ref{rpart}.
\begin{thm}\label{rpart2}
Given $r \in \mathbb N$ there exists an $n_0 \in \mathbb N$ such that the following holds. Suppose $G$ is an $r$-partite digraph
with vertex classes $V_1, \dots, V_r$ where $|V_i|=n\geq n_0$ for all $1\le i \leq r$.
If $$\bar{\delta}^+(G) , \bar{\delta}^-(G)\geq (1-1/2r)n+1$$ then $G$ contains a perfect $K_r$-packing.
\end{thm}
Here $\bar{\delta} ^+(G)$ ($\bar{\delta} ^-(G)$) denotes the minimum outdegree (indegree) of a vertex from one vertex class $V_i$ to another vertex class $V_j$.
\subsection{Matchings in digraphs}
A \emph{matching}  in a (di)graph $G$ is a collection of vertex-disjoint edges $M \subseteq E(G)$. We write $V(M)$ for the set of vertices covered by the edges from $M$.
We say that $M$ is a \emph{$d$-matching} if $|M|=d$. We say that $M$ is a \emph{perfect matching} if $V(M)=V(G)$.

\begin{prop}\label{simplematch}
Let $d, n \in \mathbb N$. Suppose that $G$ is a graph on $n \geq 2d$ vertices such that $\delta (G) \geq d$. Let $X \subseteq V(G)$ such that $|X|=d$. Then $G$ contains a $d$-matching  that covers all
the vertices in $X$.
\end{prop}
\proof
It is easy to see that $G$ contains a $d$-matching. Let $M$ be a $d$-matching in $G$ that covers the maximum number of vertices from $X$. Suppose for a contradiction that there is a vertex $x \in X$ uncovered by $M$.
In particular, $M$ covers more vertices in $V(G) \setminus X$ than in $X$.
There exist non-negative integers $a,b,c$ such that $a+b+c=d$ and:
\begin{itemize}
\item[(i)] $M$ contains precisely $a$ edges $wz $ where $w \in X$ and $z \in V(G)\setminus X$;
\item[(ii)] $M$ contains precisely $b$ edges with both endpoints in $X$;
\item[(iii)] $M$ contains precisely $c$ edges with both endpoints in $V(G)\setminus X$.
\end{itemize}
Since $M$ covers more vertices in $V(G)\setminus X$ than in $X$, $b<c$ (and so $c\geq 1$). Suppose $x$ has a neighbour  $y \in V(G) \setminus V(M)$. Then add $xy$ to $M$ and delete an edge  $wz $ from $M$
such that $w,z \in V(G)\setminus X$. Then $M$ is a $d$-matching covering more vertices in $X$ than before, a contradiction. So $x$ only has neighbours in $V(M)$. 

Suppose $wz$ is an edge in $M$ such that $w \in X$ and $z \in V(G)\setminus X$. If $xw \in E(G)$ then delete $wz$ from $M$ and add $xw$ to $M$. So again $M$ is a $d$-matching covering more vertices in $X$ than before, a contradiction. Thus, $x$ is not adjacent to $w$. A similar argument shows that, if $wz \in M$ with $w,z \in V(G)\setminus X$, then $xw,xz \not \in E(G)$.
Together with (i)--(iii) this shows that $x$ has at most $a+2b <a+b+c=d$ neighbours in $G$, a contradiction, as desired.
\endproof
The following immediate consequence of Proposition~\ref{simplematch} will be applied in the proof of  Lemma~\ref{ex1}.
\begin{prop}\label{simplematch2}
Let $d, n \in \mathbb N$. Suppose that $G$ is a digraph on $n \geq 2d$ vertices such that, for any $x \in V(G)$, 
$d^+(x) \geq d$ or $d^- (x) \geq d$. Let $X \subseteq V(G)$ such that $|X|=d$. Then $G$ contains a $d$-matching  that covers all
the vertices in $X$.
\end{prop}

Let $\eps >0$. Suppose that $G$ is a (di)graph $G$ on $n$ vertices. Then we say that $G$ is \emph{$\eps$-close  to $2K_{n/2}$} if there exists a partition $A,B$ of $V(G)$ such that $|A|=\lfloor n/2 \rfloor$,
$|B|= \lceil n/2 \rceil$ and $e_G (A,B) \leq \eps n^2$. 

\begin{prop}\label{matchstab}
Let $\gamma >0$ and $n \in \mathbb N$ be even such that $1/n \ll \gamma$. Suppose that $G$ is a graph on $n$ vertices so that
\begin{align}\label{matmin}
\delta (G) \geq (1/2-\gamma) n.
\end{align}
Then at least one of the following conditions holds:
\begin{itemize}
\item $G$ contains a $3 \gamma$-independent set of size at least $n/2$;
\item $G$ is $3\gamma$-close to $2K_{n/2}$;
\item $G$ contains a perfect matching.
\end{itemize}
\end{prop}
\proof
Suppose that $G$ does not contain a perfect matching. Let $M$ be a maximal matching in $G$. So there exists distinct $x,y \in V(G) \setminus V(M)$. The maximality of $M$ implies that
$N(x),N(y) \subseteq V(M)$. Define $SN(x):=\{ z \in V(M) : wz \in M \text{ for some } w\in N(x)\}$. Define $SN(y)$ analogously. (\ref{matmin}) implies that
\begin{align}\label{SNbound}
|SN(x)|,|SN(y)| \geq (1/2 -\gamma )n.
\end{align}

Suppose for a contradiction that there is an edge $zz' \in E(G)$ such that $z \in SN(x)$ and $z' \in SN(y)$. If $zz' \in M$ then by definition of $SN(x)$ and $SN(y)$, $xz',yz \in E(G)$.
Define $M':=(M\setminus \{zz'\}) \cup \{xz',yz\} \subseteq E(G)$. Thus, $M'$ is a larger matching than $M$, a contradiction. 
So $zz' \not \in M$. Let $w,w' \in V(M) $ such that $wz,w'z' \in M$. Then by definition of $SN(x)$ and $SN(y)$, $xw,yw' \in E(G)$.
Set $M':=(M\setminus \{wz,w'z'\}) \cup \{xw,yw',zz'\} \subseteq E(G)$. Then $M'$ is a larger matching than $M$, a contradiction. 
This proves that there is no such edge $zz'$.

Define  $SN(x,y):= SN(x) \cap SN(y)$.
Suppose that $SN(x,y) \not = \emptyset$. Consider any $z \in SN(x,y)$. Then in $G$, $z$ does not have any neighbours in $SN(x) \cup SN(y)$. So (\ref{matmin}) implies that 
$|SN(x) \cup SN(y)| \leq (1/2+\gamma )n$. So together with (\ref{SNbound}) this implies that $|SN(x,y)|\geq (1/2- 3 \gamma )n$. Further, $SN (x,y)$ is an independent set in $G$.
By adding at most $3 \gamma n$ arbitrary vertices to $SN(x,y)$ we obtain a $3 \gamma$-independent set of size at least $n/2$ in $G$.

Finally, suppose that $SN(x,y) =\emptyset$. So $SN(x)$ and $SN(y)$ are disjoint and $e_G (SN(x),SN(y))=0$. Together with (\ref{SNbound}) this implies that $G$ is $3\gamma$-close to $2K_{n/2}$, as desired.
\endproof

We will also apply the following consequence of Proposition~\ref{matchstab}.
\begin{prop}\label{matchstab2}
Let $\gamma >0$ and $n \in \mathbb N$ be even such that $1/n \ll \gamma$. Suppose that $G$ is a digraph on $n$ vertices so that, for every $x \in V(G)$,
\begin{align*}
d^+ (x) \geq \left(1/2 -\gamma \right) n \ \text{ or } \ d^- (x) \geq \left(1/2 -\gamma \right) n.
\end{align*}
Then at least one of the following conditions holds:
\begin{itemize}
\item $G$ contains a $6 \gamma$-independent set of size at least $n/2$;
\item $G$ is $6\gamma$-close to $2K_{n/2}$;
\item $G$ contains a perfect matching.
\end{itemize}
\end{prop}

\section{Proof of Lemma~\ref{ex1}}\label{secex1}
Define constants $\gamma , \gamma _1, \gamma _2, \dots , \gamma _{r}$ and $n_0 \in \mathbb N$ such that
$$0<1/n_0 \ll \gamma \ll \gamma _1  \ll \gamma _2 \ll \dots \ll \gamma _{r} \ll 1/r.$$
Let $T \in \mathcal T_r$ and $G$ be a digraph on $n\geq n_0$ vertices as in the statement of the lemma.
By assumption $G$ contains a $\gamma$-independent set $A_1$ of size $n/r$. (So $A_1$ is also a $\gamma _1$-independent set in $G$.) Consider $G_1:=G\setminus A_1$. If $G_1$ contains a $\gamma_2$-independent set
$A_2$ of size $n/r$ set $G_2:=G_1\setminus A_2$. (Note that $A_2$ is also a $\gamma_2$-independent set in $G$.) Otherwise let $B:=V(G_1)$. 
Repeating this process, for some $1 \leq s \leq r$, we obtain a partition $A_1, \dots , A_s, B$ of $V(G)$ such that:
\begin{itemize}
\item $A_i$ is a $\gamma _i$-independent set of size $n/r$ in $G$ (for all $1 \leq i \leq s$);
\item $|B|=(r-s)n/r$ and $G[B]$ does not contain a $\gamma _{s+1}$-independent set of size $n/r$.
\end{itemize}
(The latter condition is vacuous if $B=\emptyset$.)
If $B = \emptyset$ define additional constants $\alpha , \beta ',  \beta $ so that $$\gamma _r \ll \alpha \ll \beta ' \ll \beta \ll 1/r.$$ If $B\not = \emptyset$ then define $\alpha , \beta ', \beta , \eta$ so that $$\gamma _s \ll \alpha \ll \beta ' \ll \beta  \ll \eta \ll \gamma _{s+1}.$$

Let $\delta >0$ and $1\leq i \leq s$. We now introduce a number of definitions.
\begin{itemize}
\item We say that a vertex $ x\in A_i$ is \emph{$(\delta,i)$-bad} if $d^+ _G (x,A_i) \geq \delta n$ or $d^- _G (x,A_i) \geq \delta n $.
Otherwise we say that $x$ is \emph{$(\delta,i)$-good}.
\item We say that a vertex $x \in V(G) \setminus A_i$ is \emph{$(\delta ,i)$-exceptional} if $d^+ _G (x,A_i) , d^- _G (x,A_i)\leq \delta n$. Otherwise we say that $x$ is \emph{$(\delta,i)$-acceptable}.
\item We say that a vertex $x \in V(G) \setminus A_i$ is \emph{$(\delta ,i)$-excellent} if $d^+ _G (x,A_i) , d^- _G (x,A_i)\geq |A_i|-\delta n$. 
\item Similarly, we say that a vertex $x \in V(G) \setminus B$ is \emph{$(\delta ,B)$-excellent} if $d^+ _G (x,B) , d^- _G (x,B)\geq |B|-\delta n$. 
\end{itemize}
Later on we will modify the vertex classes $A_1,  \dots , A_s, B$. When referring to, for example, $(\delta,i)$-bad vertices, we mean with respect to the current class $A_i$ and not the original class.

For each $1 \leq i \leq s$, since $A_i$ is a $\gamma _i$-independent set in $G$ and $\gamma _i \ll \alpha  \ll \beta $, there are at most $\alpha  n$ vertices in $A_i$ that are $(\beta , i)$-bad.
Furthermore, (\ref{minex1}) implies that there are at least
$$2\delta ^0 (G)|A_i|-2\gamma _i n^2 \geq 2 |A_i||V(G)\setminus A_i|-2 \gamma _i n^2$$
edges in $G$ with one endpoint in $A_i$ and the other in $V(G) \setminus A_i$. So as $\gamma _i \ll \alpha  \ll \beta ' $, there are at most $\alpha  n$ vertices $x \in V(G) \setminus A_i$ that
are not $(\beta ' ,i)$-excellent. (This implies that there are at most $\alpha n$ $(\beta ,i)$-exceptional vertices.)

\medskip

\noindent
{\bf Modifying the partition $A_1, \dots ,A_s ,B$.}
Let $t$ be the largest integer such that there exists both $t$ $(\beta,1)$-bad vertices $x_1, \dots , x_t \in A_1$ \emph{and} $t$ $(\beta ,1)$-exceptional vertices $y_1, \dots , y_t \in V(G) \setminus A_1$.
Note that $t \leq \alpha n$. Move $y_1, \dots , y_t$ into $A_1$ and remove $x_1, \dots, x_t$ from $A_1$ so that each $x_i$ replaces $y_i$ in their respective classes. (So if originally $y_i \in A_j$ then 
we move $x_i$ into $A_j$. If originally $y_i \in B$ then we move $x_i$ into $B$.) We call this `phase' \emph{Step~1}.

If a vertex $x \in A_1$ was initially $(\beta, 1)$-good then after Step~1, $x$ is still $(\beta + \alpha, 1)$-good. Further, each $y_i$ is now $(\beta+\alpha, 1)$-good. So if $A_1$ initially contained \emph{precisely}
$t$ $(\beta,1)$-bad vertices, then $A_1$ no longer contains any $(\beta +\alpha, 1)$-bad vertices.

If a vertex $y \in V(G) \setminus A_1$ was  initially $(\beta,1)$-acceptable, then after Stage~1, $y$ is still  $(\beta- \alpha,1)$-acceptable. Further, each vertex $x_i$ is 
$(\beta- \alpha,1)$-acceptable. So if initially there were \emph{precisely}
$t$ $(\beta,1)$-exceptional vertices, then after Stage~1 there are no  $(\beta -\alpha, 1)$-exceptional vertices.

Thus, after Stage~1 we have that:
\begin{itemize}
\item $A_i$ is an $\alpha$-independent set of size $n/r$ in $G$ (for all $1\leq i \leq s$);
\item If $B\not = \emptyset$ then $G[B]$ does not contain any $(\gamma _{s+1} -2\alpha r)$-independent set of size $n/r$;
\item There are no $(\beta +\alpha, 1)$-bad vertices in $A_1$ or there are no $(\beta -\alpha, 1)$-exceptional vertices in $V(G)\setminus A_1$;
\item $A_i$ contains at most $2 \alpha n$ $(\beta +\alpha, i)$-bad vertices (for all $1\leq i \leq s$);
\item There are at most $2 \alpha n$ vertices in $V(G) \setminus A_i$ that are not $(\beta ' +\alpha, i)$-excellent (for each $1 \leq i \leq s$).
\end{itemize}
Suppose that $s \geq 2$.
 We now explain \emph{Stage~2} of the switching procedure.
Let $t'$ be the largest integer such that there exists both $t'$ $(\beta+\alpha,2)$-bad vertices $x_1, \dots , x_{t'} \in A_2$ {and} $t'$ $(\beta ,2)$-exceptional vertices $y_1, \dots , y_{t'} \in V(G) \setminus A_2$
at the end of Stage 1.
Note that $t' \leq 2\alpha n$. Move $y_1, \dots , y_{t'}$ into $A_2$ and remove $x_1, \dots, x_{t'}$ from $A_2$ so that each $x_i$ replaces $y_i$ in their respective classes. 

If a vertex $x \in A_2$ was  $(\beta+\alpha, 2)$-good after Stage~1 then $x$ is still $(\beta + 3\alpha, 2)$-good. Further, each $y_i$ is now $(\beta+3\alpha, 2)$-good. So if at the end of Stage 1, $A_2$ contained \emph{precisely}
$t'$ $(\beta+\alpha, 2)$-bad vertices, then $A_2$ no longer contains any $(\beta +3\alpha, 2)$-bad vertices.

If a vertex $y \in V(G) \setminus A_2$ was  $(\beta,2)$-acceptable at the end of Stage~1, then $y$ is still $(\beta- 2\alpha,2)$-acceptable. Further, each vertex $x_i$ is 
$(\beta- 2\alpha,2)$-acceptable. So if at the end of Stage~1, there were \emph{precisely}
$t'$ $(\beta,2)$-exceptional vertices, then there are now no  $(\beta -2\alpha, 2)$-exceptional vertices.

Recall that after Stage~1 there are no $(\beta +\alpha, 1)$-bad vertices in $A_1$ or there are no $(\beta -\alpha, 1)$-exceptional vertices in $V(G)\setminus A_1$.
Suppose that the former holds. Then (\ref{minex1}) implies that after Stage~1 every vertex in $A_1$ is $(\beta +\alpha,2)$-excellent. In particular, $A_1$ is not modified in Stage~2.
Next suppose that after Stage~1 there were no $(\beta -\alpha, 1)$-exceptional vertices in $V(G)\setminus A_1$. Suppose that $x$ is a vertex that lies in $V(G) \setminus A_1$ both after Stage~1 and after Stage~2.
Then after Stage~2 $x$ is  a $(\beta -3\alpha, 1)$-acceptable vertex. Suppose that $x$ is a vertex in $A_1$ after Stage~1 and a vertex in $V(G)\setminus A_1$ after Stage~2. Then $x\in A_2$ after
Stage~2 and so was a $(\beta ,2)$-exceptional vertex after Stage~1. Together with (\ref{minex1}), this implies that, after Stage~2, $x$ is   a $(\beta +2\alpha, 1)$-excellent vertex
(in particular, $x$ is not $(\beta -3\alpha, 1)$-exceptional). 
Overall this implies that, after Stage~2, there are no $(\beta +3\alpha, 1)$-bad vertices in $A_1$ or there are no $(\beta -3\alpha, 1)$-exceptional vertices in $V(G)\setminus A_1$.

Therefore, after Stage~2 we have that:
\begin{itemize}
\item $A_i$ is a $3\alpha$-independent set of size $n/r$ in $G$ (for all $1\leq i \leq s$);
\item If $B\not = \emptyset$ then $G[B]$ does not contain any $(\gamma _{s+1} -6\alpha r)$-independent set of size $n/r$;
\item There are no $(\beta +3\alpha, i)$-bad vertices in $A_i$ or there are no $(\beta -3\alpha , i)$-exceptional vertices in $V(G)\setminus A_i$ for $i=1,2$;
\item $A_i$ contains at most $4 \alpha n$ $(\beta +3\alpha, i)$-bad vertices (for all $1\leq i \leq s$);
\item There are at most $4 \alpha n$ vertices in $V(G) \setminus A_i$ that are not $(\beta ' +3\alpha, i)$-excellent (for each $1 \leq i \leq s$).
\end{itemize}

By applying an analogous switching procedure iteratively for $A_3, \dots , A_s$ we modify the partition $A_1, \dots , A_s , B$ of $V(G)$ such that the following conditions hold:
\begin{itemize}
\item[($\alpha _1$)] $A_1, \dots , A_s , B$ remains a partition of $V(G)$ so that $A_i$ is a $\sqrt{\alpha}$-independent set of size $n/r$ in $G$ (for all $1\leq i \leq s$);
\item[($\alpha _2$)] If $B\not = \emptyset$ then $G[B]$ does not contain any $(\gamma _{s+1}/2)$-independent set of size $n/r$;
\item[($\alpha _3$)] There are no $(2\beta , i)$-bad vertices  in $A_i$ or there are no $(\beta /2, i)$-exceptional vertices in $V(G) \setminus A_i$ (for all $1\leq i \leq s$);
\item[($\alpha _4$)] $A_i$ contains at most $\sqrt{ \alpha} n$ $(2\beta , i)$-bad vertices (for all $1\leq i \leq s$);
\item[($\alpha _5$)] There are at most $\sqrt{ \alpha} n$ vertices in $V(G) \setminus A_i$ that are not $(2 \beta ', i)$-excellent (for each $1 \leq i \leq s$).
\end{itemize}
Note that if $B\not = \emptyset$ then (\ref{minex1}) implies that 
\begin{align}\label{minGB}
\delta ^0 (G[B]) \geq \left( 1- \frac{1}{r-s} \right ) |B|.
\end{align}

\subsection{The case when $|B|\not = 2n/r$ or $G[B]$ is not close to $2K_{n/r}$}\label{firstcase}
In this subsection we will assume that either (i) $r-s \not = 2$ (and so $|B|\not = 2n/r$) or (ii) $r-s=2$ and $G[B]$ is not $\eta$-close to $2K_{n/r}$.
The case when $r-s=2$ and $G[B]$ is $\eta$-close to $2K_{n/r}$ is considered in Section~\ref{secondcase}.

\medskip

\noindent
{\bf Covering the exceptional vertices with matchings.}
Given any $1 \leq i \leq s$, let $V_{ex,i}$ denote the set of $(\beta /2, i)$-exceptional vertices in $V(G) \setminus A_i$. 
($\alpha _5$) implies that $c_i:=|V_{ex,i}|\leq \sqrt{\alpha }n$ for all $1 \leq i \leq s$. (\ref{minex1}) implies that
a vertex cannot be both $(\beta /2, i)$-exceptional
and $(\beta /2, j)$-exceptional for $i \not = j$. So $V_{ex,i}$ and $V_{ex,j}$ are disjoint for all $i \not = j$.

Given $1 \leq i \leq s$, if $V_{ex,i}=\emptyset$, set $G_i$ to be the empty digraph. 
If $V_{ex,i}\not =\emptyset$, set $G_i:= G[A_i \cup V_{ex,i}]$. 
If $V_{ex,i} \not = \emptyset$ then by $(\alpha _3)$, there are no $(2 \beta ,i)$-bad vertices in $A_i$. Thus, by (\ref{minex1}) every vertex in $A_i$ is
$(2 \beta ,j)$-excellent for all $j \not =i$. In particular, if $x \in A_i$ then $ x \not \in V_{ex,j}$. Therefore the digraphs $G_1, \dots , G_s$ are vertex-disjoint.

If $V_{ex,i}\not =\emptyset$ then, since $|G_i|=n/r+c_i$, (\ref{minex1}) implies that $\delta ^0 (G_i) \geq c_i$ (for $1\leq i \leq s$). 
Further, $|G_i|\geq 2c_i$.
So Proposition~\ref{simplematch2} implies that there are disjoint matchings $M_1, \dots , M_s$ in $G$ such that:
\begin{itemize}
\item[($\beta _1$)] $M_i$ is a $c_i$-matching in $G_i$ that covers all the vertices in $V_{ex,i}$ (for all $1 \leq i \leq s$).
\end{itemize}
Note that if $V_{ex,i} =\emptyset$ then $M_i$ is empty.

\medskip

\noindent
{\bf Extending the matchings $M_i$ into $T$-packings.} Our next task is to find vertex-disjoint $T$-packings $\mathcal M_1, \dots , \mathcal M_s$ in $G$ so that, for each $1 \leq i \leq s$:
\begin{itemize}
\item[($\gamma _1$)] $\mathcal M_i$ contains precisely $c_i$ disjoint copies of $T$;
\item[($\gamma _2$)] $\mathcal M_i$ covers $M_i$. In particular, each copy of $T$ in $\mathcal M_i$ contains a unique edge from $M_i$;
\item[($\gamma _3$)] $\mathcal M_i$ covers precisely $c_i$ vertices from $A_j$ (for each $1\leq j \leq  s$) and precisely $(r-s)c_i$ vertices from $B$.
\end{itemize}
Suppose that for some $1 \leq i \leq s$ we have found our desired $T$-packings $\mathcal M_1 , \dots , \mathcal M_{i-1}$. We now construct $\mathcal M_i$.
If $M_i$ is empty then we set $\mathcal M_i =\emptyset$ and then ($\gamma _1$)--($\gamma _3$) are vacuously true for $\mathcal M_i$.
So suppose that $M_i$ is non-empty. Since $|V_{ex,i}|=c_i$, ($\beta _1$) implies that there exist non-negative integers $a_i, b_i$ so that $c_i=a_i+2b_i$ and 
\begin{itemize}
\item[(i)] $M_i$ contains precisely $a_i$ edges with one endpoint in $V_{ex,i}$ and the other in $A_i$;
\item[(ii)] $M_i$ contains precisely $b_i$ edges with both endpoints in $V_{ex,i}$;
\item[(iii)] $M_i$ contains precisely $b_i$ edges with both endpoints in $A_i$.
\end{itemize}

Consider any edge $e$ in $M_i$ with one endpoint $x \in A_i$ and one endpoint $y \in V_{ex,i}$.
Since $V_{ex,i} \not = \emptyset$, ($\alpha _3$) implies that $x$ is $(2 \beta, i)$-good. In particular, together with (\ref{minex1}) this implies that 
$d^+ (x,A_j),d^-(x,A_j) \geq |A_j|-2 \beta n$ for all $j \not = i$ and $d^+ (x,B),d^-(x,B) \geq |B|-2 \beta n$. Since $y$ is $(\beta/2,i)$-exceptional, (\ref{minex1}) implies that
$d^+ (y,A_j),d^-(y,A_j) \geq |A_j|-\beta n /2$ for all $j \not = i$ and $d^+ (y,B),d^-(y,B) \geq |B|- \beta n/2$.
It is easy to see that, together with (\ref{minGB}), ($\alpha_4$) and ($\alpha _5$), this implies that we can greedily construct a copy $T'$ of $T$ in $G$ such that:
\begin{itemize}
\item $T'$ is vertex-disjoint from $\mathcal M_1 , \dots , \mathcal M_{i-1}$ and $M_{i}\setminus \{e\},M_{i+1}, \dots , M_s$;
\item $T'$ contains $e$ and contains precisely one vertex from each of $A_1. \dots, A_s$ and $r-s$ vertices from $B$.
\end{itemize}
Further, we can repeat this process for all $a_i$ such edges $e$ so that the $a_i$ copies of $T$ thus obtained are vertex-disjoint.
Let $\mathcal M'_i$ denote the set of these copies of $T$. So $\mathcal M'_i$ covers $a_i$ vertices from each $A_j$ and $a_i(r-s)$ vertices from $B$.

Next pair off each of the $b_i$ edges from (ii) with a unique edge from (iii). Consider one such pair $(e,e')$ of edges. So the endpoints $x,y$ of $e$ lie in $V_{ex,i}$ and the endpoints
$x',y'$ of $e'$ lie in $A_i$. 
Suppose that $x,y \in A_{i'}$ for some $i' \not = i$. (The other cases are similar.)
Since $x,y$ are $(\beta/2,i)$-exceptional, (\ref{minex1})  implies that 
$d^+ (x,A_j),d^-(x,A_j), d^+ (y,A_j),d^-(y,A_j) \geq |A_j|- \beta n/2$ for all $j \not = i$ and $d^+ (x,B),d^-(x,B) , d^+ (y,B),$ $d^-(y,B)\geq |B|- \beta n/2$. 
It is easy to see that, together with (\ref{minGB}), ($\alpha_4$) and ($\alpha _5$), this implies that we can greedily construct a copy $T'$ of $T$ in $G$ such that:
\begin{itemize}
\item $T'$ is vertex-disjoint from $\mathcal M_1 , \dots , \mathcal M_{i-1}, \mathcal M'_i$ and $M_{i}\setminus \{e\},M_{i+1}, \dots ,M_s$;
\item $T'$ contains $e$ and contains two vertices from $A_{i'}$ (namely $x$ and $y$),  no vertices from $A_i$, one vertex from $A_j$ (for $j \not = i,i'$) and $r-s$ vertices from $B$.
\end{itemize}
Similarly 
we can greedily construct a copy $T''$ of $T$ in $G$ such that:
\begin{itemize}
\item $T''$ is vertex-disjoint from $\mathcal M_1 , \dots , \mathcal M_{i-1}, \mathcal M'_i$, $T'$ and $M_{i}\setminus \{e'\},M_{i+1}, \dots  ,M_s$;
\item $T''$ contains $e'$ and contains two vertices from $A_{i}$ (namely $x'$ and $y'$),  no vertices from $A_{i'}$, one vertex from $A_j$ (for $j \not = i,i'$) and $r-s$ vertices from $B$.
\end{itemize}
So together $T'$ and $T''$ cover precisely two vertices from each $A_j$ and $2(r-s)$ vertices from $B$.
Further, we can repeat this process for all  such pairs of edges $(e,e')$ so that the $2b_i$ copies of $T$ thus obtained are vertex-disjoint.
Let $\mathcal M''_i$ denote the set of these copies of $T$.
Then by construction $\mathcal M''_i$ covers precisely $2b_i$ vertices from each $A_j$ and $2b_i(r-s)$ vertices from $B$.
Set $\mathcal M_i:= \mathcal M'_i \cup \mathcal M''_i$. 
So $\mathcal M_i$ is a $T$-packing in $G$.
By construction $\mathcal M_i$ is vertex-disjoint from $\mathcal M_1, \dots , \mathcal M_{i-1}$ and satisfies ($\gamma _1$)--($\gamma _3$), as desired.

\medskip

\noindent
{\bf Covering the remaining vertices.} Remove all those vertices covered by $\mathcal M_1 , \dots , \mathcal M_s$ from $G$ (and from the classes $A_1, \dots, A_s, B$). Call the resulting digraph $G'$.
So $n':=|G'|\geq (1- r^2 \sqrt{\alpha} )n$ by ($\gamma _1$), $|A_i|=n'/r$ for all $1\leq i \leq s$ and $|B|=(r-s)n'/r$ by ($\gamma _3$).
Further, (\ref{minex1}) and ($\alpha _1$)--($\alpha _5$) imply that the following conditions hold:
\begin{itemize}
\item[($\delta _1$)] $\delta ^0 (G') \geq (1-1/r-r^2 \sqrt{\alpha} )n \geq (1-1/r- r^2 \sqrt{\alpha}) n'$; 
\item[($\delta _2$)] $A_1, \dots , A_s , B$ is a partition of $V(G')$ so that $A_i$ is a $2\sqrt{\alpha}$-independent set of size $n'/r$ in $G'$ (for all $1\leq i \leq s$);
\item[($\delta _3$)] If $B\not = \emptyset$ then $G'[B]$ does not contain any $(\gamma _{s+1}/3)$-independent set of size $n'/r$;
\item[($\delta _4$)] Every vertex in $V(G') \setminus A_i$ is $(\beta /3,i)$-acceptable (for each $1 \leq i \leq s$);
\item[($\delta _5$)] There are at most $\sqrt{ \alpha} n$ vertices in $V(G') \setminus A_i$ that are not $(2 \beta ', i)$-excellent (for each $1 \leq i \leq s$).
\end{itemize}
In particular, note that ($\delta _4$) follows since $\mathcal M_i$ contains the vertices in $V_{ex,i}$. If $B\not = \emptyset$ then
($\delta _1$) implies that 
\begin{align}\label{Bboundy}
d^+ _{G'} (y,B) , d^- _{G'} (y,B) \geq \left (1-\frac{1}{r} -r^2 \sqrt{\alpha} \right )n'-\frac{sn'}{r} \geq \left (1- \frac{1}{r-s} -\alpha ^{1/3} \right) |B|
\end{align}
for all $y \in V(G')$.

We now split the proof into cases depending on the size of $B$.

\smallskip

\noindent
{\bf Case 1: $B=\emptyset$.}  In this case $s=r$ and $A_1, \dots , A_r$ is a partition of $V(G')$.
Then $G'$ contains a $T$-packing $\mathcal M'$ such that:
\begin{itemize}
\item[($\eps _1$)] $\mathcal M'$ contains at most $r \sqrt{\alpha} n$ copies of $T$;
\item[($\eps _2$)] If $x \in V(G') \setminus A_i$ is not $(2\beta ' ,i)$-excellent  then $x$ is contained in a copy of $T$ in $\mathcal M'$ (for any $1 \leq  i \leq r$);
\item[($\eps _3$)] Each copy of $T$ in $\mathcal M'$ covers exactly one vertex from $A_i$ (for each $1 \leq i \leq r$).
\end{itemize}
To see that such a $T$-packing $\mathcal M'$ exists, suppose that we have found a $T$-packing $\mathcal M^*$ in $G'$ that satisfies ($\eps _1$) and ($\eps _3$).
Suppose that for some $1\leq i \leq r,$ $x \in V(G) \setminus A_i$ so that $x$ is not $(2\beta ' ,i)$-excellent and $x$ is not covered by $\mathcal M^*$. (By ($\delta _5$) there are at most $r \sqrt{\alpha} n$ such vertices $x$.) Without loss of generality assume that $x \in A_1$. Then by ($\delta _1$) and ($\delta _4$)  there exist $2 \leq i'\not = i'' \leq r$ such that:
\begin{itemize}
\item $d^+ _{G'} (x,A_{i'} ) \geq \beta n/3$;
\item $d^- _{G'} (x,A_{i''} ) \geq \beta n/3$;
\item $d^+_{G'} (x,A_j), d^- _{G'} (x,A_j) \geq \beta n/3$ for all $2 \leq j \leq r$ such that $j \not = i',i''$.
\end{itemize}
Without loss of generality assume that $i'=2$ and $i''=3$.
Write $V(T)=\{x_1, \dots , x_r\}$ where $x_1x_2 ,x_3 x_1 \in E(T)$. 
Since $d^+ _{G'} (x,A_{2} ) \geq \beta n/3$, ($\delta _5$) implies that there is a vertex $y_2 \in A_2 \setminus V(\mathcal M^*)$ such that $xy_2 \in E(G')$ and 
$y_2$ is $(2 \beta ', i)$-excellent for all $3 \leq i \leq r$.
Further, since $d^- _{G'} (x,A_{3} ) \geq \beta n/3$, the choice of $y_2$ together with ($\delta _5$) ensures that there is a vertex 
$y_3 \in A_3 \setminus V(\mathcal M^*)$ such that $y_3x, y_2y_3,y_3y_2 \in E(G')$ and 
$y_3$ is $(2 \beta ', i)$-excellent for all $4 \leq i \leq r$. 
In particular, $\{x,y_2,y_3\}$ spans a copy of $T[x_1,x_2,x_3]$ in $G'$ that is vertex-disjoint from $\mathcal M^*$.
Continuing in this way we obtain vertices $y_2, \dots , y_r$ such that $y_i \in A_i \setminus V(\mathcal M^*)$ and $\{x,y_2, \dots , y_r\}$ spans a copy of $T$ in $G'$ where 
$x,y_2, \dots , y_r$ play the roles of $x_1, \dots ,x_r$ respectively. 
This argument shows that we can indeed find a $T$-packing $\mathcal M'$ that satisfies ($\eps _1$)--($\eps _3$).

Remove all those vertices covered by $\mathcal M'$ from $G'$ (and from the classes $A_1, \dots, A_r$). Call the resulting digraph $G''$.
So $n'':=|G''|\geq (1- 2r^2 \sqrt{\alpha} )n$ by ($\eps _1$) and  $|A_i|=n''/r$ for all $1\leq i \leq r$ by ($\eps _3$).
Further, ($\eps _2$) implies that, given any $x \in V(G'') \setminus A_i$, $x$ is $(2 \beta ',i)$-excellent (for all $ 1 \leq i \leq r$).
So Theorem~\ref{rpart2} implies that $G''$ contains a perfect $K_r$-packing and thus a perfect $T$-packing $\mathcal M''$.
Set $\mathcal M:=\mathcal M' \cup \mathcal M'' \cup \mathcal M_1 \cup \dots \cup \mathcal M_r$. Then $\mathcal M$ is a perfect $T$-packing in $G$, as required.

\medskip

\noindent
{\bf Case 2: $B\not =\emptyset$.}
In this case $s\leq r-1$. Given any $1 \leq i \leq s$, ($\delta _1$) and ($\delta _2$) imply that there are at least
$$2|B||A_i|-r^2\sqrt{\alpha} n^2 -4\sqrt{\alpha} n^2 \geq 2|B||A_i|-\alpha ^{1/3} n^2$$
edges in $G'$ with one endpoint in $A_i$ and the other endpoint in $B$. 
Since $\alpha \ll \beta \ll 1/r$, this implies that there are at most $\alpha ^{1/4} n /r$ vertices in $A_i$ that are not $(\beta ,B)$-excellent.
Let $V_{ex,B}$ denote the set of all those vertices in $V(G')\setminus B$ that are not $(\beta,B)$-excellent. So $|V_{ex,B}|\leq \alpha ^{1/4} n$.

Then $G'$ contains a $T$-packing $\mathcal M'$ such that:
\begin{itemize}
\item[($\eps ' _1$)] $\mathcal M'$ contains   $m' \leq 2 r \sqrt{\alpha} n + 2 \alpha ^{1/4} n \leq 3 \alpha ^{1/4} n$ copies of $T$;
\item[($\eps ' _2$)] If $x \in V(G') \setminus A_i$ is not $(\beta  ,i)$-excellent  then $x$ is contained in a copy of $T$ in $\mathcal M'$ (for any $1 \leq  i \leq s$).
Similarly, if $x \in V(G') \setminus B$ is not $(\beta  ,B)$-excellent  then $x$ is contained in a copy of $T$ in $\mathcal M'$;
\item[($\eps ' _3$)] $\mathcal M'$ covers exactly $m'$ vertices from $A_i$ (for each $1 \leq i \leq s$) and $m'(r-s)$ vertices from $B$.
\end{itemize}
To prove that such a $T$-packing $\mathcal M'$ exists, we will use the follow three claims.

\begin{claim}\label{claima}
Let $x \in V(G') \setminus B$ be such that $x$ is not $(\beta  ,B)$-excellent and let $W \subseteq V(G') \setminus \{x\}$ where $|W| \leq \alpha ^{1/5} n$.
Then there are two vertex-disjoint copies $T',T''$ of $T$ in $G'$ so that:
\begin{itemize}
\item[(i)] $V(T') \cup V(T'')$ contains two vertices from $A_i$ (for each $1 \leq i \leq s$) and $2(r-s)$ vertices from $B$;
\item[(ii)] $x \in V(T') \cup V(T'')$ and $V(T') \cup V(T'')$  is disjoint from $W$.
\end{itemize}
\end{claim}
\proof
To prove the claim consider a vertex  $x \in V(G') \setminus B$ that is not $(\beta  ,B)$-excellent.
If $s=1$ then $x \in A_1$. Further, since $x$ is not $(\beta  ,B)$-excellent, 
($\delta _1$) implies that 
\begin{itemize}
\item $d^+ _{G'} (x,A_{1} ) \geq \beta n-r^2 \sqrt{\alpha} n' \geq \beta n/2 $ or $d^- _{G'} (x,A_{1} ) \geq \beta n/2$.
\end{itemize}
Without loss of generality assume that $d^+ _{G'} (x,A_{1} ) \geq \beta n/2$.

Fix a vertex $y$ in $A_1$ such that 
\begin{itemize}
\item $xy \in E(G')$;
\item $y$ is $(\beta,B)$-excellent;
\item $y \not \in W$.
\end{itemize}
Note that there are at least $\beta n/2 -\alpha ^{1/4} n -\alpha ^{1/5} n \geq \beta n/4$ choices for $y$ since $|V_{ex,B}|\leq \alpha ^{1/4} n$ and $|W|\leq \alpha ^{1/5} n$. 
Then by repeatedly applying (\ref{Bboundy}) we can greedily extend $xy$ to a copy $T'$ of $T$ in $G'$ containing two vertices from $A_1$ (namely $x $ and $y$) and $r-2$ vertices from $B$ so that
$T'$ is disjoint from $W$.

Next suppose that $s\geq 2$. Since $x$ is not $(\beta  ,B)$-excellent,
($\delta _1$) and ($\delta _4$) imply that there exist $1 \leq i'\not = i'' \leq s$ such that:
\begin{itemize}
\item $d^+ _{G'} (x,A_{i'} ) \geq \beta n/3$;
\item $d^- _{G'} (x,A_{i''} ) \geq \beta n/3$;
\item $d^+_{G'} (x,A_j), d^- _{G'} (x,A_j) \geq \beta n/3$ for all $1 \leq j \leq r$ such that $j \not = i',i''$.
\end{itemize}
Without loss of generality assume that $x \in A_1$,  $i'=1$ and $i''=2$ (the other cases are similar).

Write $V(T)=\{x_1, \dots , x_r\}$ where $x_1x_2 ,x_3 x_1 \in E(T)$. 
Since $d^+ _{G'} (x,A_{1} ) \geq \beta n/3$, $|V_{ex,B}|\leq \alpha ^{1/4} n$ and $|W|\leq \alpha ^{1/5} n$, ($\delta _5$) implies that there is a vertex $y_2 \in A_1 $ such that:
\begin{itemize}
\item $xy_2 \in E(G')$;
\item $y_2$ is $(2 \beta ', i)$-excellent for all $2 \leq i \leq s$; 
\item $y_2$ is $(\beta ,B)$-excellent;
\item $y_2 \not \in W$.
\end{itemize}
Since $d^- _{G'} (x,A_{2} ) \geq \beta n/3$, the choice of $y_2$ together with ($\delta _5$) ensures that there is a vertex 
$y_3 \in A_2 $ such that:
\begin{itemize}
\item  $y_3x, y_2y_3,y_3y_2 \in E(G')$;
\item $y_3$ is $(2 \beta ', i)$-excellent for all $3\leq i \leq s$;
\item $y_3$ is $(\beta ,B)$-excellent;
\item $y_3 \not \in W$.
\end{itemize}
In particular, $\{x,y_2,y_3\}$ spans a copy of $T[x_1,x_2,x_3]$ in $G'$.
Continuing in this fashion and then repeatedly applying (\ref{Bboundy}) we can greedily find a copy $T'$ of $T$ in $G'$ that covers 
two vertices in $A_1$ (namely $x$ and $y_2$), one vertex from $A_j$ (for $2 \leq j \leq s$) and $r-s-1$ vertices from $B$ so that $T'$ is disjoint from $W$.
So in both cases we have found a copy $T'$ of $T$ in $G'$ that covers two vertices in $A_1$ (including $x$), one vertex from $A_j$ (for $2 \leq j \leq s$) and $r-s-1$ vertices from $B$.

Let $T^*$ be a subtournament of $T$ of size $r-s+1$.
Let $B'$ denote the set of vertices $x \in B \setminus (V(T')\cup W)$ that are $(2 \beta ',j)$-excellent for all $1 \leq j \leq s$. Then $|B'|\geq |B|-r\sqrt{\alpha } n -r-\alpha^{1/5}n $ by ($\delta _5$).
Together with~(\ref{Bboundy}) this implies that
$$\delta ^0 (G'[B'])\geq \left ( 1- \frac{1}{r-s} -\alpha ^{1/6}\right ) |B'|.$$
Moreover, ($\delta _3$) implies that $G'[B']$ does not contain any $(\gamma _{s+1}/4)$-independent set of size $|B'|/(r-s)$. 
Proposition~\ref{turanstab} (with $G'[B']$, $|B'|$, $\alpha ^{1/6}$, $T^*$, $r-s+1$ playing the roles of $G$, $n$, $\eps$, $T$, $r$ respectively) implies that  
$G'[B']$ contains a copy $T_1 ^*$ of $T^*$. The choice of $B'$ ensures that we can greedily extend $T^* _1$ to a copy $T''$ of $T$ in $G'$ that is disjoint from $V(T')\cup W$ and that covers
no vertices from $A_1$, one vertex from $A_j$ (for $2 \leq j \leq s$) and $r-s+1$ vertices from $B$.
So together $T'$ and $T''$ satisfy (i) and (ii). This completes the proof of Claim~\ref{claima}.
\endproof

\begin{claim}\label{claimb}
Let $x \in V(G') \setminus (A_i \cup B)$ be such that $x$ is not $(\beta , i)$-excellent (for some $1 \leq i \leq s$) and let $W \subseteq V(G')\setminus \{x\}$ where $|W| \leq \alpha ^{1/5} n$.
Then there are two vertex-disjoint copies $T',T''$ of $T$ in $G'$ so that:
\begin{itemize}
\item[(i)] $V(T') \cup V(T'')$ contains two vertices from $A_j$ (for each $1 \leq j \leq s$) and $2(r-s)$ vertices from $B$;
\item[(ii)] $x \in V(T') \cup V(T'')$ and $V(T') \cup V(T'')$  is disjoint from $W$.
\end{itemize}
\end{claim}
The proof of Claim~\ref{claimb} is essentially identical to the proof of Claim~\ref{claima}, so we omit it.

\begin{claim}\label{claimc}
Let $x \in B$ be such that $x$ is not $(\beta , i)$-excellent (for some $1 \leq i \leq s$) and let $W \subseteq V(G')\setminus \{x\}$ where $|W| \leq \alpha ^{1/5} n$.
Then there is a copy $T'$ of $T$ in $G'$ so that:
\begin{itemize}
\item[(i)] $V(T')$ contains one vertex from $A_j$ (for each $1 \leq j \leq s$) and $r-s$ vertices from $B$;
\item[(ii)] $x \in V(T') $ and $V(T')$  is disjoint from $W$.
\end{itemize}
\end{claim}
It is easy to see that ($\delta _1$), ($\delta _4$) and (\ref{Bboundy}) imply that we can greedily construct a copy
$T'$ of $T$ as in Claim~\ref{claimc}.

\medskip

Recall that $|V_{ex,B}|\leq \alpha ^{1/4} n$. 
Together with ($\delta _5$)  this implies that we can repeatedly apply Claims~\ref{claima}--\ref{claimc} to obtain  a $T$-packing $\mathcal M'$ in $G'$ satisfying ($\eps '_1$)--($\eps '_3$).

Remove all those vertices covered by $\mathcal M'$ from $G'$ (and from the classes $A_1, \dots, A_s, B$). Call the resulting digraph $G''$.
So $n'':=|G''|\geq (1- {\alpha} ^{1/5})n$ by ($\eps' _1$) and  $|A_i|=n''/r$ for all $1\leq i \leq s$ and $|B|=(r-s)n''/r$ by ($\eps '_3$).
Further, ($\eps '_2$) implies that, given any $x \in V(G'') \setminus A_i$, $x$ is $( \beta ,i)$-excellent (for all $ 1 \leq i \leq s$)
and every vertex $y \in V(G'') \setminus B$ is $(\beta, B)$-excellent. 

Suppose that $|B|=n''/r$. Then as in Case~1,
 Theorem~\ref{rpart2} implies that $G''$ contains a perfect $K_r$-packing and thus a perfect $T$-packing $\mathcal M''$.
Set $\mathcal M:=\mathcal M' \cup \mathcal M'' \cup \mathcal M_1 \cup \dots \cup \mathcal M_s$. Then $\mathcal M$ is a perfect $T$-packing in $G$, as required.

Next suppose that $|B|\geq 2n''/r$. Let $T^*$ be a subtournament of $T$ on $r-s$ vertices such that $T^* \not = C_3$. (Note that if $r-s=3$ then $r \geq 4$. Every tournament on at least four vertices
contains $T_3$, so we indeed may choose $T^* \not = C_3$.) By (\ref{Bboundy}) and (${\eps '_1}$) we have that 
$$\delta ^0 (G''[B]) \geq (1-1/(r-s) -\alpha ^{1/5} ) |B|.$$
Moreover, ($\delta _3$) implies that $G''[B]$ does not contain any $(\gamma _{s+1}/4)$-independent set of size $n''/r=|B|/(r-s)$.
Further, if $|B|=2n''/r$ (and so $r-s=2$) then by assumption $G''[B]$ is not $\eta /2$-close to $2K_{n''/r}$.
Thus, by Theorem~\ref{1nonex} and Proposition~\ref{matchstab2}, $G''[B]$ contains a perfect $T^*$-packing $\mathcal M^*$.

\smallskip

Define an auxiliary digraph $G^*$ from $G''$ as follows. $G^*$ has vertex set $A_1 \cup \dots \cup A_s \cup B^*$ where $|B^*|=n''/r$ and each vertex $x \in B^*$ corresponds to a unique copy $T^* _x$ of $T^*$
from $\mathcal M^*$. The edge set of $G^*$ consists of every edge $wz \in E(G'')$ such that $w \in A_i$ and $z \in A_j$ for some $i \not =j$ together with the following edges:
Suppose that $ x \in B^*$ and $y \in V(G^*) \setminus B^*$. Then
\begin{itemize}
\item $yx \in E(G^*)$ precisely if $y$ sends an edge to every vertex in $T^* _x$ in $G''$;
\item $xy \in E(G^*)$ precisely if $y$ receives an edge from every vertex in $T^* _x$ in $G''$.
\end{itemize}
Note that $G^*$ is an $(s+1)$-partite digraph with vertex classes of size $n''/r$. Further,
$$\bar{\delta} ^+ (G^*), \bar{\delta }^- (G^*) \geq n''/r - \beta r n.$$
So by Theorem~\ref{rpart2}, $G^*$ contains a perfect $K_{s+1}$-packing. By construction of $G^*$ this implies that $G''$ contains a perfect $T$-packing $\mathcal M''$.
Set $\mathcal M:=\mathcal M' \cup \mathcal M'' \cup \mathcal M_1 \cup \dots \cup \mathcal M_s$. Then $\mathcal M$ is a perfect $T$-packing in $G$, as required.

\subsection{The case when $|B| = 2n/r$ and $G[B]$ is close to $2K_{n/r}$}\label{secondcase} In this subsection we consider the case when $|B|=2n/r$ and
 $G[B]$ is $\eta$-close to $2K_{n/r}$. Thus, there exists a partition $B_1, B_2$ of $B$ such that $|B_1|=|B_2|=n/r$ and $e_G (B_1,B_2) \leq \eta |B|^2$.
(\ref{minGB}) implies that $\delta ^0 (G[B])\geq |B|/2=n/r$.

For $i=1,2$ and $\delta >0$ we say that a vertex $x \in B_i$ is \emph{$(\delta, B_i)$-excellent} if $d^+ _G (x,B_i) , d^- _G (x,B_i) \geq |B_i| - \delta n$.
(Later on we will modify the classes $B_1, B_2$. When referring to, for example,  $(\delta,B_1)$-excellent vertices, we mean with respect to the current class $B_1$ and not the original class.)
Note that there are at most $\eta ^{1/2} |B|$ vertices $x \in B_i$ that are not $(\eta ^{1/2} ,B_i)$-excellent for $i=1,2$.

Let $V_{ex} ^1$ denote the set of vertices $x \in B_1$ that are not $(\eta ^{1/2}, B_1)$-excellent. Define $V^2 _{ex}$ analogously. Given a vertex $x \in V^1 _{ex}$, if $d^+ _G (x,B_2)
\geq n/2r$ then  move $x$ into $B_2$. Similarly, if $x \in V^2 _{ex}$ and $d^+ _G (x,B_1)
\geq n/2r$ then  move $x$ into $B_1$. Thus, the following conditions hold:
\begin{itemize}
\item[($\zeta _1$)] $n/r - \eta ^{1/2} n \leq |B_1|,|B_2| \leq n/r +\eta ^{1/2}n$;
 \item[($\zeta _2$)] $e_G (B_1,B_2) \leq  5 \eta ^{1/2} |B|^2$;
\item[($\zeta _3$)] There are at most $2\eta ^{1/2} |B|$ vertices $x \in B_i$ that are not $(2\eta ^{1/2} ,B_i)$-excellent (for $i=1,2$);
\item[($\zeta _4$)] Given any $x \in B_i$, $d^+ _G (x, B_i) \geq n/3r$ (for $i=1,2$).
\end{itemize}
Actually, there is slack in conditions ($\zeta _1$) and ($\zeta _2$). Indeed, if we move a single vertex from $B_2$ to $B_1$ (or vice versa) then ($\zeta _1$) and ($\zeta _2$) still hold.

Since $|B|=2n/r$ is even, either $|B_1|$ and $|B_2|$ are even or $|B_1|$ and $|B_2|$ are odd. Suppose that $|B_1|$ and $|B_2|$ are odd. Without loss of generality assume that $|B_2|\geq n/r$.
Fix a vertex $b_1 \in B_1$ such that
\begin{itemize}
\item[(i)] $b_1$ is $(2 \beta ' ,i)$-excellent for all $1 \leq i \leq s$.
\end{itemize}
(Such a vertex $b_1$ exists by ($\alpha _5$).) 
Then by (\ref{minex1}) there is a vertex $b_2 \in B_2$ such that $b_1b_2 \in E(G)$. Further, (\ref{minex1}) implies that 
\begin{itemize}
\item[(ii)] $d^+ _G (b_2,A_i), d^- _G (b_2,A_i) \geq \eta n$ for all $1 \leq i \leq s$ or;
\item[(iii)]  $d^+ _G (b_2,B_1) \geq n/3r$ or $d^- _G (b_2,B_1) \geq n/3r$.
\end{itemize}
If (iii) holds then move $b_2$ into $B_1$. Otherwise, we leave the partition $B_1,B_2$ of $B$ unchanged. 
Thus, the following conditions hold:
\begin{itemize}
\item[($\eta _1$)] $n/r - \eta ^{1/2} n \leq |B_1|,|B_2| \leq n/r +\eta ^{1/2}n$;
 \item[($\eta _2$)] $e_G (B_1,B_2) \leq  5 \eta ^{1/2} |B|^2$;
\item[($\eta _3$)] There are at most $3\eta ^{1/2} |B|$ vertices $x \in B_i$ that are not $(3\eta ^{1/2} ,B_i)$-excellent (for $i=1,2$);
\item[($\eta _4$)] Given any $x \in B_i$, $d^+ _G (x, B_i) \geq n/4r$ or $d^- _G (x, B_i) \geq n/4r$ (for $i=1,2$).
\end{itemize}
Additionally, one of the following conditions holds:
\begin{itemize}
\item[($\eta _5$)] $|B_1|$ and $|B_2|$ are even or;
 \item[($\eta _6$)] $|B_1|$ and $|B_2|$ are odd. Further, there exist $b_1 \in B_1$ and $b_2 \in B_2$ such that
\begin{itemize}
\item $b_1b_2 \in E(G)$;
\item $b_1$ is $(2 \beta ' ,i)$-excellent for all $1 \leq i \leq s$;
\item $d^+ _G (b_2,A_i), d^- _G (b_2,A_i) \geq \eta n$ for all $1 \leq i \leq s$.
\end{itemize}
\end{itemize}
If ($\eta _6$) holds then we will extend the edge $b_1b_2$ into a copy of $T$ in $G$. First though we will cover the `exceptional vertices' in $G$ with $T$-packings.

\medskip

\noindent
{\bf Covering the exceptional vertices with $T$-packings.}
Given any $1 \leq i \leq s$, let $V_{ex,i}$ denote the set of $(\beta /2, i)$-exceptional vertices in $V(G) \setminus A_i$. (Note that if ($\eta _6$) holds then $b_1, b_2 \not \in V_{ex,i}$.)
($\alpha _5$) implies that $c_i:=|V_{ex,i}|\leq \sqrt{\alpha }n$ for all $1 \leq i \leq s$.
Then there exist vertex-disjoint $T$-packings $\mathcal M_1, \dots , \mathcal M_s$ in $G$ so that, for each $1 \leq i \leq s$:
\begin{itemize}
\item[($\theta _1$)] $\mathcal M_i$ contains precisely $c_i$ disjoint copies of $T$;
\item[($\theta _2$)] Each vertex from $V_{ex,i}$ lies in a copy of $T$ in $\mathcal M_i$;
\item[($\theta _3$)] $\mathcal M_i$ covers precisely $c_i$ vertices from $A_j$ (for each $1\leq j \leq  s$) and precisely $2c_i$ vertices from $B$;
\item[($\theta _4$)]  $\mathcal M_i$ covers an even number of vertices from $B_1$ and an even number of vertices from $B_2$;
\item[($\theta _5$)] If ($\eta _6$) holds then $\mathcal M_i$ does not cover $b_1$ or $b_2$.
\end{itemize}
Note that the same argument used to construct $\mathcal M_1, \dots , \mathcal M_s$
 in Section~\ref{firstcase} shows that we can construct $\mathcal M_1, \dots , \mathcal M_s$ here so that ($\theta _1$)--($\theta _3$) hold. It is not difficult to see that we can additionally ensure that
($\theta _4$) and ($\theta _5$) hold.
\COMMENT{To see why ($\theta _4$) holds: If one of the $a_i$ edges with one endpoint in $V_{ex,i}$ and one in $A_i$. Then $T'$ chosen so that two vertices from $B_1$ or from $B_2$, one from the $A_j$.
For the pairs $(e.e')$:
\begin{itemize}
\item If $x,y \in A_{i'}$: $T'$ contains 2 vertices from $B_1$, 2 from $A_{i'}$, 0 from $A_i$ and 1 from rest.
$T''$ contains 2 from $B_2$, $2$ from $A_i$, 0 from $A_{i'}$, 1 from rest.
\item If $x \in A_{i'}$, $y \in A_{j'}$: as above.
\item If $x \in A_{i'}$, $y \in B$: as above.
\item If $x , y \in B_1$: $T'$ contains 3 vertices from $B_1$, 1 from $A_{j}$ $j \not = i$, 0 from $A_i$ .
$T''$ contains 1 from $B_1$, $2$ from $A_i$,  1 from rest of $A_j$.
\item If $x \in B_1$, $y \in B_2$: $T'$ contains 2 vertices from $B_1$ 1 from $B_2$, 1 from $A_{j}$ $j \not = i$, 0 from $A_i$ .
$T''$ contains 1 from $B_2$, $2$ from $A_i$,  1 from rest of $A_j$.
\end{itemize}}

\medskip

\noindent
{\bf Extending the edge $b_1b_2$ to a copy of $T$.}
If $|B_1|$ and $|B_2|$ are even set $\mathcal T := \emptyset$. Otherwise, ($\eta _6$) holds. In this case, we can greedily construct a copy $T'$ of $T$ in $G$ such that:
\begin{itemize}
\item $T'$ is vertex-disjoint from $\mathcal M_1, \dots , \mathcal M_s$;
\item $T'$ contains $b_1b_2$ (and so one vertex from each of $B_1$ and $B_2$) and precisely one vertex from each of $A_1, \dots , A_s$.
\end{itemize}
Set $\mathcal T :=\{T'\}$. 

\medskip

\noindent
{\bf Covering the remaining vertices.} Remove all those vertices covered by $\mathcal M_1 , \dots , \mathcal M_s, \mathcal T$ from $G$ (and from the classes $A_1, \dots, A_s, B$ and from $B_1, B_2$). Call the resulting digraph $G'$.
So $n':=|G'|\geq (1- 2r^2 \sqrt{\alpha} )n$ by ($\theta _1$), $|A_i|=n'/r$ for all $1\leq i \leq s$ and $|B|=2n'/r$ by ($\theta _3$) and the choice of $\mathcal T$.

Further, (\ref{minex1}) and ($\alpha _1$)--($\alpha _5$)  imply that the following conditions hold:
\begin{itemize}
\item[($\iota _1$)] $\delta ^0 (G') \geq (1-1/r-2r^2 \sqrt{\alpha} )n \geq (1-1/r- 2r^2 \sqrt{\alpha}) n'$; 
\item[($\iota  _2$)] $A_1, \dots , A_s , B$ is a partition of $V(G')$ so that $A_i$ is a $2\sqrt{\alpha}$-independent set of size $n'/r$ in $G'$ (for all $1\leq i \leq s$);
\item[($\iota  _3$)] Every vertex in $V(G') \setminus A_i$ is $(\beta /3,i)$-acceptable (for each $1 \leq i \leq s$);
\item[($\iota  _4$)] There are at most $\sqrt{ \alpha} n$ vertices in $V(G') \setminus A_i$ that are not $(2 \beta ', i)$-excellent (for each $1 \leq i \leq s$).
\end{itemize}
In particular, note that ($\iota _3$) follows from ($\theta _2$). 
Further, ($\eta _1$)--($\eta _6$) and ($\theta _1$)--($\theta _5$) together with the choice of $\mathcal T$ implies that:
\begin{itemize}
\item[($\kappa _1$)] $n/r - 2\eta ^{1/2} n \leq |B_1|,|B_2| \leq n/r +\eta ^{1/2}n$ and $|B_1|$ and $|B_2|$ are even;
\item[($\kappa _2$)] There are at most $4\eta ^{1/2} |B|$ vertices $x \in B_i$ that are not $(3\eta ^{1/2} ,B_i)$-excellent (for $i=1,2$);
\item[($\kappa _3$)] Given any $x \in B_i$, $d^+ _{G'} (x, B_i) \geq n/5r$ or $d^- _{G'} (x, B_i) \geq n/5r$ (for $i=1,2$).
\end{itemize}
Note that ($\iota _1$) implies that
$$
d^+ _{G'} (y,B) , d^- _{G'} (y,B) \geq \left (\frac{1}{2} -\alpha ^{1/3} \right) |B|
$$
for all $y \in V(G')$.

Let $V_{ex,B}$ denote the set of all those vertices in $V(G')\setminus B$ that are not $(\beta,B)$-excellent.
By arguing as in Case~2 from Section~\ref{firstcase} we see that $|V_{ex,B}|\leq \alpha ^{1/4} n$.

$G'$ contains a $T$-packing $\mathcal M'$ such that:
\begin{itemize}
\item[($\lambda _1$)] $\mathcal M'$ contains   $m' \leq 2 r \sqrt{\alpha} n + 2 \alpha ^{1/4} n \leq 3 \alpha ^{1/4} n$ copies of $T$;
\item[($\lambda _2$)] If $x \in V(G') \setminus A_i$ is not $(\beta  ,i)$-excellent  then $x$ is contained in a copy of $T$ in $\mathcal M'$ (for any $1 \leq  i \leq s$).
Similarly, if $x \in V(G') \setminus B$ is not $(\beta  ,B)$-excellent  then $x$ is contained in a copy of $T$ in $\mathcal M'$;
\item[($\lambda_3$)] $\mathcal M'$ covers exactly $m'$ vertices from $A_i$ (for each $1 \leq i \leq s$) and $2m'$ vertices from $B$.
Further, $\mathcal M'$ covers an even number of vertices from $B_1$ and an even number of vertices from $B_2$.
\end{itemize}
To prove that such a $T$-packing $\mathcal M'$ exists, we will use the follow three claims.

\begin{claim}\label{claima1}
Let $x \in V(G') \setminus B$ be such that $x$ is not $(\beta  ,B)$-excellent and let $W \subseteq V(G')\setminus \{x\}$ where $|W| \leq \alpha ^{1/5} n$.
Then there are two vertex-disjoint copies $T',T''$ of $T$ in $G'$ so that:
\begin{itemize}
\item[(i)] $V(T') \cup V(T'')$ contains two vertices from $A_i$ (for each $1 \leq i \leq s$) and four vertices from $B$;
\item[(ii)] $x \in V(T') \cup V(T'')$ and $V(T') \cup V(T'')$  is disjoint from $W$.
\item[(iii)] $V(T') \cup V(T'')$ contains an even number of vertices from $B_1$ and an even number of vertices from $B_2$.
\end{itemize}
\end{claim}
\proof
To prove the claim consider a vertex  $x \in V(G') \setminus B$ that is not $(\beta  ,B)$-excellent. Without loss of generality suppose that $x \in A_1$.
By arguing precisely as in Claim~\ref{claima} one can find a copy $T'$ of $T$ in $G'$ 
that covers 
two vertices in $A_1$ (including $x$), one vertex from $A_i$ (for $2 \leq i \leq s$) and one vertex from $B$ so that $T'$ is disjoint from $W$.

Without loss of generality suppose that $T'$ covers a vertex from  $B_1$. Then by applying ($\kappa _2$) and ($\iota _4$)  it is easy to see that we can greedily
construct a copy $T''$ of $T$ so that $T''$ covers three vertices from  $B_1$, no vertices from $A_1$ and one vertex from $A_j$ (for all $2 \leq j \leq s$) and so that 
$T''$ is disjoint from $V(T') \cup W$. Together $T'$ and $T''$ satisfy (i)--(iii). This completes the proof of Claim~\ref{claima1}.
\endproof

\begin{claim}\label{claimb1}
Let $x \in V(G') \setminus (A_i \cup B)$ be such that $x$ is not $(\beta , i)$-excellent (for some $1 \leq i \leq s$) and let $W \subseteq V(G')\setminus \{x\}$ where $|W| \leq \alpha ^{1/5} n$.
Then there are two vertex-disjoint copies $T',T''$ of $T$ in $G'$ so that:
\begin{itemize}
\item[(i)] $V(T') \cup V(T'')$ contains two vertices from $A_i$ (for each $1 \leq i \leq s$) and four vertices from $B$;
\item[(ii)] $x \in V(T') \cup V(T'')$ and $V(T') \cup V(T'')$  is disjoint from $W$.
\item[(iii)] $V(T') \cup V(T'')$ contains an even number of vertices from $B_1$ and an even number of vertices from $B_2$.
\end{itemize}
\end{claim}
The proof of Claim~\ref{claimb1} is essentially identical to the proof of Claim~\ref{claima1}, so we omit it.

\begin{claim}\label{claimc1}
Let $x \in B$ be such that $x$ is not $(\beta , i)$-excellent (for some $1 \leq i \leq s$) and let $W \subseteq V(G')\setminus \{x\}$ where $|W| \leq \alpha ^{1/5} n$.
Then there is a copy $T'$ of $T$ in $G'$ so that:
\begin{itemize}
\item[(i)] $V(T')$ contains one vertex from $A_j$ (for each $1 \leq j \leq s$) and two vertices from either $B_1$ or $B_2$;
\item[(ii)] $x \in V(T') $ and $V(T')$  is disjoint from $W$.
\end{itemize}
\end{claim}
It is easy to see that   ($\kappa _3$), ($\iota _3$) and ($\iota _4$)    imply that we can greedily construct a copy
$T'$ of $T$ as in Claim~\ref{claimc1}.

\medskip

Recall that $|V_{ex,B}|\leq \alpha ^{1/4} n$. 
Together with ($\iota _4$)  this implies that we can repeatedly apply Claims~\ref{claima1}--\ref{claimc1} to obtain  a $T$-packing $\mathcal M'$ in $G'$ satisfying ($\lambda _1$)--($\lambda _3$).

Remove all those vertices covered by $\mathcal M'$ from $G'$ (and from the classes $A_1, \dots, A_s, B$ and from $B_1$ and $B_2$). Call the resulting digraph $G''$.
So $n'':=|G''|\geq (1- {\alpha} ^{1/5})n$ by ($\lambda_1$) and  $|A_i|=n''/r$ for all $1\leq i \leq s$ and $|B|=2n''/r$ by ($\lambda _3$).
($\lambda _2$) implies that, given any $x \in V(G'') \setminus A_i$, $x$ is $( \beta ,i)$-excellent (for all $ 1 \leq i \leq s$)
and every vertex $y \in V(G'') \setminus B$ is $(\beta, B)$-excellent.

Moreover, ($\kappa_1$) and ($\lambda _3$) imply that $|B_1|$ and $|B_2|$ are even. 
($\kappa _1$)--($\kappa  _3$) and ($\lambda _1$) imply that
\begin{itemize}
\item $ |B_1|,|B_2| \geq n/r - 3\eta ^{1/2} n $;
\item There are at most $5\eta ^{1/2} |B|$ vertices $x \in B_i$ that are not $(3\eta ^{1/2} ,B_i)$-excellent (for $i=1,2$);
\item Given any $x \in B_i$, $d^+ _{G''} (x, B_i) \geq n/6r$ or $d^- _{G''} (x, B_i) \geq n/6r$ (for $i=1,2$).
\end{itemize}
It is easy to see that this implies that $G''[B]$ contains a perfect matching $\mathcal P$. 

\smallskip

Define an auxiliary digraph $G^*$ from $G''$ as follows. $G^*$ has vertex set $A_1 \cup \dots \cup A_s \cup B^*$ where $|B^*|=n''/r$ and each vertex $x \in B^*$ corresponds to a unique edge $e_x$
from $\mathcal P$. The edge set of $G^*$ consists of every edge $wz \in E(G'')$ such that $w \in A_i$ and $z \in A_j$ for some $i \not =j$ together with the following edges:
Suppose that $ x \in B^*$ and $y \in V(G^*) \setminus B^*$. Then
\begin{itemize}
\item $yx \in E(G^*)$ precisely if $y$ sends an edge to both vertices on $e_x$ in $G''$;
\item $xy \in E(G^*)$ precisely if $y$ receives an edge from both vertices on  $e_x$ in $G''$.
\end{itemize}
Note that $G^*$ is an $(r-1)$-partite digraph with vertex classes of size $n''/r$. Further,
$$\bar{\delta} ^+ (G^*), \bar{\delta }^- (G^*) \geq n''/r - 2\beta  n.$$
So by Theorem~\ref{rpart2}, $G^*$ contains a perfect $K_{r-1}$-packing. By construction of $G^*$ this implies that $G''$ contains a perfect $T$-packing $\mathcal M''$.
Set $\mathcal M:=\mathcal M' \cup \mathcal M'' \cup \mathcal M_1 \cup \dots \cup \mathcal M_s \cup \mathcal T$. Then $\mathcal M$ is a perfect $T$-packing in $G$, as required.
\endproof

\section{Proof of Lemma~\ref{exC3}} \label{secexC3}
 Let $0 < 1/n_0 \ll \alpha \ll \beta  \ll \gamma   \ll 1$. Suppose that $G$ is as in the statement of the lemma
and let $A_1,A_2,A_3$ denote the partition of $V(G)$ corresponding to the vertex classes of $Ex(n)$.

Given any $1 \leq i \leq 3$, any $ x \in A_i$ and any $\delta >0$, we say that $x$ is \emph{externally $\delta$-excellent} if $x$ sends out at least
$(1-\delta)|A_{i+1}|$ edges to $A_{i+1}$ in $G$ and receives at least $(1-\delta )|A_{i-1}|$ edges from $A_{i-1}$ in $G$.
(Here indices are taken mod $3$.)
Otherwise we say that
$x$ is \emph{externally $\delta$-bad}.
Since $G$ $\alpha$-contains $Ex(n)$ and $\alpha \ll \beta  \ll \gamma$, there are at most $\beta n$ vertices
in $G$ that are  externally $\gamma$-bad.

We also require analogous definitions corresponding to edges inside our vertex classes. Indeed, 
given any $1 \leq i \leq 3$, any $ x \in A_i$ and any $\delta >0$,   
we say that $x$ is \emph{internally $\delta$-excellent} if $x$ sends out at least
$(1-\delta)|A_i|$ edges in $G[A_{i}]$ and receives at least $(1-\delta)|A_i|$ edges in $G[A_{i}]$.
Otherwise we say that
$x$ is \emph{internally $\delta$-bad}.
Since $G$ $\alpha$-contains $Ex(n)$ and $\alpha \ll \beta \ll \gamma $, there are at most $\beta n$ vertices
in $G$ that are internally $\gamma $-bad. Throughout the proof we will modify the classes $A_1$, $A_2$ and $A_3$. When
referring to, for example, internally excellent vertices, we mean with respect to the current classes  $A_1$, $A_2$ and $A_3$
rather than the original partition of $V(G)$.

Since $\delta ^0 (G) \geq 2n/3-1$, given any vertex $x \in V(G)$ there is an $1\leq i_x \leq 3$ such that $x$ sends out 
at least $n/10$ edges to $A_{i_x}$ in $G$ and receives at least $n/10$ edges from $A_{i_x}$ in $G$. For each vertex $x \in V(G)$
that is internally $\gamma $-bad we move $x$ into the class $A_{i_x}$. 
Thus, we now have that:
\begin{itemize}
\item[(a)] $|A_i|= n/3 \pm 2\beta n$ for each $1 \leq i \leq 3$;
\item[(b)] $\delta ^0 (G[A_i]) \geq n/20$ for each $1 \leq i \leq 3$;
\item[(c)] All but at most $\beta n $ vertices in $G$ are internally $2\gamma $-excellent;
\item[(d)] All but at most $2\beta n$ vertices in $G$ are externally $2 \gamma $-excellent. \COMMENT{Now $2\beta n$ since moved
vertices may no longer be externally excellent}
\end{itemize}
Actually there is some slack in these conditions. Indeed, (a)--(d) hold even if we remove three vertices from $G$ (and thus
from $A_1$, $A_2$ and $A_3$).

\medskip

\noindent
{\bf Changing the parity of the class sizes.}
Our next task is to remove (the vertices of) at most one copy of $C_3$ from $G$ to ensure that $|A_1| \equiv |A_2| \equiv |A_3|$ (mod $3$).
If $|A_1 |\equiv |A_2 |\equiv |A_3|$ (mod $3$) already then we do not remove a copy of $C_3$. 
Recall that $n \equiv 0$ (mod $3$).
Therefore, without loss of generality we may
assume that $|A_1| \equiv 0$, $|A_2| \equiv 1$ and $|A_3| \equiv 2$ (mod $3$). (The other cases follow analogously.)

In this case there exists a $1\leq j \leq 3$ such that  $|A_{j}|\geq n/3+1$. 
Suppose that $|A_2|\geq n/3+1$. 
Fix some $a \in A_1$ that is externally $2\gamma$-excellent. 
(Such a vertex exists by (a) and (d).)
Since  $d^- (a)\geq 2n/3-1$ there exists
a vertex $b \in A_2$ such that $ba \in E(G)$. Further, since $a$ is externally $2\gamma$-excellent and 
$d^- _{G[A_2]}(b) \geq n/20$ by (b), there is a vertex $c \in A_2$ such that $ac, cb \in E(G)$. Thus, $a$, $b$ and $c$ together span
a copy $C'_3$ of $C_3$ in $G$. Remove $V(C'_3)$ from $G$ (and thus from $A_1$ and $A_2$). So now $|A_1|\equiv |A_2|\equiv |A_3|\equiv 2$ (mod $3$). In all other cases, we can similarly remove three vertices from $G$ that span a copy $C'_3$ of $C_3$ so that
$|A_1| \equiv |A_2| \equiv |A_3|$ (mod $3$). As outlined earlier, (a)--(d) still hold.

\medskip

\noindent{\bf Covering the externally bad vertices and balancing the class sizes.}
(b)--(d) ensure that we can greedily construct a collection $\mathcal C_1$ of at most $2 \beta n$ vertex-disjoint 
copies of $C_3$ in $G$ that together cover all those vertices in $G$ that are externally $2\gamma$-bad.
In particular, (b) and (c) ensure that we can choose each such copy of $C_3$ to lie in one of $G[A_1]$, $G[A_2]$ and $G[A_3]$.
So after removing the vertices in $\mathcal C_1$ from $G$ we still have that 
$|A_1| \equiv |A_2| \equiv |A_3|$ (mod $3$). Further, $n/3-8\beta n \leq |A_i|\leq n/3 +2 \beta n$ for each $1 \leq i \leq 3$.

These last two properties together with (b) and (c) ensure that
 we can greedily construct a collection $\mathcal C_2$ of at most $7 \beta n$ vertex-disjoint copies of $C_3$ in $G$ such that:
\begin{itemize}
\item The copies of $C_3$ in $\mathcal C_2 $ are vertex-disjoint from the copies of $C_3$ in $\mathcal C_1$;
\item Each copy of $C_3$ in $\mathcal C_2$ lies in  one of $G[A_1]$, $G[A_2]$ and $G[A_3]$.
\item By removing the vertices in $\mathcal C_1 \cup \mathcal C_2$ from $G$ we  have that 
$|A_1| = |A_2| = |A_3| \geq n/3 -  8\beta  n$.
\end{itemize}

\medskip

\noindent{\bf Covering the remaining vertices.}
Remove the vertices in $\mathcal C_1 \cup \mathcal C_2$ from $G$ (and thus from $A_1$, $A_2$ and $A_3$). The choice of $\mathcal C_1$ ensures that
every vertex now in $G$ is externally $3 \gamma$-excellent and the choice of $\mathcal C_2$ ensures that now
$|A_1| = |A_2| = |A_3| \geq n/3 -  8\beta  n$.
Let $G':=G[A_1,A_2] \cup G[A_2,A_3]\cup G[A_3,A_1]$. 
By ignoring the orientations of the edges in $G'$, we can (for example) apply Theorem~\ref{rpart} to find
a perfect $C_3$-packing  $\mathcal C_3$ in $G'$. 
(Indeed, the underlying graph of $G'$ satisfies the minimum degree condition in Theorem~\ref{rpart} since every vertex in $V(G')$
is externally $3 \gamma$-excellent.)
The union of $\mathcal C_1$, $\mathcal C_2$, $\mathcal C_3$ and $C'_3$
(if it exists) is a perfect $C_3$-packing in $G$, as desired.
\endproof

\section{Concluding remarks}
In this section we raise a number of open questions concerning perfect packings in digraphs.

\subsection{Minimum degree conditions forcing perfect tournament packings}
In~\cite{cdkm}, Czygrinow, DeBiasio,  Kierstead and  Molla proved the following minimum degree result for perfect transitive tournament packings.
\begin{thm}[Czygrinow, DeBiasio,  Kierstead and  Molla~\cite{cdkm}]\label{cdthm2}
Let $n,r \in \mathbb N$ such that $r$ divides $n$. Then every digraph $G$ on $n$
vertices with $$\delta  (G) \geq 2(1-1/r)n-1$$  contains a perfect $T_r$-packing.
\end{thm}
Note that Conjecture~\ref{conj2}  would imply Theorem~\ref{cdthm2}. At first sight one may ask whether $T_r$ can be replaced by 
 \emph{any} $T \in \mathcal T_r$ in Theorem~\ref{cdthm2}.
However, the following result of
Wang~\cite{wang} shows that one requires a significantly larger minimum degree condition in the case when $T=C_3$.
\begin{thm}[Wang~\cite{wang}]\label{wang1}
Let $n  \in \mathbb N$ such that $3$ divides $n$. If $G$ is a digraph on $n$ vertices and
$$\delta (G) \geq \frac{3n-3}{2}$$ then $G$ contains a perfect $C_3$-packing. Moreover, if $n/3$ is odd, then there is a digraph $G'$ on $n$ vertices
with $\delta (G') = \frac{3n-5}{2}$  which does not contain a perfect $C_3$-packing.
\end{thm}
Together with Zhang~\cite{zhang}, Wang also characterised the minimum degree threshold that ensures a digraph contains a perfect $C_4$-packing. (Here $C_4$ denotes the directed cycle on four vertices.)
Czygrinow,  Kierstead and  Molla~\cite{ckm} showed that the degree condition in Theorem~\ref{wang1} can be relaxed
to $\delta (G) \geq (4n-3)/3$ if we instead seek a perfect packing consisting of a fixed number of  cyclic triangles and at least one transitive triangle.

In light of Theorems~\ref{cdthm2} and~\ref{wang1} we ask the following question.
\begin{ques}\label{ques1} Let $n , r \in \mathbb N$ such that $r$ divides $n$. Let $T\in \mathcal T_r \setminus \{C_3\}$.
Does every digraph $G$ on $n$ vertices with $\delta (G) \geq 2(1-1/r)n-1$ contain a perfect $T$-packing?
\end{ques}
 Czygrinow, DeBiasio,  Kierstead and  Molla~\cite{cdkm} have answered Question~\ref{ques1} in the affirmative under the additional assumptions that $r$ is sufficiently large and $\delta (G) \geq 2(1-1/r+o(1))n$.



\subsection{Packing other directed graphs}
It is also natural to seek minimum degree conditions which ensure a digraph contains a perfect $H$-packing where  $H$ is some digraph other than a tournament. 
Let $K_r$ denote the complete digraph on $r$ vertices, and write $K_r ^-$ to denote $K_r$ minus an edge.  
(In the undirected setting we also use $K_r$ to denote the complete graph on $r$ vertices.)
The following result is a simple consequence of the
Hajnal--Szemer\'edi theorem.
\begin{prop}\label{simple} Let $ n,r \in \mathbb N$ such that $r$ divides $n$. Suppose that $G$ is a digraph on $n$ vertices
such that $$\delta  (G) \geq (2-1/r)n-1.$$ Then $G$ contains a perfect $K_r$-packing.
\end{prop}
\proof
Let $G'$ denote the \emph{graph} on $V(G)$ whose edge set consists of all those pairs $xy$ such that both $xy, yx \in E(G)$.
Since $\delta  (G) \geq (2-1/r)n-1,$ we have that $\delta (G') \geq (1-1/r)n$. Thus, Theorem~\ref{hs} implies that
$G'$ contains a perfect $K_r$-packing. By definition of $G'$ this implies that $G$ contains a perfect $K_r$-packing.
\endproof
Note that the minimum degree condition in Proposition~\ref{simple} is best-possible: Let $n \in \mathbb N$ such that $r$ divides $n$.
Suppose $A$ and $B$ are disjoint vertex sets where $|A|=n/r+1$ and $|B|=(1-1/r)n-1$. Let $G_1$ be the digraph with vertex set $A \cup B$ such that
$G_1$ contains all possible edges between $A$ and $B$, all edges in $B$ and so that $A$ induces a tournament.
 Then $\delta (G_1)= (2-1/r)n-2$ and $G$ does not contain a perfect $K_r$-packing since every
copy of $K_r$ in $G_1$ contains at most one vertex from $A$.

Proposition~\ref{simple} implies that a digraph $G$ whose order $n$ is divisible by $r$ and with
$\delta ^0 (G)\geq (1-1/2r)n-1/2$ contains a perfect $K_r$-packing. Further, in the digraph $G_1$ above, set
$n/r$ to be even and 
$G_1[A]$ to be a regular tournament.  Then $G_1$ does not contain a perfect
$K_r$-packing but $\delta ^0 (G_1)=(1-1/2r)n-1$.
Thus, together with Theorem~\ref{mainthm}, this shows that the minimum semidegree threshold that forces a perfect $K_r$-packing
is much higher than the threshold that forces a perfect $T$-packing for any tournament $T$ on $r$ vertices.
It would be interesting to establish the minimum semidegree threshold that forces a perfect $K^- _r$-packing in a digraph.
In particular, is this threshold significantly  lower than the corresponding threshold for perfect $K _r$-packings?

Let $m \in \mathbb N$ be divisible by $6$ and set $n:=2m+3$. Suppose that $G$ is a digraph on $n$ vertices with the following properties: $V(G)=V_1 \cup V_2$ where $|V_1|=m+1$ and $|V_2|=m+2$;
$G[V_1]$ and $G[V_2]$ are complete digraphs; the edges between $V_1$ and $V_2$ in $G$ form a bipartite tournament that is as regular as possible. Note that, since $|V_1|$ and $|V_2|$ are not divisible by $3$,
$G$ does not contain a perfect $K_3 ^-$-packing. Further, $\delta ^0 (G) =m/2+1+m= (3n-5)/4$.
\begin{ques}\label{quesz} Let $n  \in \mathbb N$ be divisible by $3$. 
Does every digraph $G$ on $n$ vertices with $\delta ^0 (G) \geq 3n/4$ contain a perfect $K_3 ^-$-packing?
\end{ques}

\section*{Acknowledgements}
The author is very grateful to Reshma Ramadurai for many helpful discussions, to Richard Mycroft for a discussion on~\cite{my1} and to Daniela K\"uhn, Deryk Osthus and Katherine Staden for their comments on the manuscript. The author is also grateful to the two referees for  quick and helpful reviews. In particular, the comments of one reviewer helped to significantly improve the presentation of the paper.

{\footnotesize \obeylines \parindent=0pt

Andrew Treglown
School of Mathematics
University of Birmingham
Edgbaston
Birmingham
B15 2TT
UK
\tt{a.c.treglown@bham.ac.uk}}






\end{document}